\documentclass[10pt, a4paper, twoside]{amsart}

\usepackage{amsthm}
\usepackage{amsmath}
\usepackage{amssymb}
\usepackage{amscd}

\usepackage[all]{xy}
\usepackage{tikz}
\usepackage{tikz-cd}

\usepackage[backref=page]{hyperref}
\usepackage{aliascnt}
\usepackage{cleveref}
\usepackage{adjustbox}
\usepackage{mathtools}
\usepackage{enumitem}
\usepackage{xcolor}
\usepackage[normalem]{ulem}
\usepackage{comment}

\hypersetup{
    colorlinks=true,
    linkcolor=blue,
    anchorcolor=blue,
    citecolor=blue
}

\newtheorem{thm}{\bf Theorem}[section]

\newaliascnt{prop}{thm}
\newtheorem{prop}[prop]{\bf Proposition}
\aliascntresetthe{prop}

\newaliascnt{lem}{thm}
\newtheorem{lem}[lem]{\bf Lemma}
\aliascntresetthe{lem}

\newaliascnt{cor}{thm}
\newtheorem{cor}[cor]{\bf Corollary}
\aliascntresetthe{cor}

\newaliascnt{q}{thm}

\aliascntresetthe{q}

\newaliascnt{conj}{thm}
\newtheorem{conj}[conj]{\bf Conjecture}
\aliascntresetthe{conj}

\newtheorem*{thm*}{\bf Theorem}
\newtheorem*{cor*}{\bf Corollary}

\newaliascnt{p-thm}{thm}
\newtheorem{p-thm}[p-thm]{\bf pre-Theorem}
\aliascntresetthe{p-thm}

\newaliascnt{p-prop}{thm}
\newtheorem{p-prop}[p-prop]{\bf pre-Proposition}
\aliascntresetthe{p-prop}

\newaliascnt{p-lem}{thm}
\newtheorem{p-lem}[p-lem]{\bf pre-Lemma}
\aliascntresetthe{p-lem}

\newaliascnt{p-cor}{thm}
\newtheorem{p-cor}[p-cor]{\bf pre-Corollary}
\aliascntresetthe{p-cor}

\theoremstyle{definition}

\newaliascnt{df}{thm}

\aliascntresetthe{df}

\newaliascnt{rem}{thm}
\newtheorem{rem}[rem]{\it Remark}
\aliascntresetthe{rem}

\newaliascnt{ex}{thm}

\aliascntresetthe{ex}

\newaliascnt{ass}{thm}
\newtheorem{ass}[ass]{\bf Assumption}
\aliascntresetthe{ass}

\newtheorem*{df*}{\bf Definition}
\newtheorem*{dfs*}{\bf Definitions}
\newtheorem*{ack*}{\bf Acknowledgements}

\crefname{thm}{theorem}{theorems}
\Crefname{thm}{Theorem}{Theorems}

\crefname{prop}{proposition}{propositions}
\Crefname{prop}{Proposition}{Propositions}

\crefname{lem}{lemma}{lemmas}
\Crefname{lem}{Lemma}{Lemmas}

\crefname{cor}{corollary}{corollaries}
\Crefname{cor}{Corollary}{Corollaries}

\crefname{q}{question}{questions}
\Crefname{q}{Question}{Questions}

\crefname{conj}{conjecture}{conjectures}
\Crefname{conj}{Conjecture}{Conjectures}

\crefname{p-thm}{pre-theorem}{pre-theorems}
\Crefname{p-thm}{Pre-Theorem}{Pre-Theorems}

\crefname{p-prop}{pre-proposition}{pre-propositions}
\Crefname{p-prop}{Pre-Proposition}{Pre-Propositions}

\crefname{p-lem}{pre-lemma}{pre-lemmas}
\Crefname{p-lem}{Pre-Lemma}{Pre-Lemmas}

\crefname{p-cor}{pre-corollary}{pre-corollaries}
\Crefname{p-cor}{Pre-Corollary}{Pre-Corollaries}

\crefname{df}{definition}{definitions}
\Crefname{df}{Definition}{Definitions}

\crefname{rem}{remark}{remarks}
\Crefname{rem}{Remark}{Remarks}

\crefname{ex}{example}{examples}
\Crefname{ex}{Example}{Examples}

\crefname{ass}{assumption}{assumptions}
\Crefname{ass}{Assumption}{Assumptions}

\numberwithin{equation}{section}

\def\P{\mathbb{P}}
\def\C{\mathbb{C}}
\def\F{\mathbb{F}}
\def\Q{\mathbb{Q}}

\def\Z{\mathbb{Z}}

\DeclareMathOperator{\alg}{alg}

\DeclareMathOperator{\tors}{tors}
\DeclareMathOperator{\Ker}{Ker}
\DeclareMathOperator{\Image}{Im}

\DeclareMathOperator{\nr}{nr}
\DeclareMathOperator{\cl}{cl}

\DeclareMathOperator{\Pic}{Pic}
\DeclareMathOperator{\red}{red}

\DeclareMathOperator{\Sym}{Sym}
\DeclareMathOperator{\id}{id}

\DeclareMathOperator{\Griff}{Griff}

\newcommand{\ov}{\overline}
\newcommand{\on}{\operatorname}
\newcommand{\mc}{\mathcal}

\author{Federico Scavia}

\address{CNRS\\
Institut Galil\'ee\\
Universit\'e Sorbonne Paris Nord\\
99 avenue Jean-Baptiste Cl\'ement, 93430\\
Villetaneuse, France}

\email{scavia@math.univ-paris13.fr}

\author{Fumiaki Suzuki}

\address{BICMR\\
Peking University\\
5 Yiheyuan Road, Haidian District, Beijing 100871\\
China}

\email{suzuki@bicmr.pku.edu.cn}

\title[Vanishing of degree $3$ unramified cohomology]
{Vanishing of degree $3$ unramified cohomology over finite fields}

\definecolor{MyDarkGreen}{rgb}{0.0,0.5,0.0}



\begin{document}

\begin{abstract}
Let $k$ be a finite field of characteristic $p\neq 3$, let $E/k$ be the Fermat cubic curve, and let $\ell\neq p$ be a prime. If $p\equiv1\pmod3$, assume moreover that $\ell>3$. Then $H^3_{\mathrm{nr}}(k(E^3)/k,\mathbb{Q}_\ell/\mathbb{Z}_\ell(2))=0$ and the cycle map \[CH^2(E^3)_{\mathbb{Z}_\ell}\longrightarrow H^4(E^3,\mathbb{Z}_\ell(2))\] is surjective. In particular, $H^3_{\mathrm{nr}}(\overline{k}(E^3)/\overline{k},\mathbb{Q}_\ell/\mathbb{Z}_\ell(2))=0$.

Assuming the Tate conjecture for surfaces over finite fields, we prove an analogous surjectivity result, for all but finitely many primes $\ell\neq p$, for the integral cycle maps for $1$-cycles on every smooth projective variety of dimension $d$ over a finite field of characteristic different from $2$ which admits a smooth projective lift to the ring of Witt vectors. 

We apply our results to a conjecture of Colliot-Th\'el\`ene on the local--global principle for zero-cycles over global function fields. To further illustrate these results, we exhibit examples showing that vanishing of degree-$3$ unramified cohomology over the algebraic closure of the ground field does not imply vanishing over any finite subextension, not even for Fano varieties.
\end{abstract}

\maketitle

\section{Introduction}

Degree-$3$ unramified cohomology is expected to behave very differently over finite fields than it does over fields of characteristic zero. In this paper we make this contrast explicit, both through a particular threefold and through general results (the latter conditional on the Tate conjecture).

\subsection{The triple self-product of the Fermat cubic} Our main example is the triple self-product of the Fermat cubic
\[
E=\{x^3+y^3+z^3=0\}\subset \mathbb P^2.
\]
Over algebraically closed fields of characteristic zero, this threefold
exhibits extremely large groups of codimension-$2$ cycles: for every
algebraically closed field $k$ containing $\mathbb Q$ and every prime
$\ell>5$, the group $CH^2(E_k^3)/\ell$ is infinite, by work of Schoen \cite{schoen2002complex} and of the first
author \cite{scavia2024varieties}. In particular, the group $H^3_{\mathrm{nr}}(k(E^3)/k,\mathbb Q_\ell/\mathbb Z_\ell(2))$ has infinite corank.

We show that the opposite phenomenon occurs over finite fields.

\begin{thm}\label{main-fermat-cubic-thm}
Let $k$ be a finite field of characteristic $p\neq 3$, let $E\subset\P^2_{k}$ be the Fermat cubic, and let $\ell\neq p$ be a prime. If $p\equiv 1\pmod 3$, assume moreover that $\ell>3$. Then
\[H^3_{\nr}(k(E^3)/k,\Q_\ell/\Z_\ell(2))=0\]
and the cycle map
\[
CH^2(E^3)_{\Z_\ell}
\longrightarrow
H^4(E^3,\Z_\ell(2))
\]
is surjective. In particular, \[H^3_{\nr}(\ov{k}(E^3)/\ov{k},\Q_\ell/\Z_\ell(2))=0.\]
\end{thm}

Several remarks about \Cref{main-fermat-cubic-thm} are in order.

\begin{enumerate}
    \item \Cref{main-fermat-cubic-thm} is closely related to a question of Colliot-Th\'el\`ene and Kahn \cite[Question 5.4]{colliot2013cycles}: Does there exist a smooth projective threefold $X$ over a finite field $k$ such that
    \[H^3_{\mathrm{nr}}(k(X)/k,\mathbb Q_\ell/\mathbb Z_\ell(2))\neq0?\]
    The assumption $\on{dim}(X)=3$ is important: examples of smooth projective varieties $X/k$ of dimension $d$ such that $H^3_{\on{nr}}(k(X)/k,\Q_\ell/\Z_\ell(2))\neq 0$ were produced by Pirutka \cite{pirutka2011groupe} for $d=5$, and by the authors \cite{scavia2022cohomology} for $d=4$. This question was also posed by Colliot-Th\'el\`ene at the American Institute of Mathematics in San Jose in 2019; see \cite[Problem 1.14]{AimPLRationality}. The expectation is that the answer should be positive, and products of three elliptic curves have long been natural candidates. The triple product of the Fermat cubic is the most classical such candidate; see for example the work of Schoen \cite{schoen1995computation, schoen1999image} and Kahn \cite{kahn2012classes}. \Cref{main-fermat-cubic-thm} shows that this candidate behaves contrary to expectations.

    \item Earlier results of Colliot-Th\'el\`ene and the authors for products $C\times S$, where $C$ is a curve and $S$ is a surface, establish vanishing only under additional Galois-theoretic conditions, which are not preserved under further finite extension; see \cite{colliot2020conjecture,scavia2022autour,scavia2023coniveau}.

    \item \Cref{main-fermat-cubic-thm} gives, in particular, the first examples of non-geometrically uniruled threefolds over finite fields for which degree-$3$ unramified cohomology vanishes. Earlier vanishing results in dimension $3$ concern geometrically uniruled varieties: Parimala--Suresh proved the result for conic bundles over surfaces \cite{parimala2016degree}, while Tian proved it for threefolds admitting a fibration over a curve with smooth projective geometrically rational generic fiber \cite{tian2025local}; a key geometric input in the latter is the work of Koll\'ar--Tian \cite[Theorem~7]{kollar2025stable}.

    \item In view of the aforementioned results of Schoen and the first author, \Cref{main-fermat-cubic-thm} shows that a threefold whose degree-$3$ unramified cohomology has infinite corank in characteristic zero can have vanishing unramified cohomology after reduction to a finite field. 

    \item \Cref{main-fermat-cubic-thm} has an application to the local--global principle for zero-cycles over global function fields; see \S\ref{subsec:local-global} below.

    \item Theorem~\ref{main-fermat-cubic-thm} has remarkable consequences for codimension-$2$ cycles on $E^3$ over finite fields: we explore them in detail in \S\ref{sec:7}. We first determine in \Cref{lem:transcendental-E3} the coniveau filtration on $H^3(E^3,\mathbb Z_\ell)$ over $\overline{\F}_p$ and show in \Cref{c-sc-E3} that coniveau and strong coniveau agree in degree $3$. We then use the vanishing of degree-$3$ unramified cohomology to obtain Galois descent results for codimension-$2$ cycles: under the hypotheses of \Cref{main-fermat-cubic-thm}, $CH^2$ satisfies Galois descent after tensoring with $\mathbb Z_\ell$ (\Cref{cor:CTK}), while \Cref{cor:waj} gives a description of algebraically trivial codimension-$2$ cycles over finite extensions in terms of the algebraic Abel--Jacobi map. Finally, we compute in \Cref{griffiths} the $\ell$-primary Griffiths group over $\overline{\F}_p$, completing a calculation of Schoen \cite[\S14]{schoen1995computation}: in the ordinary case it is $(\mathbb Q_\ell/\mathbb Z_\ell)^2$, whereas in the supersingular case it vanishes.
\end{enumerate}

\subsection{Strong forms of the integral Tate conjecture}
Our methods also give a general result, conditional only on the Tate conjecture for divisors on surfaces. 

\begin{thm}\label{main-conditional-liftable}
Let $k_0$ be a finite field of characteristic $p>2$, and assume that the Tate conjecture for divisors holds for every smooth projective surface over every finite extension of $k_0$. Let $X/k_0$ be a smooth projective geometrically integral variety of dimension $d$ which admits a smooth projective lifting to the ring of Witt vectors $W(k_0)$. For all but finitely many primes $\ell\neq p$, for every finite extension $k/k_0$,
the cycle map
\[CH^{d-1}(X_k)_{\Z_\ell}\longrightarrow H^{2d-2}(X_k,\Z_\ell(d-1))\]
is surjective. Moreover, if $d=3$ the abelian group $H^3_{\nr}(k(X)/k,\Q_\ell/\Z_\ell(2))$ is divisible.
\end{thm}

We make some remarks about \Cref{main-conditional-liftable}.

\begin{enumerate}
    \item Unlike \Cref{main-fermat-cubic-thm}, here we do not prove the vanishing of the group $H^3_{\nr}(k(X)/k,\Q_\ell/\Z_\ell(2))$ itself, but only the vanishing of its quotient by its maximal divisible subgroup.
    \item \Cref{main-conditional-liftable} is an arithmetic strengthening of a celebrated theorem of Schoen \cite[Theorem~0.5]{schoen1998integral} in the case of liftable varieties. Under the same assumption on the Tate conjecture for divisors on surfaces, Schoen proved the integral Tate conjecture for one-cycles over the algebraic closure of $k_0$: for all primes $\ell\neq p$, the cycle map
    \[CH^{d-1}(X_{\overline{k}_0})_{\Z_\ell}\longrightarrow H^{2d-2}(X_{\overline{k}_0},\Z_\ell(d-1))^{(1)}\]
    is surjective, where the superscript $(1)$ denotes the subgroup of classes fixed by an open subgroup of $G_{k_0}$. For all but finitely many primes $\ell\neq p$, this is implied by \Cref{main-conditional-liftable}. Conversely, Schoen's theorem together with the surjectivity of the $\ell$-adic Abel--Jacobi map over all finite extensions $k/k_0$ implies \Cref{main-conditional-liftable}.
    \item \Cref{main-conditional-theorem} is optimal for $d\geq 4$, in view of the fourfold examples $X/k_0$ in any odd characteristic constructed by the authors \cite[Theorem 1.5]{scavia2022cohomology}, where the group $H^3_{\nr}(k(X)/k,\Q_2/\Z_2(2))$ is not divisible for all finite extensions $k/k_0$.
    Taking the product of $X$ with a suitable projective space yields examples of arbitrary dimension $\geq 4$ with the same property.
    \item Under a stronger form of the Tate conjecture (\Cref{tate-conjecture}), we obtain corresponding results without the liftability assumption, for $1$-cycles as well as codimension-$2$ cycles; see \Cref{main-conditional-theorem}. 
    \item Let $\widetilde{X}$ be a smooth projective lift of $X$ over the fraction field $K$ of $W(k_0)$, and fix an embedding $K\subset \mathbb{C}$. It is possible that the integral Hodge conjecture for $1$-cycles fails for $\widetilde{X}$. For example, $X$ could be the mod $p$ reduction of the product of an Enriques surface with a very general elliptic curve; see \cite{benoist2020failure}.
    \item We also apply these results to the local--global principle for zero-cycles over global function fields; see \S\ref{subsec:local-global} below.
\end{enumerate}

\subsection{Applications to the local--global principle over function fields}
\label{subsec:local-global}

One motivation for strong forms of the integral Tate conjecture over finite fields comes from the local--global principle for zero-cycles over global function fields. We recall the $\ell$-primary form of a conjecture of Colliot-Th\'el\`ene; see \cite[Conjecture~2.2]{colliot1999conjectures} and \cite[Conjecture~7.1]{colliot2013cycles}.

Let $k$ be a finite field, let $C/k$ be a smooth projective geometrically integral curve, and set $K\coloneqq k(C)$. Let $X/k$ be a smooth, projective, geometrically integral variety of dimension $d=d'+1$, equipped with a dominant morphism $f\colon X\to C$ whose generic fiber $V/K$ is smooth and geometrically integral of dimension $d'$. For every place $v$ of $C$, let $K_v$ be the completion of $K$ at $v$ and set $V_v\coloneqq V\times_KK_v$.

\begin{conj}[Colliot-Th\'el\`ene]\label{conj-bmo}
Let $\ell$ be a prime invertible in $k$, and for every closed point $v$ of
$C$ let $z_v\in CH_0(V_v)$. Suppose that
\[
\sum_v\operatorname{inv}_v(A|_{z_v})=0\quad \text{in}\quad \Q_\ell/\Z_\ell
\]
for every $A\in\operatorname{Br}(V)\{\ell\}$. Then, for every $n>0$, there exists $z_n\in CH_0(V)$ such that, for every closed point $v$ of $C$,
\[
\operatorname{cl}(z_n)=\operatorname{cl}(z_v)
\quad\text{in}\quad
H^{2d'}(V_v,\mu_{\ell^n}^{\otimes d'}).
\]
\end{conj}

Colliot-Th\'el\`ene proved that \Cref{conj-bmo} follows from the
surjectivity of the integral cycle map
\[
CH^{d-1}(X)_{\Z_\ell}
\longrightarrow
H^{2d-2}(X,\Z_\ell(d-1));
\]
see \cite[Proposition~3.2]{colliot1999conjectures}. Thus
\Cref{main-fermat-cubic-thm} immediately gives an unconditional example
where the generic fiber is an abelian surface.

\begin{cor}\label{cor:local-global-fermat}
Let $p\neq3$ be a prime, let $E/\F_p$ be the Fermat cubic, and let $\ell\neq p$ be a prime.
If $p\equiv 1\pmod 3$, assume moreover that $\ell>3$. Then, for every finite extension $k/\F_p$, \Cref{conj-bmo} holds for $(E\times_kE)_{k(E)}/k(E)$ and the prime $\ell$.
\end{cor}

Indeed, $(E\times_kE)_{k(E)}$ is the generic fiber of the first projection $E_k^3\to E_k$ and the required cycle map is surjective by \Cref{main-fermat-cubic-thm}. 

No result of this sort is known over number fields; see for example the introduction of \cite{wills2025local}. There is some evidence for products of elliptic curves; see the work of Gazaki--Koutsianas \cite{gazaki2024weak} and, for CM self-products, Wills \cite{wills2025local}. 

Our conditional results give the same conclusion in greater generality.

\begin{cor}\label{cor:local-global-conditional}
Let $k_0$ be a finite field of characteristic $p>2$, let $C/k_0$ be a smooth projective geometrically connected curve, and let $f\colon X\to C$ be a dominant morphism whose generic fiber $V/k_0(C)$ is smooth and geometrically integral. Assume that the Tate conjecture for divisors holds for every smooth projective surface over every finite extension of $k_0$, and that $X$ admits a smooth projective lifting to $W(k_0)$.

For all but finitely many primes $\ell\neq p$, for every finite extension $k/k_0$,
\Cref{conj-bmo} holds for $V_{k(C)}/k(C)$.
\end{cor}

This follows immediately from \Cref{main-conditional-liftable} and \cite[Proposition~3.2]{colliot1999conjectures}. As we mentioned above, under the stronger form of the Tate conjecture considered in \Cref{tate-conjecture}, the liftability assumption can be removed; see \Cref{main-conditional-theorem}.

\subsection{A strengthening of Schoen's connecting-map criteria}
\label{subsec:schoen-criterion}

A key technical input in this paper which is of independent interest is a strengthening of two criteria of Schoen for the non-vanishing and surjectivity of connecting maps arising from the localization sequence on a curve. The first version appears in \cite[\S\S6--9]{schoen1995computation}, in the course of Schoen's study of complex multiplication cycles on self-fiber-products of elliptic surfaces. A more general version was developed in \cite[\S\S2--3]{schoen1999image} in order to study the image of the $\ell$-adic Abel--Jacobi map over the algebraic closure of a finite field. Our \Cref{schoen-unified} gives a strengthening and a shorter proof of both versions at the same time.

Let $M$ be a finite-dimensional $\F_\ell$-representation of the Galois group of the function field of a smooth projective curve $B$ over a finite field $k$, let $g\colon \eta\to  B$ be the inclusion of the generic point, let $\mathcal M=g_*M$, and let $\Gamma$ be the geometric monodromy group. For $x$ in a suitable open subset $U\subset B$, localization gives a connecting map
\[\delta_x\colon H^2_{\bar x}(B_{\bar k},\mathcal M)_0^{G_k} \longrightarrow H^1(k,H^1(B_{\bar k},\mathcal M)).\]
The goal is to understand whether the target is spanned by the images of these maps as $x$ varies. This is the basic mechanism used in both
\cite{schoen1995computation} and \cite{schoen1999image}.

Associated with the arithmetic monodromy cover of $M$ is an extension
\[
1\longrightarrow A\longrightarrow G\longrightarrow\Gamma\longrightarrow1.
\]
Let
\[
\mathcal K\coloneqq \Ker [
H^1(\Gamma,M^\vee(1))
\longrightarrow
\prod_{b\in B_{\bar k}\setminus U_{\bar k}}
H^1(I_b,M^\vee(1)) ].
\]
Suppose that $M$ is irreducible as an $\F_\ell[\Gamma]$-module and that
there exists $\xi\in\Gamma$ of order prime to $\ell$ such that $\dim_{\F_\ell}M^{\langle\xi\rangle}=1$. Then, under geometric assumptions which can be arranged after replacing $k$ by a finite extension (see \Cref{schoen-situation-projective}), we prove in \Cref{schoen-unified} the following formula:
\[
\bigcup_{\substack{x\in U(k)\\
\kappa_k(\operatorname{Frob}_x)\sim\xi}}
\operatorname{Im}(\delta_x)
=
(\operatorname{Inf}(\mathcal K))^\perp
\subset
H^1(k,H^1(B_{\bar k},\mathcal M)).
\]
In particular, if $\mathcal K=0$, then for every $\epsilon\in H^1(k,H^1(B_{\bar k},\mathcal M))$ there exists $x\in U(k)$ such that $\epsilon\in\operatorname{Im}(\delta_x)$; see \Cref{schoen-unified-cor}.

This should be compared first with \cite[Theorems~6.7 and~6.8]{schoen1995computation}. There Schoen gives a criterion for a certain coboundary map to be non-zero in terms of an indecomposable projective $\F_\ell[\Gamma]$-module and its radical, and then obtains an injectivity statement in the particular case when the monodromy group $\Gamma$  is $\operatorname{PSL}_2(\F_\ell)$ and $M=\operatorname{Sym}^2(\F_\ell^2)$, under an additional condition on the ramification of the monodromy cover. The proof occupies \cite[\S\S7--9]{schoen1995computation} and uses the structure of projective
$\F_\ell[\Gamma]$-modules in an essential way. Our \Cref{schoen-unified-cor} is not tied to the special form of the monodromy group and the monodromy representation: we only use the irreducibility of $M$ and the existence of a prime-to-$\ell$ element
$\xi\in \Gamma$ such that $M^{\langle \xi\rangle}$ is one-dimensional.

The comparison with \cite{schoen1999image} is even more direct. Our
Assumptions~\ref{schoen-situation-projective}(1)--
\ref{schoen-situation-projective}(8) coincide with Schoen's assumptions
\cite[(3.2.1)--(3.2.6), (3.3.1)--(3.3.2)]{schoen1999image}.
Schoen further assumes that $M$ is absolutely irreducible, whereas we
only assume that it is irreducible. Moreover, in place of
\cite[(3.3.4)--(3.3.5)]{schoen1999image},
\[
H^1(\Gamma,M)=0,
\qquad
H^1(\Gamma,M^\vee)=0,
\]
we only require the existence of an element $\xi\in \Gamma$ such that $M^{\langle \xi\rangle}$ is one-dimensional. Finally,
Schoen's injectivity assumption
\[
H^1(\Gamma,M^\vee(1))
\longrightarrow
\prod_{b\in B_{\bar k}\setminus U_{\bar k}}
H^1(I_b,M^\vee(1))
\]
amounts in our notation to the vanishing of $\mathcal K$. We do not need
to assume this: \Cref{schoen-unified} describes exactly what is obtained
when $\mathcal K$ is non-zero.

Besides \Cref{schoen-unified} being more general, its proof is also considerably shorter and in our view quite transparent. The main point is the formula
\[
\langle\delta_x(m),v\rangle_{\mathrm{par}}
=
\langle m,c_v(g_x)\rangle_M,
\]
where $g_x\in G$ is a Frobenius element and $c_v\colon G\to M^\vee$ is a cocycle representing $v$. We obtain this identity from the compatibility between localization and the Hochschild--Serre edge maps established in \Cref{lem:localization-hochschild-serre}. Once the formula is known, the rest of the proof is a finite-group theory argument which avoids the projective-module calculations of \cite[\S\S8--9]{schoen1995computation}.

\subsection{Absolutely but not potentially vanishing  \texorpdfstring{$H^3_{\nr}$}{H3nr}}

In Theorems~\ref{main-fermat-cubic-thm} and \ref{main-conditional-liftable}, the subtle point is to establish the statements over {\it every} finite extension $k/k_0$.
Namely, the vanishing of $H^3_{\nr}$ over $k$ does not, in general, follow from its vanishing after passing to a finite extension; see \cite{pirutka2011groupe}.
We point out another such phenomenon by showing that the vanishing of $H^3_{\nr}$ over sufficiently large finite extensions $k/k_0$ cannot, in general, be deduced from its vanishing over $\ov{k}_0$.

\begin{thm}\label{thm:potential-vs-absolute}
For every $d\geq 6$, there exist a finite field $k_0$ of odd characteristic and a smooth projective geometrically integral $k_0$-variety $X$ of dimension $d$ such that we have
\[H^3_{\nr}(\ov{k}_0(X)/\ov{k}_0,\Q_2/\Z_2(2))=0,\]
and yet for every finite extension $k/k_0$, we have
\[H^3_{\nr}(k(X)/k,\Q_2/\Z_2(2))\simeq\Z/2.\]
\end{thm}

\Cref{thm:potential-vs-absolute} is based on a construction of Ottem--Rennemo \cite{ottem2024fano} and a non-algebraicity criterion for cohomology classes over finite fields due to the authors \cite{scavia2022cohomology}. In particular, the varieties used to prove \Cref{thm:potential-vs-absolute} are Fano. To establish the vanishing over $\ov{k}_0$, we prove the integral Hodge conjecture for codimension-$2$ cycles on the classifying space $BGO(4)^\circ$.

\subsection{Organization of the paper}\label{subsec:organization} We now summarize the content of each section. \Cref{sec:2} develops the absolute and relative $\ell$-adic Abel--Jacobi maps used throughout. \Cref{sec:connecting-map} refines Schoen's connecting-map criterion and keeps track of the degree of the required constant-field extension. \Cref{sec:4} proves Theorems~\ref{main-conditional-liftable} and \ref{main-conditional-theorem}. \Cref{sec:5} recalls the auxiliary Schoen threefold and proves the required Abel--Jacobi surjectivity, including the case $\ell=5$. \Cref{sec:6} transfers the result to $E^3$ and proves \Cref{main-fermat-cubic-thm}. In \Cref{sec:7} we record consequences of \Cref{main-fermat-cubic-thm} for codimension-$2$ cycles on $E^3$, and in \Cref{sec:8} we construct examples separating potential from absolute vanishing, with an obstruction that remains nondivisible after every finite extension.

\subsection*{Notation}

For a field $k$, we let $\ov{k}$ be a separable closure of $k$, and we let $G_k\coloneqq \on{Gal}(\ov{k}/k)$ be the absolute Galois group of $k$. If $K/k$ is a field extension and $X$ is a $k$-scheme, we write $X_K\coloneqq X\times_k K$.

A variety over $k$ is a geometrically integral separated $k$-scheme of finite type; this agrees with the conventions made near the definition of the Abel--Jacobi map in \cite[\S 1, p.4]{schoen1995computation}, and the global conventions in \cite[p. 798]{schoen1999image}. For a $k$-variety $X$, we write $Z^i(X)$ for the free abelian group on integral closed subschemes of codimension $i$, $Z^i_{\mathrm{rat}}(X)$ for the subgroup of cycles rationally equivalent to zero, and we let $CH^i(X)\coloneqq Z^i(X)/Z^i_{\mathrm{rat}}(X)$ be the Chow group of codimension-$i$ cycles. For an abelian group $M$, subscripts such as $M_{\Z_\ell}\coloneqq M\otimes_{\Z}\Z_\ell$ and $M_{\Q_\ell}\coloneqq M\otimes_{\Z}\Q_\ell$ denote extension of scalars. We write $M_{\tors}$ for the torsion subgroup, $M\{\ell\}$ for the $\ell$-primary torsion subgroup, and $M_{\on{div}}$ for the maximal divisible subgroup of $M$. If $M$ is a discrete or $\ell$-adic $G_k$-module, we write $M^{(1)}\coloneqq\cup_{k\subset k'\subset\ov{k}} M^{G_{k'}}$, where $k'/k$ ranges over the finite extensions contained in $\ov{k}$; thus $M^{(1)}$ is the subgroup of elements fixed by an open subgroup of $G_k$.

If $X/k$ is smooth and geometrically integral and $A$ is a torsion coefficient $G_k$-module such that $A[p]=0$ if $\on{char}(k)=p>0$, we write $H^i_{\nr}(k(X)/k,A)$ for the subgroup of $H^i(k(X),A)$ consisting of classes unramified at every codimension-$1$ point of $X$. For $X$ proper, this agrees with the usual definition using divisorial discrete valuations of $k(X)$ trivial on $k$.

For a connected scheme $U$ with geometric point $\ov{u}$, we write $\pi_1(U,\ov{u})$ for its \'etale fundamental group and $\pi_1^t(U,\ov{u})$ for its tame quotient.

When $k$ is a finite field or an algebraically closed field, for a prime $\ell$ invertible in $k$, we write $H^i(X,\Z_\ell(r))$ for the $i$-th $\ell$-adic cohomology group with Tate twist $r$. When a complex variety is considered with integral coefficients, $H^i(X,\Z)$ denotes singular cohomology of its associated analytic space. 

Let $k$ be an algebraically closed field, let $X$ be a smooth projective $k$-variety, let $\ell$ be invertible in $k$, and let $c,i\geq0$. The coniveau subgroup $N^cH^i(X,\Z_\ell)\subset H^i(X,\Z_\ell)$ is the $\Z_\ell$-module generated by classes supported on closed subsets of codimension at least $c$, and the strong coniveau subgroup $\widetilde N^cH^i(X,\Z_\ell)$ is generated, as a $\Z_\ell$-module, by the images of Gysin maps $f_*\colon H^{i-2c}(Y,\Z_\ell(-c))\to H^i(X,\Z_\ell)$,  where $Y$ is smooth projective and $f\colon Y\to X$ is a morphism with $\dim Y=\dim X-c$. 

For a finite group $H$, we let $e(H)$ denote the exponent of $H$.

\section{The \texorpdfstring{$\ell$}{l}-adic Abel--Jacobi map}\label{sec:2}

\subsection{The Hochschild--Serre spectral sequence}
Let $k$ be a finite field, let $\ell$ be a prime number invertible in $k$, let $r\geq 0$ be an integer, and let $W$ be a smooth proper $k$-variety. Since $k$ has cohomological dimension $1$, for all $n\geq 1$, the $E_2$-page of the Hochschild--Serre spectral sequence
    \[E_2^{ij}\coloneqq H^i(k,H^j(W_{\ov{k}},\mu_{\ell^n}^{\otimes r}))\Longrightarrow H^{i+j}(W,\mu_{\ell^n}^{\otimes r})\]
    is concentrated in the first two columns. We obtain, for each $n\geq 1$, a short exact sequence
    \[
\begin{tikzcd}[column sep=small]
   0 \arrow[r]
   & H^1(k,H^{2r-1}(W_{\ov{k}},\mu_{\ell^n}^{\otimes r})) \arrow[r]
   & H^{2r}(W,\mu_{\ell^n}^{\otimes r}) \arrow[r]
   & H^{2r}(W_{\ov{k}},\mu_{\ell^n}^{\otimes r})^{G_k} \arrow[r]
   & 0.
\end{tikzcd}
\]
    These exact sequences are compatible in $n$. Since $k$ is finite, all the abelian groups appearing in the sequence are finite. It follows that the Mittag-Leffler condition is satisfied, and hence passing to the inverse limit in $n$ yields a short exact sequence of $\Z_\ell$-modules of finite type
    \begin{equation}\label{eq:hochschild-serre-short-exact}
    \begin{tikzcd}[column sep=small]
       0 \arrow[r] & H^1(k,H^{2r-1}(W_{\ov{k}},\Z_\ell(r))) \arrow[r] & H^{2r}(W,\Z_\ell(r)) \arrow[r] & H^{2r}(W_{\ov{k}},\Z_\ell(r))^{G_k} \arrow[r] & 0.
    \end{tikzcd}
    \end{equation}
By a weight argument \cite[p. 781]{colliot1983torsion}, $1-\on{Frob}_k$ induces an automorphism on $H^{2r-1}(W_{\ov{k}},\Q_\ell(r))$, and hence $H^1(k,H^{2r-1}(W_{\ov{k}},\Z_\ell(r)))$ is finite.

\subsection{The \texorpdfstring{$\ell$}{l}-adic Abel--Jacobi map}\label{paragraph-aj}
Let $k$ be a finite field, let $\ell$ be a prime number invertible in $k$, let $r\geq 0$ be an integer, and let $W$ be a smooth proper $k$-variety. We define
\[Z^r_{\mathrm{hom},\ell}(W)\coloneqq \on{Ker}[Z^r(W)\to H^{2r}(W_{\ov{k}},\Z_\ell(r))].\]
We also set
\[
CH^r_{\mathrm{hom},\ell}(W)
\coloneqq Z^r_{\mathrm{hom},\ell}(W)/Z^r_{\mathrm{rat}}(W).
\]
There is an \emph{$\ell$-adic Abel--Jacobi map}
\begin{equation}\label{eq:AJ_map_k}
\alpha^r_{W,\ell} : Z^r_{\mathrm{hom},\ell}(W) \longrightarrow H^1(k,H^{2r-1}(W_{\bar{k}},\Z_\ell(r)))
\end{equation}
which is defined in \cite[1.2]{schoen1999image} as follows. (Contrary to \cite[1.2]{schoen1999image}, we do not quotient $H^{2r-1}(W_{\bar{k}},\Z_\ell(r))$ by its torsion subgroup in the definition.)
Given a cycle $z\in Z^r_{\mathrm{hom},\ell}(W)$, let $|z|\subset W$ be the support of $z$, let $|z|_{\bar{k}}\subset W_{\bar{k}}$ be the inverse image of $|z|$ under the projection $W_{\bar{k}}\to W$, let
\begin{equation*}
H^{2r}_{|z|_{\ov{k}}}(W_{\bar{k}}, \mathbb{Z}_\ell(r))_0 \coloneqq \operatorname{Ker}[H^{2r}_{|z|_{\ov{k}}}(W_{\bar{k}}, \mathbb{Z}_\ell(r)) \to H^{2r}(W_{\bar{k}}, \mathbb{Z}_\ell(r))],
\end{equation*}
and let $\lfloor z \rfloor \in H^{2r}_{|z|_{\ov{k}}}(W_{\bar{k}}, \mathbb{Z}_\ell(r))_0$ be the fundamental class of $z$. By purity \cite[VI.5.1]{milne1980etale} we have $H^{2r-1}_{|z|_{\ov{k}}}(W_{\bar{k}}, \mathbb{Z}_\ell(r)) = 0$, and hence there is a short exact sequence of $G_k$-modules
\begin{equation}\label{eq:ses_AJ}
0 \longrightarrow H^{2r-1}(W_{\bar{k}}, \mathbb{Z}_\ell(r)) \longrightarrow H^{2r-1}((W - |z|)_{\bar{k}}, \mathbb{Z}_\ell(r)) \longrightarrow H^{2r}_{|z|_{\ov{k}}}(W_{\bar{k}}, \mathbb{Z}_\ell(r))_0 \longrightarrow 0.
\end{equation}
By definition, $\alpha^r_{W,\ell}(z)$ is the image of $\lfloor z \rfloor$ under the connecting map
\[
H^{2r}_{|z|_{\ov{k}}}(W_{\bar{k}}, \mathbb{Z}_\ell(r))^{G_k}_0 \longrightarrow H^1(k, H^{2r-1}(W_{\bar{k}}, \mathbb{Z}_\ell(r)))
\]
induced by (\ref{eq:ses_AJ}).

\begin{rem}\label{rmk:AJ_properties}
    We collect here some properties of the $\ell$-adic Abel--Jacobi map. All of them follow from the functoriality of the construction of Jannsen \cite[\S 9.4]{jannsenn1990mixed}; see also Schoen \cite[\S 1.2]{schoen1999image}. For restriction and corestriction in Galois cohomology, see Serre \cite[Ch.~I, \S 2.5]{serre1997galois}.

    (1) (Cycle class map.) We have a cartesian square
    \[
    \begin{tikzcd}
       Z_{\mathrm{hom},\ell}^r(W) \arrow[d,"\alpha^r_{W,\ell}"] \arrow[r,hook]  & Z^r(W) \arrow[d,"\mathrm{cl}_W^r"] \\
       H^1(k,H^{2r-1}(W_{\ov{k}},\Z_\ell(r))) \arrow[r,hook] & H^{2r}(W,\Z_\ell(r))
    \end{tikzcd}
    \]
    where $\mathrm{cl}_W^r$ is the cycle map and the bottom horizontal map comes from \eqref{eq:hochschild-serre-short-exact}. (The commutativity of the square is a special case of \cite[9.4]{jannsenn1990mixed}, and the fact that it is cartesian is immediate from the definition of $Z^r_{\mathrm{hom},\ell}$.)

    (2) (Correspondences.) Let $r$ and $s$ be non-negative integers, let $V$ be another smooth proper $k$-variety, and let $\Gamma \in CH^{\dim V+s-r}(V \times W)$ be a correspondence. We have the following commutative diagram:
        \[
        \begin{tikzcd}
            Z^r_{\mathrm{hom},\ell}(V) \arrow[r, "\alpha^r_{V,\ell}"] \arrow[d, "\Gamma_*"] & H^1(k, H^{2r-1}(V_{\bar{k}}, \Z_\ell(r))) \arrow[d, "{H^1(k,\Gamma_*})"] \\
            Z^s_{\mathrm{hom},\ell}(W) \arrow[r, "\alpha^s_{W,\ell}"] & H^1(k, H^{2s-1}(W_{\bar{k}}, \Z_\ell(s))).
        \end{tikzcd}
        \]
        In particular, the Abel--Jacobi map is compatible with pullback and pushforward along arbitrary morphisms $f\colon V\to W$.

    (3) (Rational Equivalence.) For every $r\geq 0$, the map $\alpha^r_{W,\ell}$ vanishes on cycles rationally equivalent to zero. Consequently, it factors through the Chow group, inducing a well-defined homomorphism:
        \[CH^r_{\mathrm{hom},\ell}(W) \longrightarrow H^1(k, H^{2r-1}(W_{\bar{k}}, \Z_\ell(r))). \]

    (4) (Field Extensions.) For every finite field extension $k'/k$, the following diagram commutes:
        \[
        \begin{tikzcd}
            Z^r_{\mathrm{hom},\ell}(W) \arrow[r, "\alpha^r_{W,\ell}"] \arrow[d, "(-)_{k'}"] & H^1(k, H^{2r-1}(W_{\bar{k}}, \Z_\ell(r))) \arrow[d, "\on{Res}_{G_{k'}}^{G_{k}}"] \\
            Z^r_{\mathrm{hom},\ell}(W_{k'}) \arrow[r, "\alpha^r_{W_{k'},\ell}"] & H^1(k', H^{2r-1}(W_{\bar{k}}, \Z_\ell(r))),
        \end{tikzcd}
        \]
        where we view $W_{k'}$ as a $k'$-variety.

    (5) (Norms.) For every finite separable extension $k'/k$, letting \[N_{k'/k}\colon Z^r_{\mathrm{hom},\ell}(W_{k'})\to Z^r_{\mathrm{hom},\ell}(W)\] be the norm map, the following diagram commutes:
        \[
        \begin{tikzcd}
            Z^r_{\mathrm{hom},\ell}(W_{k'}) \arrow[r, "\alpha^r_{W_{k'},\ell}"] \arrow[d, "N_{k'/k}"] & H^1(k', H^{2r-1}(W_{\bar{k}}, \Z_\ell(r))) \arrow[d, "\mathrm{Cor}_{G_{k'}}^{G_k}"] \\
            Z^r_{\mathrm{hom},\ell}(W) \arrow[r, "\alpha^r_{W,\ell}"] & H^1(k, H^{2r-1}(W_{\bar{k}}, \Z_\ell(r))).
        \end{tikzcd}
        \]
        For every finite Galois extension $k'/k$, the Abel--Jacobi map $\alpha^r_{W_{k'},\ell}$ is $\on{Gal}(k'/k)$-equivariant.

    (6) (Prime-to-$\ell$ restriction-corestriction.) Suppose given a finite extension $k'/k$ of prime-to-$\ell$ degree such that (\ref{eq:AJ_map_k}) is surjective over $k'$. Then (\ref{eq:AJ_map_k}) is surjective over $k$. Indeed, (4) and (5) give a commutative diagram
\[
\adjustbox{max width=\textwidth}{
\begin{tikzcd}
    Z^r_{\mathrm{hom},\ell}(W)
        \arrow[d,"\alpha^r_{W,\ell}"]
        \arrow[r,"(-)_{k'}"]
    & Z^r_{\mathrm{hom},\ell}(W_{k'})
        \arrow[d,"\alpha^r_{W_{k'},\ell}"]
        \arrow[r,"N_{k'/k}"]
    & Z^r_{\mathrm{hom},\ell}(W)
        \arrow[d,"\alpha^r_{W,\ell}"] \\
    H^1(k,H^{2r-1}(W_{\overline{k}},\mathbb{Z}_\ell(r)))
        \arrow[r,"\on{Res}^{G_k}_{G_{k'}}"]
    & H^1(k',H^{2r-1}(W_{\overline{k}},\mathbb{Z}_\ell(r)))
        \arrow[r,"\on{Cor}^{G_k}_{G_{k'}}"]
    & H^1(k,H^{2r-1}(W_{\overline{k}},\mathbb{Z}_\ell(r)))
\end{tikzcd}}
\]
   and the composite of each row is given by multiplication by $[k':k]$. Since the target is a finite $\ell$-primary group, multiplication by the prime-to-$\ell$ integer $[k':k]$ is an automorphism, and hence surjectivity over $k'$ implies surjectivity over $k$.

   (7) (Codimension one.) Suppose that $r=1$, and let $A\coloneqq (\Pic^0_{W/k})_{\red}$. The canonical isomorphism
   \[
   T_\ell A_{\ov{k}}\xlongrightarrow{\sim}H^1(W_{\ov{k}},\Z_\ell(1))
   \]
   identifies the restriction of $\alpha^1_{W,\ell}$ to $\Pic^0(W)(k)$ with the inverse limit of the Kummer connecting maps
   \[
   A(k)/\ell^n\longrightarrow H^1(k,A_{\ov{k}}[\ell^n]).
   \]
   These maps are surjective by Lang's theorem. Consequently $\alpha^1_{W,\ell}$ is surjective. This is also immediate from (1) and the description of the cycle map in codimension one; see \cite[Lemma~3.26]{jannsen1988continuous}.
\end{rem}

\begin{lem}\label{lem:Abel--Jacobi-corestriction}
Let $k$ be a finite field, let $\ell\neq\operatorname{char}(k)$ be a prime, let $W/k$ be a smooth proper variety, and let $r\geq0$. Let $k'/k$ be a finite extension. If the Abel--Jacobi map
\[
\alpha^r_{W_{k'},\ell}\colon CH^r_{\mathrm{hom},\ell}(W_{k'})
\longrightarrow
H^1(k',H^{2r-1}(W_{\ov{k}},\Z_\ell(r)))
\]
is surjective, then
\[
\alpha^r_{W,\ell}\colon CH^r_{\mathrm{hom},\ell}(W)
\longrightarrow
H^1(k,H^{2r-1}(W_{\ov{k}},\Z_\ell(r)))
\]
is surjective.
\end{lem}

\begin{proof}
Set $d\coloneqq [k':k]$. We have $\operatorname{Frob}_{k'}=(\on{Frob}_k)^d$, so that \[(\operatorname{Frob}_{k'}-1)H^{2r-1}(W_{\ov{k}},\Z_\ell(r))\subset (\operatorname{Frob}_k-1)H^{2r-1}(W_{\ov{k}},\Z_\ell(r)),\] and evaluation of $1$-cocycles at $\on{Frob}_k$ and $\on{Frob}_{k'}$ gives rise to the following commutative diagram:
\[
\begin{tikzcd}[column sep = small]
   CH^r_{\mathrm{hom},\ell}(W_{k'}) \arrow[d,"N_{k'/k}"]  \arrow[r,"\alpha^r_{W_{k'},\ell}"]  & H^1(k',H^{2r-1}(W_{\ov{k}},\Z_\ell(r)))\arrow[r,"\sim"] \arrow[d,"\operatorname{Cor}_{k'/k}"] & H^{2r-1}(W_{\ov{k}},\Z_\ell(r))/(\operatorname{Frob}_{k'}-1) \arrow[d,->>] \\
  CH^r_{\mathrm{hom},\ell}(W) \arrow[r,"\alpha^r_{W,\ell}"]  & H^1(k,H^{2r-1}(W_{\ov{k}},\Z_\ell(r)))\arrow[r,"\sim"] & H^{2r-1}(W_{\ov{k}},\Z_\ell(r))/(\operatorname{Frob}_k-1).
\end{tikzcd}
\]
The vertical map on the right is the obvious surjection, and the top Abel--Jacobi map is surjective by assumption. Thus, the commutativity of the diagram implies that $\alpha^r_{W,\ell}$ is surjective.
\end{proof}

\subsection{A relative \texorpdfstring{$\ell$}{l}-adic Abel--Jacobi map}\label{paragraph-relative-aj}
Let $k$ be a finite field, let $\ell$ be a prime number invertible in $k$, and let $W$ be a smooth proper $k$-variety. Suppose that $\on{dim}(W)=2m+1$ for some integer $m\geq 1$,  let $B$ be a smooth proper $k$-curve, and suppose given a flat generically smooth morphism $f\colon W \to B$. Let $\dot{B}\subset B$ be the largest open subscheme over which $f$ is smooth. For every $x\in \dot{B}(k)$, write $i_x\colon f^{-1}(x)\hookrightarrow W$ for the inclusion and let
\[Z^m(f^{-1}(x))_0\coloneqq \{z\in Z^m(f^{-1}(x))\mid i_{x*}z\in Z^{m+1}_{\mathrm{hom},\ell}(W)\},\]
and we define
\[Z_f^{m+1}(W)\coloneqq \bigoplus_{x\in \dot{B}(k)}Z^m(f^{-1}(x))_0\subset Z^{m+1}_{\mathrm{hom},\ell}(W).\]
We also define
\[Z_f^{m+1}(W_{\overline{k}})\coloneqq \varinjlim Z_f^{m+1}(W_{k'})\subset Z^{m+1}(W_{\ov{k}})\]
where the limit is over all finite extensions $k\subset k'\subset \overline{k}$.

For every $n\geq 1$, we have the Leray spectral sequence
\[E_2^{ij}\coloneqq H^i(B_{\overline{k}},R^j(f_{\overline{k}})_*\mu_{\ell^n}^{\otimes(m+1)})\Longrightarrow H^{i+j}(W_{\overline{k}},\mu_{\ell^n}^{\otimes(m+1)}).\]
Since $k$ is finite, all abelian groups appearing in the spectral sequence are finite. Therefore, taking the inverse limit as $n\to\infty$, the Mittag-Leffler condition is satisfied, and hence we obtain the spectral sequence
\begin{equation}\label{leray}E_2^{ij}\coloneqq H^i(B_{\overline{k}},R^j(f_{\overline{k}})_*\Z_\ell(m+1))\Longrightarrow H^{i+j}(W_{\overline{k}},\Z_\ell(m+1)).
\end{equation}
Since $B$ is a curve, the $E_2$ page of (\ref{leray}) is concentrated in the first three columns, and hence (\ref{leray}) induces a filtration $\{0\}\subset L^2\subset L^1\subset L^0$ of $H^{2m+1}(W_{\overline{k}},\Z_\ell(m+1))$, where
\begin{align*}
    L^0\coloneqq& H^{2m+1}(W_{\overline{k}},\Z_\ell(m+1)), \\
    L^1\coloneqq& \Ker[H^{2m+1}(W_{\overline{k}},\Z_\ell(m+1))\to H^0(B_{\overline{k}},R^{2m+1}(f_{\overline{k}})_*\Z_\ell(m+1))], \\
    L^2\coloneqq& \Ker[L^1\to H^1(B_{\overline{k}},R^{2m}(f_{\overline{k}})_*\Z_\ell(m+1))] \\
     =&\Image[H^2(B_{\overline{k}},R^{2m-1}(f_{\overline{k}})_*\Z_\ell(m+1))\to H^{2m+1}(W_{\overline{k}},\Z_\ell(m+1))].
\end{align*}
There is a \emph{relative $\ell$-adic Abel--Jacobi map}
\begin{equation}\label{eq:AJ_map_f_k}
\alpha_{f,\ell}\colon Z^{m+1}_f(W) \to H^1(k, L^1/L^2),
\end{equation}
which is defined in \cite[2.1]{schoen1999image} as follows. Given $z\in Z^{m+1}_f(W)$, suppose $z$ is supported on the fiber $V \coloneqq f^{-1}(x)$ for some $x \in \dot{B}(k)$. Let
\begin{equation}\label{eq:fund_class_fiber}
H^{2m+2}_{V_{\ov{k}}}(W_{\ov{k}}, \mathbb{Z}_\ell(m+1))_0 \coloneqq \operatorname{Ker}\left[H^{2m+2}_{V_{\ov{k}}}(W_{\ov{k}}, \mathbb{Z}_\ell(m+1)) \to H^{2m+2}(W_{\ov{k}}, \mathbb{Z}_\ell(m+1))\right]
\end{equation}
and let $\lfloor z \rfloor_V \in H^{2m+2}_{V_{\ov{k}}}(W_{\ov{k}}, \mathbb{Z}_\ell(m+1))_0$
be the fundamental class of $z$ relative to the fiber $V$, that is, the image of $\lfloor z \rfloor$ under the natural map
\begin{equation}\label{eq:support-to-fiber}
H^{2m+2}_{|z|_{\bar{k}}}(W_{\bar{k}},\Z_\ell(m+1))_0
\longrightarrow
H^{2m+2}_{V_{\bar{k}}}(W_{\bar{k}},\Z_\ell(m+1))_0.
\end{equation}

Define $B'\coloneqq B\setminus\{x\}$, let $W'\coloneqq f^{-1}(B')=W\setminus V$, and let $f'\colon W'\to B'$ be the restriction of $f$ to $W'$. The Leray spectral sequence for $(f')_{\overline{k}}$
\begin{equation}\label{leray'}E_2^{ij}\coloneqq H^i(B'_{\overline{k}},R^j((f')_{\overline{k}})_*\Z_\ell(m+1))\Longrightarrow H^{i+j}((W')_{\overline{k}},\Z_\ell(m+1))
\end{equation}
induces a filtration $\{0\}\subset (L')^2\subset (L')^1\subset (L')^0$ on $H^{2m+1}((W')_{\overline{k}},\Z_\ell(m+1))$. Since $B'$ is affine, by Artin vanishing we have $(L')^2=0$. By \cite[(2.1.5)]{schoen1999image}, we have a short exact sequence
\begin{equation}\label{eq:ses_AJ_f}
0 \longrightarrow L^1/L^2 \longrightarrow (L')^1 \longrightarrow H^{2m+2}_{V_{\bar{k}}}(W_{\ov{k}}, \mathbb{Z}_\ell(m+1))_0 \longrightarrow 0,
\end{equation}
where the map $(L')^1 \to H^{2m+2}_{V_{\ov{k}}}(W_{\ov{k}}, \mathbb{Z}_\ell(m+1))_0$ is the restriction of the surjection $H^{2m+1}((W')_{\overline{k}},\Z_\ell(m+1))\to H^{2m+2}_{V_{\ov{k}}}(W_{\ov{k}}, \mathbb{Z}_\ell(m+1))_0$ coming from the excision sequence, and where the map $L^1/L^2\to (L')^1/(L')^2=(L')^1$ is induced from the restriction map $H^{2m+1}(W_{\overline{k}},\Z_\ell(m+1))\to H^{2m+1}((W')_{\overline{k}},\Z_\ell(m+1))$, using the fact that $(L')^2=0$. By definition, $\alpha_{f,\ell}(z)$ is the image of $\lfloor z \rfloor_V$ under the connecting map
\[
\delta_{f}: H^{2m+2}_{V_{\ov{k}}}(W_{\ov{k}}, \mathbb{Z}_\ell(m+1))_0^{G_k} \longrightarrow H^1(k, L^1/L^2)
\]
induced by \eqref{eq:ses_AJ_f}.

\begin{lem}\label{relative-aj-compatibility}
    We have a commutative diagram
    \[
    \begin{tikzcd}
        Z^{m+1}_f(W) \arrow[rr,"\alpha_{f,\ell}"] \arrow[d,hook] && H^1(k,L^1/L^2) \arrow[d,"\zeta_1"] \\
        Z^{m+1}_{\mathrm{hom},\ell}(W) \arrow[r,"\alpha_{W,\ell}^{m+1}"]  & H^1(k,L^0) \arrow[r,->>,"\zeta_0"] & H^1(k,L^0/L^2).
    \end{tikzcd}
    \]
\end{lem}

\begin{proof}
    By additivity, it is enough to prove the assertion for a cycle
$z\in Z^m(f^{-1}(x))_0$ supported on a single smooth fiber
$V=f^{-1}(x)$, with $x\in \dot{B}(k)$. Let $U\coloneqq W-|z|$ and
$W'\coloneqq W-V$. Since $|z|\subset V$, we have an open immersion
$W'\hookrightarrow U$.

The restriction map
\[
H^{2m+1}(U_{\bar{k}},\Z_\ell(m+1))
\longrightarrow
H^{2m+1}((W')_{\bar{k}},\Z_\ell(m+1))
\]
together with the natural map on cohomology with supports gives a morphism
from the exact sequence \eqref{eq:ses_AJ} for $z$, namely
\[
0\longrightarrow L^0\longrightarrow
H^{2m+1}(U_{\bar{k}},\Z_\ell(m+1))
\longrightarrow
H^{2m+2}_{|z|_{\bar{k}}}(W_{\bar{k}},\Z_\ell(m+1))_0
\longrightarrow 0,
\]
to the exact sequence \eqref{eq:ses_AJ_f} defining $\alpha_{f,\ell}$, after pushing the former out by the quotient map $L^0\to L^0/L^2$ and the latter out by the inclusion $L^1/L^2\hookrightarrow L^0/L^2$. In other
words, we have a commutative diagram with exact rows:
\[
\begin{tikzcd}[column sep=small]
    0 \arrow[r] &
    L^0/L^2 \arrow[r] \arrow[d,equal] &
    H^{2m+1}(U_{\bar{k}},\Z_\ell(m+1))/L^2 \arrow[r] \arrow[d] &
    H^{2m+2}_{|z|_{\bar{k}}}(W_{\bar{k}},\Z_\ell(m+1))_0
    \arrow[r] \arrow[d] & 0 \\
    0 \arrow[r] &
    L^0/L^2 \arrow[r] &
    (L')^1\oplus_{L^1/L^2}(L^0/L^2) \arrow[r] &
    H^{2m+2}_{V_{\bar{k}}}(W_{\bar{k}},\Z_\ell(m+1))_0
    \arrow[r] & 0.
\end{tikzcd}
\]
Applying $G_k$-cohomology to this diagram, the image of
$\lfloor z\rfloor$ under the connecting homomorphism for the first row is
precisely $\zeta_0(\alpha^{m+1}_{W,\ell}(z))$ \eqref{eq:AJ_map_k}, while the image of
$\lfloor z\rfloor_V$ under the connecting homomorphism for the second row
is precisely $\zeta_1(\alpha_{f,\ell}(z))$ \eqref{eq:AJ_map_f_k}. Since $\lfloor z\rfloor$ maps to
$\lfloor z\rfloor_V$, the two classes
in $H^1(k,L^0/L^2)$ are equal. This proves the commutativity of the
diagram.
\end{proof}

\begin{lem}\label{conditions}
     Assume the following.
    \begin{itemize}
        \item[(i)] There exists $x\in B(k)$ such that $V\coloneqq f^{-1}(x)$ is smooth and $\alpha^m_{V,\ell}$ is surjective.
        \item[(ii)] The map $\alpha_{f,\ell}\colon Z^{m+1}_f(W)\to H^1(k,L^1/L^2)$ is surjective.
        \item[(iii)] The composite \[Z^{m+1}_f(W)\xlongrightarrow{\alpha_{W,\ell}^{m+1}} H^1(k,L^0) \longrightarrow H^1(k,L^0/L^1)\] is surjective.
    \end{itemize}
    Then the map $\alpha^{m+1}_{W,\ell}$ is surjective.
\end{lem}

\begin{proof}
Recall that $L^0=H^{2m+1}(W_{\ov{k}},\Z_\ell(m+1))$. For $i=0,1,2$, we let \[L^iH^1(k,L^0)\coloneqq \on{Im}[H^1(k,L^i)\to H^1(k,L^0)].\]
Let $x\in B(k)$ be as in (i). We have a commutative square
\[
\begin{tikzcd}
    H^2_{\ov{x}}(B_{\ov{k}},R^{2m-1}(f_{\ov{k}})_*\Z_\ell(m+1)) \arrow[r]\arrow[d,"\wr"]  & H^2(B_{\ov{k}},R^{2m-1}(f_{\ov{k}})_*\Z_\ell(m+1)) \arrow[d] \\
    H^{2m+1}_{V_{\ov{k}}}(W_{\ov{k}},\Z_\ell(m+1)) \arrow[r,"i_{V*}"] & H^{2m+1}(W_{\ov{k}},\Z_\ell(m+1)),
\end{tikzcd}
\]
where the left vertical map is an isomorphism by purity, and the top horizontal map is surjective by Artin vanishing applied to the affine curve $B\setminus\{x\}$. This implies that
\[i_{V*}(H^{2m+1}_{V_{\ov{k}}}(W_{\ov{k}},\Z_\ell(m+1)))=L^2.\]
It follows that the composite
\[H^{2m-1}(V_{\ov{k}},\Z_\ell(m))\xlongrightarrow{\sim}H^{2m+1}_{V_{\ov{k}}}(W_{\ov{k}},\Z_\ell(m+1))\longrightarrow H^{2m+1}(W_{\ov{k}},\Z_\ell(m+1))\]
has image equal to $L^2$. Passing to Galois cohomology, we deduce that the image of the map
    \[H^1(k,H^{2m-1}(V_{\ov{k}},\Z_\ell(m)))\longrightarrow H^1(k,H^{2m+1}(W_{\ov{k}},\Z_\ell(m+1)))\]
    is equal to $L^2H^1(k,H^{2m+1}(W_{\ov{k}},\Z_\ell(m+1)))$. Now, using the commutative square
\[
\begin{tikzcd}
    Z^m_{\mathrm{hom},\ell}(V) \arrow[d,hook,"i_{V*}"] \arrow[r,"\alpha^m_{V,\ell}"]  & H^1(k,H^{2m-1}(V_{\ov{k}},\Z_\ell(m))) \arrow[d,"i_{V*}"]  \\
    Z^{m+1}_{\mathrm{hom},\ell}(W)\arrow[r,"\alpha^{m+1}_{W,\ell}"] & H^1(k,H^{2m+1}(W_{\ov{k}},\Z_\ell(m+1))),
\end{tikzcd}
\]
we deduce that \[L^2H^1(k,L^0)\subset \alpha^{m+1}_{W,\ell}(Z^m_{\mathrm{hom},\ell}(V))\subset \alpha^{m+1}_{W,\ell}(Z^{m+1}_f(W)).\]
The commutative diagram with exact rows
\[
\begin{tikzcd}
    0 \arrow[r]  & L^1 \arrow[r] \arrow[d] & L^0 \arrow[r] \arrow[d] & L^0/L^1 \arrow[r] \arrow[d,equal] & 0 \\
    0 \arrow[r] & L^1/L^2 \arrow[r] & L^0/L^2 \arrow[r] & L^0/L^1 \arrow[r] & 0
\end{tikzcd}
\]
yields a commutative diagram with exact rows
\[
\begin{tikzcd}
    H^1(k,L^1) \arrow[r] \arrow[d] & H^1(k,L^0) \arrow[r] \arrow[d,"\zeta_0"] & H^1(k,L^0/L^1) \arrow[r] \arrow[d,equal] & 0 \\
    H^1(k,L^1/L^2) \arrow[r,"\zeta_1"] & H^1(k,L^0/L^2) \arrow[r] & H^1(k,L^0/L^1) \arrow[r] & 0.
\end{tikzcd}
\]
We deduce
\[L^1H^1(k,L^0)=\zeta_0^{-1}(\zeta_1(H^1(k,L^1/L^2)))=\zeta_0^{-1}(\zeta_1(\alpha_{f,\ell}(Z^{m+1}_f(W)))),\]
where for the second equality we have used (ii). Since $\on{Ker}(\zeta_0)=L^2H^1(k,L^0)$ is contained in $\alpha_{W,\ell}^{m+1}(Z_f^{m+1}(W))$, this implies that
\[L^1H^1(k,L^0)\subset \alpha_{W,\ell}^{m+1}(Z_f^{m+1}(W)).\]
In order to conclude, it now suffices to prove the surjectivity of the composite
\[Z^{m+1}_f(W)\xlongrightarrow{\alpha_{W,\ell}^{m+1}} H^1(k,L^0) \longrightarrow H^1(k,L^0/L^1),\]
which is guaranteed by (iii).
\end{proof}
\subsection{An \texorpdfstring{$\ell$}{l}-adic Abel--Jacobi map associated to a constructible \texorpdfstring{$\ell$}{l}-adic sheaf}
Let $k$ be a finite field, and let $B$ be a smooth projective geometrically connected curve over $k$. Let $\mc{E}=\{\mc{E}_n\}_{n\geq 1}$ be a constructible $\ell$-adic sheaf, set $\mc{E}(1) \coloneqq \mc{E} \otimes \Z_\ell(1)$, and for every $x\in B(k)$ let
\[H^2_{\ov{x}}(B_{\ov{k}},\mc{E}(1))_0\coloneqq \on{Ker}[H^2_{\ov{x}}(B_{\ov{k}},\mc{E}(1))\to H^2(B_{\ov{k}},\mc{E}(1))].\]
Define
\[
Z(\mc{E})\coloneqq \bigoplus_{x \in B(k)} H^2_{\ov{x}}(B_{\ov{k}},\mc{E}(1))_0^{G_k}.
\]
Following \cite[\S 2.3]{schoen1999image}, we construct a \emph{relative $\ell$-adic Abel--Jacobi map associated to the sheaf} $\mc{E}$
\begin{equation}\label{eq:AJ_map_E_k}
\alpha_{\mc{E}} : Z(\mc{E}) \longrightarrow H^1(k, H^1(B_{\ov{k}}, \mc{E}(1)))
\end{equation}
as follows. For every $x\in B(k)$, letting $B'\coloneqq B\setminus\{x\}$, we have the short exact sequence of $G_k$-modules
\begin{equation}\label{eq:ses_AJ_E}
0 \longrightarrow H^1(B_{\ov{k}}, \mc{E}(1)) \longrightarrow H^1(B'_{\ov{k}}, \mc{E}(1)) \longrightarrow H^2_{\ov{x}}(B_{\ov{k}}, \mc{E}(1))_0 \longrightarrow 0.
\end{equation}
Given $z\in H^2_{\ov{x}}(B_{\ov{k}},\mc{E}(1))_0^{G_k}$, by definition $\alpha_{\mc{E}}(z)$ is the image of $z$ under the connecting map
\[
\delta_{\mc{E},x}: H^2_{\ov{x}}(B_{\ov{k}}, \mc{E}(1))_0^{G_k} \longrightarrow H^1(k, H^1(B_{\ov{k}}, \mc{E}(1)))
\]
induced by \eqref{eq:ses_AJ_E}.

Similarly, we have an exact sequence
\[
0 \longrightarrow H^1(B_{\ov{k}}, \mc{E}_1(1)) \longrightarrow H^1(B'_{\ov{k}}, \mc{E}_1(1)) \longrightarrow H^2_{\ov{x}}(B_{\ov{k}}, \mc{E}_1(1))_0 \longrightarrow 0
\]
and hence a connecting map
\[\delta_{\mc{E}_1,x}: H^2_{\ov{x}}(B_{\ov{k}}, \mc{E}_1(1))_0^{G_k} \longrightarrow H^1(k, H^1(B_{\ov{k}}, \mc{E}_1(1))).\]
The following surjectivity criterion for $\delta_{\mc{E}}$ will be useful.

\begin{lem}\label{schoen-2.4.3-finer}
    Suppose that $\mc{E}_n$ is flat over $\Z/\ell^n$ for every $n\geq 1$. Moreover, let $\dot{B}\subset B$ be a dense open subscheme of $B$, write $j\colon \dot{B}\hookrightarrow B$ for the corresponding open immersion, and assume that $H^0(\dot{B}_{\ov{k}},\mc{E}_1^\vee)=0$ and that the sheaf $j^*\mc{E}_n$ is locally constant for all $n\geq 1$. Suppose given $x_1,\dots,x_t\in \dot{B}(k)$ and $z_i\in H^2_{\ov{x}_i}(B_{\ov{k}},\mc{E}(1))_0^{G_k}$.    Let $\overline z_i$ denote the reduction of $z_i$ in
    $H^2_{\ov{x}_i}(B_{\ov{k}},\mc{E}_1(1))_0^{G_k}$. If the classes
    $\delta_{\mc{E}_1,x_i}(\overline z_i)$ generate
    $H^1(k,H^1(B_{\ov{k}},\mc{E}_1(1)))$, then the classes
    $\alpha_{\mc{E}}(z_i)$ generate the $\Z_\ell$-module
    $H^1(k,H^1(B_{\ov{k}},\mc{E}(1)))$.
\end{lem}

\begin{proof}
By \cite[Lemma (2.4.1)]{schoen1999image}, the reduction map $H^1(B_{\ov{k}},\mc{E}(1))\to H^1(B_{\ov{k}},\mc{E}_1(1))$ induces an isomorphism
    \[H^1(B_{\ov{k}},\mc{E}(1))/\ell\xlongrightarrow{\sim} H^1(B_{\ov{k}},\mc{E}_1(1)).\]
    Since $k$ is a finite field, this implies the surjectivity of the right vertical map in the commutative square
    \[
    \begin{tikzcd}
        H^2_{\ov{x}}(B_{\ov{k}}, \mc{E}(1))_0^{G_k} \arrow[d] \arrow[r,"\delta_{\mc{E},x}"] &  H^1(k, H^1(B_{\ov{k}}, \mc{E}(1))) \arrow[d,->>] \\
        H^2_{\ov{x}}(B_{\ov{k}}, \mc{E}_1(1))_0^{G_k} \arrow[r,"\delta_{\mc{E}_1,x}"] &  H^1(k, H^1(B_{\ov{k}}, \mc{E}_1(1)))
    \end{tikzcd}
    \]
    for every $x\in \dot{B}(k)$.
    It follows that the $\alpha_{\mc{E}}(z_i)$ generate $H^1(k,H^1(B_{\ov{k}},\mc{E}(1)))/\ell$. Since $H^1(k,H^1(B_{\ov{k}},\mc{E}(1)))$ is a $\Z_\ell$-module of finite type, the conclusion follows from Nakayama's lemma.
\end{proof}

Now, as in \S\ref{paragraph-relative-aj}, let $W$ be a smooth proper variety of dimension $2m+1$ for some $m\geq 1$, and let $f\colon W \to B$ be a flat, generically smooth morphism. Consider the $\ell$-adic sheaf $\mc{H}\coloneqq R^{2m}f_*\Z_\ell(m+1)$, and suppose that $\mc{E}$ is a direct summand of $\mc{H}(-1)$. Fix a projection $q\colon \mc{H}(-1) \to \mc{E}$. We obtain a map \begin{equation}\label{q-sharp}q_{\#}\colon  Z^{m+1}_f(W) \longrightarrow Z(\mc{E})\end{equation} which, for every $x\in B(k)$ and every $z\in Z^m(f^{-1}(x))_0$, sends $z$ to its image under the composite
\[Z^m(f^{-1}(x))_0\to H^{2m+2}_{f^{-1}(x)_{\ov{k}}}(W_{\ov{k}},\Z_\ell(m+1))^{G_k}\xrightarrow{\sim} H^2_{\ov{x}}(B_{\ov{k}},\mc{H})^{G_k}_0\xrightarrow{q_*} H^2_{\ov{x}}(B_{\ov{k}},\mc{E}(1))_0^{G_k}.\]
It also induces a map \[q_{*}\colon H^1(k, L^1/L^2)\xlongrightarrow{\sim} H^1(k,H^1(B_{\ov{k}},\mc{H})) \longrightarrow H^1(k, H^1(B_{\ov{k}}, \mc{E}(1))).\]

\begin{lem}\label{lem:comp_f_E}
We have a commutative square
\[
\begin{tikzcd}
    Z_f^{m+1}(W) \arrow[r,"\alpha_{f,\ell}"] \arrow[d,"q_{\#}"] & H^1(k,L^1/L^2) \arrow[d,"q_*"] \\
    Z(\mc{E}) \arrow[r,"\alpha_{\mc{E}}"] & H^1(k,H^1(B_{\ov{k}},\mc{E}(1))).
\end{tikzcd}
\]
\end{lem}

\begin{proof}
The proof is analogous to that of \cite[Lemma (2.3.5)]{schoen1999image}.
By additivity, it is enough to prove the assertion for a cycle
$z\in Z^m(f^{-1}(x))_0$ supported on a single smooth fiber
$V=f^{-1}(x)$, with $x\in \dot{B}(k)$. Set $B'\coloneqq B-\{x\}$
and $W'\coloneqq W-V$.

Recall that $\mc{H}=R^{2m}f_*\Z_\ell(m+1)$ and that
$q\colon \mc{H}(-1)\to \mc{E}$ induces, by twisting, a map
\[
q(1)\colon \mc{H}\longrightarrow \mc{E}(1).
\]
By the definition of $q_{\#}$, the element $q_{\#}(z)$ is the image of the
fundamental class
\[
\lfloor z\rfloor_V\in
H^{2m+2}_{V_{\bar{k}}}(W_{\bar{k}},\Z_\ell(m+1))_0^{G_k}
\]
under the composite
\[
H^{2m+2}_{V_{\bar{k}}}(W_{\bar{k}},\Z_\ell(m+1))_0^{G_k}
\xlongrightarrow{\sim}
H^2_{\bar{x}}(B_{\bar{k}},\mc{H})_0^{G_k}
\xlongrightarrow{q(1)_*}
H^2_{\bar{x}}(B_{\bar{k}},\mc{E}(1))_0^{G_k},
\]
where the first map is the identification used in the definition
of $q_{\#}$.

The map $q(1)\colon \mc{H}\to \mc{E}(1)$ induces a morphism from the
relative Abel--Jacobi exact sequence \eqref{eq:ses_AJ_f} to the exact
sequence \eqref{eq:ses_AJ_E}. More explicitly, using the identifications
\[
L^1/L^2\simeq H^1(B_{\bar{k}},\mc{H})
\]
and
\[
H^{2m+2}_{V_{\bar{k}}}(W_{\bar{k}},\Z_\ell(m+1))_0
\simeq H^2_{\bar{x}}(B_{\bar{k}},\mc{H})_0,
\]
we obtain a commutative diagram of short exact sequences of $G_k$-modules
\[
\begin{tikzcd}[column sep=small]
0 \arrow[r] &
L^1/L^2 \arrow[r] \arrow[d,"q_*"] &
(L')^1 \arrow[r] \arrow[d] &
H^{2m+2}_{V_{\bar{k}}}(W_{\bar{k}},\Z_\ell(m+1))_0
\arrow[r] \arrow[d,"q(1)_*"] &
0 \\
0 \arrow[r] &
H^1(B_{\bar{k}},\mc{E}(1)) \arrow[r] &
H^1(B'_{\bar{k}},\mc{E}(1)) \arrow[r] &
H^2_{\bar{x}}(B_{\bar{k}},\mc{E}(1))_0
\arrow[r] &
0.
\end{tikzcd}
\]
The middle vertical map is induced by $q(1)$ and by the restriction from
$B_{\bar{k}}$ to $B'_{\bar{k}}$.

By definition, $\alpha_{f,\ell}(z)$ is the image of $\lfloor z\rfloor_V$
under the connecting homomorphism associated with the first row
\eqref{eq:ses_AJ_f}. Similarly, $\alpha_{\mc{E}}(q_{\#}(z))$ is the image
of $q_{\#}(z)$ under the connecting homomorphism associated with the second
row \eqref{eq:ses_AJ_E}. Since the diagram above is commutative, the
functoriality of connecting homomorphisms in Galois cohomology gives $q_*(\alpha_{f,\ell}(z))=\alpha_{\mc{E}}(q_{\#}(z))$, as desired.
\end{proof}

\subsection{A localization compatibility for Abel--Jacobi constructions}

The Abel--Jacobi maps considered above are defined using connecting homomorphisms arising from localization sequences. We conclude this section by recording a compatibility between such a localization coboundary map and the Hochschild--Serre edge maps for ordinary cohomology and cohomology with support. This compatibility will be used in the proof of \Cref{schoen-unified}.

\begin{lem}\label{lem:localization-hochschild-serre}
Let $k$ be a finite field, let $\ell\neq\operatorname{char}(k)$ be a prime, let $B$ be a smooth projective geometrically connected curve over $k$, and let $\mathcal F$ be a constructible $\F_\ell$-sheaf on $B$. Let $x\in B$ be a closed point such that $\mathcal F$ is locally constant in a neighborhood of $x$, and set $U_x\coloneqq B\setminus\{x\}$. Assume that $H^2(B_{\ov{k}},\mathcal F)=0$. 

\begin{enumerate}
\item The geometric localization sequence induces a short exact sequence
\begin{equation}\label{eq:localization-H1-H2}
0\longrightarrow H^1(B_{\ov{k}},\mathcal F)\longrightarrow H^1((U_x)_{\ov{k}},\mathcal F) \longrightarrow H^2_{\ov{x}}(B_{\ov{k}},\mathcal F) \longrightarrow 0.
\end{equation}
Let \[\delta_x\colon H^2_{\ov{x}}(B_{\ov{k}},\mathcal F)^{G_k} \longrightarrow H^1(k,H^1(B_{\ov{k}},\mathcal F))\] be the connecting homomorphism in Galois cohomology associated with \eqref{eq:localization-H1-H2}.

\item The Hochschild--Serre spectral sequences for ordinary cohomology and cohomology with support of $\mathcal F$ give canonical isomorphisms
\[
\epsilon_x\colon
H^2_x(B,\mathcal F)
\xlongrightarrow{\sim}
H^2_{\ov{x}}(B_{\ov{k}},\mathcal F)^{G_k}
\]
and
\[
\epsilon_B\colon
H^2(B,\mathcal F)
\xlongrightarrow{\sim}
H^1(k,H^1(B_{\ov{k}},\mathcal F)).
\]

\item With respect to these isomorphisms, the following diagram commutes:
\[
\begin{tikzcd}
H^2_x(B,\mathcal F)
    \arrow[r,"i_{x*}"]
    \arrow[d,"\epsilon_x"']
&
H^2(B,\mathcal F)
    \arrow[d,"\epsilon_B"]
\\
H^2_{\ov{x}}(B_{\ov{k}},\mathcal F)^{G_k}
    \arrow[r,"\delta_x"]
&
H^1(k,H^1(B_{\ov{k}},\mathcal F)).
\end{tikzcd}
\]
\end{enumerate}
\end{lem}

\begin{proof}
(1) Since $\mathcal F$ is locally constant in a neighborhood of $x$ and
$B_{\ov{k}}$ is smooth of dimension one, purity \cite[VI.5.1]{milne1980etale} gives $H^q_{\ov{x}}(B_{\ov{k}},\mathcal F)=0$ for all $q\neq 2$. The localization long exact sequence \cite[III, Proposition~1.25]{milne1980etale}, together with the assumption $H^2(B_{\ov{k}},\mathcal F)=0$, therefore gives \eqref{eq:localization-H1-H2}.

(2) We first construct $\epsilon_B$. Consider the Hochschild--Serre spectral sequence for the ordinary cohomology of $\mathcal F$. Since $k$ is finite, $\operatorname{cd}_\ell(k)=1$. Thus, in degree $2$, the only
possibly non-zero terms are $H^2(B_{\ov{k}},\mathcal F)^{G_k}$, which vanishes by assumption, and $H^1(k,H^1(B_{\ov{k}},\mathcal F))$. Hence the edge map gives an isomorphism
\[
H^1(k,H^1(B_{\ov{k}},\mathcal F))
\xlongrightarrow{\sim}
H^2(B,\mathcal F),
\]
and we denote its inverse by $\epsilon_B$.

We now construct $\epsilon_x$. Consider the Hochschild--Serre spectral sequence for the cohomology with support of $\mathcal F$. By purity, $H^0_{\ov{x}}(B_{\ov{k}},\mathcal F)=H^1_{\ov{x}}(B_{\ov{k}},\mathcal F)=0$. Thus, in total degree $2$, the only non-zero term is $H^2_{\ov{x}}(B_{\ov{k}},\mathcal F)^{G_k}$. It follows that restriction to $\ov{k}$ gives the canonical isomorphism
\[
\epsilon_x\colon
H^2_x(B,\mathcal F)
\xlongrightarrow{\sim}
H^2_{\ov{x}}(B_{\ov{k}},\mathcal F)^{G_k}.
\]

(3) Consider the $G_k$-equivariant localization triangle
\[
R\Gamma_{\ov{x}}(B_{\ov{k}},\mathcal F)
\longrightarrow
R\Gamma(B_{\ov{k}},\mathcal F)
\longrightarrow
R\Gamma((U_x)_{\ov{k}},\mathcal F)
\xlongrightarrow{+1}.
\]
Applying the derived functor of $G_k$-invariants gives the localization
triangle over $k$. The morphism $R\Gamma_{\ov{x}}(B_{\ov{k}},\mathcal F)
\to R\Gamma(B_{\ov{k}},\mathcal F)$ therefore induces the map
\[
i_{x*}\colon
H^2_x(B,\mathcal F)
\longrightarrow
H^2(B,\mathcal F).
\]
By purity and the assumption $H^2(B_{\ov{k}},\mathcal F)=0$, the
cohomology exact sequence of the localization triangle yields a short
exact sequence
\[
0\longrightarrow
H^1(B_{\ov{k}},\mathcal F)
\longrightarrow
H^1((U_x)_{\ov{k}},\mathcal F)
\longrightarrow
H^2_{\ov{x}}(B_{\ov{k}},\mathcal F)
\longrightarrow0.
\]
Equivalently, this extension determines a morphism $H^2_{\ov{x}}(B_{\ov{k}},\mathcal F)
\to H^1(B_{\ov{k}},\mathcal F)[1]$ in the derived category. After applying $R\Gamma(k,-)$, the induced map in degree $2$ is, under the identifications of (2), of the form
\[
H^2_{\ov{x}}(B_{\ov{k}},\mathcal F)^{G_k}
\longrightarrow
H^1(k,H^1(B_{\ov{k}},\mathcal F)).
\]
By the standard description of the connecting homomorphism associated
with the short exact sequence above, this map sends $q\in H^2_{\ov{x}}(B_{\ov{k}},\mathcal F)^{G_k}$
to the cohomology class of the $1$-cocycle $\sigma\mapsto \sigma w-w$, where $w\in H^1((U_x)_{\ov{k}},\mathcal F)$ is any lift of $q$. Hence this map is precisely the connecting homomorphism
\[
\delta_x\colon
H^2_{\ov{x}}(B_{\ov{k}},\mathcal F)^{G_k}
\longrightarrow
H^1(k,H^1(B_{\ov{k}},\mathcal F))
\]
associated with \eqref{eq:localization-H1-H2}. It follows that $\epsilon_B\circ i_{x*}=\delta_x\circ\epsilon_x$.
\end{proof}

\section{A strengthening of Schoen's non-vanishing criteria}\label{sec:connecting-map}

Let $k_0$ be a finite field of characteristic $p>2$, let $\ell\neq p$ be a prime, let $B$ be a smooth projective geometrically connected curve over $k_0$, let $\dot{B}\subset B$ be a dense affine open subscheme, and let $g\colon \eta\to  B$ denote the inclusion of the generic point. Let $M$ be a finite-dimensional $\F_\ell$-vector space, and let $\kappa\colon G_{k_0(B)}\to \on{GL}(M)$ be a continuous homomorphism.

To each finite extension $k/k_0$, we associate the following data. Consider the short exact sequence
\begin{equation}\label{eq:galois-b}1\longrightarrow G_{\ov{k}(B)}\longrightarrow G_{k(B)}\longrightarrow G_k \longrightarrow 1.\end{equation}
Let $\kappa_k\colon G_{k(B)}\to \on{GL}(M)$ and $\ov{\kappa}\colon G_{\ov{k}(B)} \to \on{GL}(M)$ be the restrictions of $\kappa$ to $G_{k(B)}$ and $G_{\ov{k}(B)}$, respectively. Let $\Gamma_{\mathrm{arith},k}\coloneqq \mathrm{Im}(\kappa_k)$ and $\Gamma_{\mathrm{geom}}\coloneqq \mathrm{Im}(\ov{\kappa})$ be the arithmetic monodromy and geometric monodromy groups of $M$, respectively. Since $G_{\ov{k}(B)}$ is normal in $G_{k(B)}$, the group $\Gamma_{\mathrm{geom}}$ is normal in $\Gamma_{\mathrm{arith},k}$. Moreover, letting $k_c$ be the algebraic closure of $k$ in $\ov{k(B)}^{\on{Ker}\kappa_k}$, sequence \eqref{eq:galois-b} fits into a commutative diagram with exact rows
\begin{equation}\label{eq:monodromy-diagram}
\begin{tikzcd}
    1 \arrow[r] & G_{\ov{k}(B)} \arrow[d,->>,"\ov{\kappa}"] \arrow[r] & G_{k(B)} \arrow[r] \arrow[d,->>,"\kappa_k"] & G_k \arrow[r]  \arrow[d,->>] & 1 \\
    1 \arrow[r] & \Gamma_{\mathrm{geom}} \arrow[r] & \Gamma_{\mathrm{arith},k} \arrow[r] & \operatorname{Gal}(k_c/k) \arrow[r] & 1 
\end{tikzcd}
\end{equation}

We let $\rho\colon C\to B_k$ be a surjective morphism from a smooth projective connected curve over $k$ such that $k(C)=\ov{k(B)}^{\on{Ker}\kappa_k}$ as subfields of $\overline{k(B)}$, and we let $\Sigma\subset B_k$ be the branch locus of $\rho$. The curve $C$ and the morphism $\rho$ are uniquely determined up to isomorphism. The morphism $\rho\colon C\to B_k$ is a Galois cover with Galois group $\Gamma_{\mathrm{arith},k}$, and each of the $[k_c:k]$ connected components of $\rho_{\ov{k}}\colon C_{\ov{k}}\to B_{\ov{k}}$ is a Galois cover with Galois group $\Gamma_{\mathrm{geom}}$. In particular, $k$ is algebraically closed in $k(C)$ if and only if $\Gamma_{\mathrm{arith},k}=\Gamma_{\mathrm{geom}}$. When this is the case, we will write $\Gamma$ for $\Gamma_{\mathrm{arith},k}=\Gamma_{\mathrm{geom}}$; see also \Cref{schoen-situation-projective}(1) below.

View $M$ as an \'etale sheaf over $\eta$, and consider the \'etale sheaf $\mc{M}\coloneqq g_*M$ over $B$. For every $x\in B_k(k)$ we have an exact sequence of $G_k$-modules
\[0\longrightarrow H^1(B_{\ov{k}},\mc{M})\longrightarrow H^1((B_k\setminus \{x\})_{\ov{k}},\mc{M})\longrightarrow H^2_{\ov{x}}(B_{\ov{k}},\mc{M})_0\longrightarrow 0,\]
where \[H^2_{\ov{x}}(B_{\ov{k}},\mc{M})_0\coloneqq \on{Ker}[H^2_{\ov{x}}(B_{\ov{k}},\mc{M})\to H^2(B_{\ov{k}},\mc{M})],\]
and hence a coboundary map
\begin{equation}\label{coboundary-situation}\delta_x\colon H^2_{\ov{x}}(B_{\ov{k}},\mc{M})_0^{G_k}\longrightarrow H^1(k,H^1(B_{\ov{k}},\mc{M})).
\end{equation}

For every $x_0\in B_k(k)$ and $c_0\in \rho^{-1}(x_0)(k)$, if they exist, we may consider the cartesian square
\begin{equation}\label{c-breve}
    \begin{tikzcd}
        \breve{C} \arrow[r] \arrow[d] &  \on{Pic}^0(C) \arrow[d,"\times \ell"]  \\
        C \arrow[r,"i_0"] & \on{Pic}^0(C),
    \end{tikzcd}
\end{equation}
where the map $i_0\colon C\to\on{Pic}^0(C)$ is given by $c\mapsto\mc{O}_C(c-c_0)$. Under Assumptions
\ref{schoen-situation-projective}(1,4,5) below, the lemma of \cite[p.~806]{schoen1999image} shows that $k(\breve C)/k(B)$ is Galois.

Whenever $C$ is geometrically connected, for every geometric point $b\in B_{\ov{k}}$ let $I_b\subset\Gamma$ denote the image of a geometric inertia subgroup in $\Gamma$. The subgroup $I_b\subset\Gamma$ is well-defined up to conjugacy, and it is trivial if $b\in(\dot B_k\setminus\Sigma)_{\ov{k}}$.

\begin{ass}\label{schoen-situation-projective}
Consider the following assumptions on $k$.
\begin{enumerate}
\item[(1)] The field $k$ is algebraically closed in $k(C)=\ov{k(B)}^{\on{Ker}\kappa_k}$, that is, $C$ is geometrically connected over $k$. Thus $\Gamma_{\mathrm{geom}}=\Gamma_{\mathrm{arith},k}$, and we let
\[
\Gamma\coloneqq\Gamma_{\mathrm{geom}}=\Gamma_{\mathrm{arith},k}\subset\operatorname{GL}(M).
\]
\item[(2)] The field $k$ contains a primitive $\ell$-th root of unity.
\item[(3)] There exists $x_0\in B_k(k)$ such that every closed point of $\rho^{-1}(x_0)\subset C$ has residue field $k$. We fix one such point $x_0$ and a point $c_0\in\rho^{-1}(x_0)(k)$.
\item[(4)] For every $b_0,b_0'\in\rho^{-1}(x_0)(k)$, the class $b_0-b_0'\in\operatorname{Pic}^0(C)(k)$ is divisible by $\ell$. Equivalently, for every $b_0\in\rho^{-1}(x_0)(k)$, the class $b_0-c_0$ is divisible by $\ell$ in $\operatorname{Pic}^0(C)(k)$.
\item[(5)] The $G_k$-action on $H^1(C_{\ov{k}},\mu_\ell)$ is trivial.
\end{enumerate}
Fix a dense open subscheme $U\subset \dot B_k\setminus\Sigma$. We further assume:
\begin{enumerate}
\item[(6)] For every conjugacy class $F\subset\operatorname{Gal}(k(\breve C)/k(B))$, there exists $x\in U(k)$ such that $F=\operatorname{Frob}_x$.
\end{enumerate}
Let $M^\vee\coloneqq\operatorname{Hom}_{\F_\ell}(M,\F_\ell)$. We also assume:
\begin{enumerate}
\item[(7)] The representation $\kappa_k$ is tamely ramified, that is, it factors through $\pi_1^t(\dot B_k)\to\operatorname{GL}(M)$.
\item[(8)] $(M^\vee)^\Gamma=0$.
\item[(9)] $M$ is irreducible as an $\F_\ell[\Gamma]$-module.
\item[(10)] There exists $\xi\in\Gamma$ of order prime to $\ell$ such that $M^{\langle\xi\rangle}\simeq\F_\ell$.
\end{enumerate}
\end{ass}

\begin{rem}
Assumptions~\ref{schoen-situation-projective}(1)--\ref{schoen-situation-projective}(8) are Schoen's assumptions \cite[(3.2.1)--(3.2.6), (3.3.1)--(3.3.2)]{schoen1999image}. Assumption~\ref{schoen-situation-projective}(9) is weaker than \cite[(3.3.3)]{schoen1999image}, where it is assumed that $M$ is absolutely irreducible, that is, that $M\otimes_{\F_\ell}\ov{\F}_\ell$ is an irreducible $\ov{\F}_\ell[\Gamma]$-module.
In place of our Assumption~\ref{schoen-situation-projective}(10), Schoen assumes the stronger conditions \cite[(3.3.4)-(3.3.5)]{schoen1999image}, that is, $H^1(\Gamma,M)=0, H^1(\Gamma,M^\vee)=0$.
\end{rem}

Suppose that Assumptions~\ref{schoen-situation-projective}(1)--\ref{schoen-situation-projective}(9) hold. Then $\breve C$ is geometrically connected and $k(\breve C)/k(B)$ is Galois; see \cite[Lemma p.~806]{schoen1999image}. Let
\[
G\coloneqq\operatorname{Gal}(k(\breve C)/k(B)),\qquad
A\coloneqq\operatorname{Gal}(k(\breve C)/k(C)).
\]
Then $A$ is an elementary abelian $\ell$-group and there is an exact
sequence
\begin{equation}\label{eq:a-g-gamma}
1\longrightarrow A\longrightarrow G\longrightarrow\Gamma
\longrightarrow1.
\end{equation}

For every finite-dimensional $\F_\ell[\Gamma]$-module $N$, let
$\mc N\coloneqq g_*N$. Define
\[
V\coloneqq C\times_BU.
\]
The pullback of $\mathcal N|_{U_{\ov{k}}}$ to $V_{\ov{k}}$ is the constant sheaf $N$.
For every $b\in B_{\ov{k}}\setminus U_{\ov{k}}$, let
\[
U_b^{\mathrm{sh}}
\coloneqq
\operatorname{Spec}(\mathcal O^{\mathrm{sh}}_{B_{\ov{k}},b})
\times_{B_{\ov{k}}}U_{\ov{k}}.
\]
We define
\begin{equation}\label{eq:def-KN-etale}
\mathcal K_N
\coloneqq
\Ker [
H^1(U_{\ov{k}},\mathcal N)
\longrightarrow
H^1(V_{\ov{k}},N)\oplus
\prod_{b\in B_{\ov{k}}\setminus U_{\ov{k}}}
H^1(U_b^{\mathrm{sh}},\mathcal N)],
\end{equation}
where the two components are pullback maps.

The Leray spectral sequence for
$j_U\colon U_{\ov{k}}\hookrightarrow B_{\ov{k}}$ gives an exact
sequence
\begin{equation}\label{eq:parabolic-etale}
0\longrightarrow H^1(B_{\ov{k}},\mc N)
\longrightarrow H^1(U_{\ov{k}},\mathcal N)
\longrightarrow
\prod_{b\in B_{\ov{k}}\setminus U_{\ov{k}}}
H^1(U_b^{\mathrm{sh}},\mathcal N).
\end{equation}
Thus every element of $\mathcal K_N$ has a unique extension to
$H^1(B_{\ov{k}},\mc N)$. We denote the resulting injective map by
\[
\operatorname{Inf}\colon
\mathcal K_N\longrightarrow H^1(B_{\ov{k}},\mc N).
\]
The notation is justified by the following description.

\begin{lem}\label{lem:KN-group-description}
\begin{enumerate}
    \item For every finite-dimensional $\F_\ell[\Gamma]$-module $N$, there is a unique isomorphism
\[
\mathcal K_N \xlongrightarrow{\sim}\on{Ker}[H^1(\Gamma,N)\longrightarrow \prod_{b\in B_{\ov{k}}\setminus U_{\ov{k}}}H^1(I_b,N)]
\]
making the following diagram commute:
\[
\begin{tikzcd}
\mathcal K_N
\arrow[r,hook,"\operatorname{Inf}"]
\arrow[d,"\wr"]
&
H^1(B_{\ov{k}},\mc N)
\arrow[d,hook,"\operatorname{Res}"]
\\
\displaystyle
\Ker [
H^1(\Gamma,N)\to
\prod_{b\in B_{\ov{k}}\setminus U_{\ov{k}}}H^1(I_b,N)]
\arrow[r,hook,"\operatorname{Inf}"']
&
H^1(U_{\ov{k}},\mathcal N)=H^1(\pi_1(U_{\ov{k}}),N).
\end{tikzcd}
\]
\item There is a unique injective map $\iota_N\colon H^1(B_{\ov{k}},\mc N)\hookrightarrow H^1(G,N)$ making the following square commute:
\[
\begin{tikzcd}
    H^1(B_{\ov{k}},\mc N) \arrow[d,hook,"\iota_N"] \arrow[r,hook,"\on{Res}"]  & H^1(U_{\ov{k}},\mc{N}) \arrow[d,equal] \\
    H^1(G,N) \arrow[r,hook,"\on{Inf}"] & H^1(\pi_1(U_{\ov{k}}),N).
\end{tikzcd}
\]
\item Letting
\[
r_N\coloneqq
\operatorname{Res}_A\circ\iota_N\colon
H^1(B_{\ov{k}},\mc N)\longrightarrow
\operatorname{Hom}_\Gamma(A,N),
\]
we have an exact sequence
\begin{equation}\label{eq:KN-exact}
0\longrightarrow
\mathcal K_N
\xlongrightarrow{\on{Inf}}
H^1(B_{\ov{k}},\mc N)
\xlongrightarrow{r_N}
\operatorname{Hom}_\Gamma(A,N).
\end{equation}
\end{enumerate}
\end{lem}

\begin{proof}
(1) The Cartan--Leray spectral sequence for the Galois cover
$V_{\ov{k}}\to U_{\ov{k}}$ gives
\[
0\longrightarrow H^1(\Gamma,N)
\xlongrightarrow{\operatorname{Inf}}
H^1(U_{\ov{k}},\mathcal N)
\longrightarrow H^1(V_{\ov{k}},N)^\Gamma.
\]
Hence $H^1(\Gamma,N)$ is identified with the kernel of the first
component in \eqref{eq:def-KN-etale}. Under this identification,
restriction to $U_b^{\mathrm{sh}}$ is induced by restriction from
$\Gamma$ to $I_b$. Indeed, if $\mathcal I_b$ denotes the geometric
inertia group at $b$, inflation--restriction for
$\mathcal I_b\twoheadrightarrow I_b$ shows that
\[
H^1(I_b,N)\longrightarrow H^1(\mathcal I_b,N)
\simeq H^1(U_b^{\mathrm{sh}},\mathcal N)
\]
is injective. 

(2) Set
\[
\breve V\coloneqq
\breve C\times_BU.
\]
The cover $\breve V_{\ov{k}}\to U_{\ov{k}}$ is Galois with group $G$, and
Cartan--Leray gives
\[
0\longrightarrow H^1(G,N)
\xlongrightarrow{\operatorname{Inf}}
H^1(U_{\ov{k}},\mathcal N)
\longrightarrow H^1(\breve V_{\ov{k}},N)^G.
\]
We claim that the image of
$H^1(B_{\ov{k}},\mc N)$ in $H^1(U_{\ov{k}},\mathcal N)$ is contained
in the image of $H^1(G,N)$. Indeed, let
$v\in H^1(B_{\ov{k}},\mc N)$. Its pullback to $V_{\ov{k}}$ extends, by
\eqref{eq:parabolic-etale} applied to
$V_{\ov{k}}\subset C_{\ov{k}}$, to a class in $H^1(C_{\ov{k}},N)$. 

We have \[H^1(C_{\ov{k}},N)=H^1(\pi_1(C_{\ov{k}}),N)=\on{Hom}_{\on{cts}}(\pi_1(C_{\ov{k}}),N),\] where the second equality follows from the fact that $\pi_1(C_{\ov{k}})$ acts trivially on $N$. Since $N$
is an $\F_\ell$-vector space, every continuous homomorphism $\pi_1(C_{\ov{k}})\to N$ factors through the maximal elementary abelian $\ell$-quotient of
$\pi_1(C_{\ov{k}})$, which is the Galois group $A$ of
$\breve C_{\ov{k}}\to C_{\ov{k}}$. Its pullback to $\breve C_{\ov{k}}$
therefore vanishes, and hence so does the pullback of $v$ to
$\breve V_{\ov{k}}$. This proves the claim and defines $\iota_N$ uniquely by
\[
\operatorname{Inf}(\iota_N(v))=v|_{U_{\ov{k}}}.
\]
It is injective by \eqref{eq:parabolic-etale}.

(3) The inflation--restriction sequence for \eqref{eq:a-g-gamma} gives
\[
0\longrightarrow H^1(\Gamma,N)
\xlongrightarrow{\operatorname{Inf}}
H^1(G,N)
\xlongrightarrow{\operatorname{Res}_A}
\operatorname{Hom}_\Gamma(A,N).
\]
Together with the preceding description of $\mathcal K_N$, this shows
that $\Ker(r_N)=\operatorname{Inf}(\mathcal K_N)$, proving
\eqref{eq:KN-exact}.
\end{proof}

Fix a trivialization $\F_\ell(1)\simeq\F_\ell$, and set $\mc M^\dagger\coloneqq g_*(M^\vee(1))$. We write \[\mathcal K\coloneqq\mathcal K_{M^\vee(1)},\qquad 
r\coloneqq r_{M^\vee(1)}\colon
H^1(B_{\ov{k}},\mc M^\dagger)
\longrightarrow
\operatorname{Hom}_\Gamma(A,M^\vee).
\]

\begin{thm}\label{schoen-unified}
Suppose that Assumptions~\ref{schoen-situation-projective}(1)--\ref{schoen-situation-projective}(10) are satisfied, and let $\xi\in\Gamma$ be as in Assumption~\ref{schoen-situation-projective}(10).

\begin{enumerate}
    \item The $G_k$-action on $H^1(B_{\ov{k}},\mc M)$ and
$H^1(B_{\ov{k}},\mc M^\dagger)$ is trivial.
\item Cup product followed by evaluation at Frobenius induces a perfect pairing
\[\langle -,-\rangle_{\mathrm{par}}\colon H^1(k,H^1(B_{\ov{k}},\mc M)) \times H^1(B_{\ov{k}},\mc M^\dagger)\longrightarrow H^1(k,\F_\ell)\xlongrightarrow{\operatorname{ev}_{\operatorname{Frob}_k}}\F_\ell.\]
\item For every $x\in U(k)$ such that
$\kappa_k(\operatorname{Frob}_x)$ is conjugate to $\xi$, the $\F_\ell$-vector space $H^2_{\ov{x}}(B_{\ov{k}},\mc M)_0^{G_k}$ is one-dimensional.
\item Consider the evaluation pairing 
\[\langle -,- \rangle_M\colon M\times M^\vee\longrightarrow\F_\ell,
 \qquad \langle m,\lambda\rangle_M=\lambda(m).
\]
Let $x\in U(k)$, let $\widetilde x$ be a geometric point of $\breve C$ above $x$, let $g_x\in G$ be a Frobenius element for $x$, and let $\gamma_x\in\Gamma$ be the image of $g_x$ under the surjection $G\to \Gamma$. Let $v\in H^1(B_{\ov{k}},\mc M^\dagger)$, let $c_v\colon G\to M^\vee$ be a cocycle representing $\iota_{M^\vee(1)}(v)$, where we use the fixed trivialization $\F_\ell(1)\simeq\F_\ell$. Then, for every $m\in M^{\langle\gamma_x\rangle}$, we have
\begin{equation}\label{eq:schoen-evaluation}
 \langle\delta_x(m),v\rangle_{\mathrm{par}}
 =
 \langle m,c_v(g_x)\rangle_M.
\end{equation}
\item With respect to the pairing $\langle -,-\rangle_{\mathrm{par}}$, we have
\begin{equation}\label{eq:union-images-delta}
\bigcup_{\substack{x\in U(k)\\
\kappa_k(\operatorname{Frob}_x)\sim\xi}}
\operatorname{Im}(\delta_x)
=
(\on{Inf}(\mathcal K))^\perp
\subset
H^1(k,H^1(B_{\ov{k}},\mc M)).
\end{equation}
\end{enumerate}
\end{thm}

\begin{proof}
(1) Let $N$ be either $M$ or $M^\vee(1)$, let $\mathcal N\coloneqq g_*N$, and let $v\in H^1(B_{\ov{k}},\mc N)$. Via $\iota_N$, write $v$ as a class in $H^1(G,N)$ represented by a cocycle $c\colon G\to N$. Consider the short exact sequence\[
1\longrightarrow \pi_1(U_{\ov{k}})\longrightarrow \pi_1(U)\longrightarrow G_k\longrightarrow 1.
\]
If $\sigma\in G_k$, choose a lift of $\sigma$ to $\pi_1(U)$, and let $h\in G$ be its image. By the definition of the conjugation action on group cohomology (see \cite[Chapter~I, \S5]{neukirch2008cohomology} or \cite[(7.8)]{schoen1995computation}), $\sigma(v)$ is represented by the cocycle
\[
g\longmapsto hc(h^{-1}gh).
\]
The cocycle identity gives
\[ hc(h^{-1}gh)=c(g)+(g-1)c(h),\]
so this cocycle differs from $c$ by a coboundary. Thus $G_k$ acts trivially on $H^1(B_{\ov{k}},\mc N)$. For $N=M^\vee(1)$ we use Assumption~\ref{schoen-situation-projective}(2), which makes the Tate twist trivial over $k$. This proves (1).

(2) Cup product, followed by the trace map $\operatorname{Tr}_{B_{\ov{k}}}\colon H^2(B_{\ov{k}},\F_\ell(1))\xrightarrow{\sim}\F_\ell$, induces a pairing
\begin{equation}\label{eq:geometric-parabolic-pairing}H^1(B_{\ov{k}},\mc M)\times H^1(B_{\ov{k}},\mc M^\dagger)\longrightarrow \F_\ell.
\end{equation} 
Let $j_U\colon U_{\ov{k}}\hookrightarrow
B_{\ov{k}}$ be the inclusion and set $\mathcal L=\mc M|_{U_{\ov{k}}}$.
Using purity and the natural morphism $(j_U)_!\mathcal L\to R(j_U)_*\mathcal L$, we obtain
\[
 \on{Im}[H^1(B_{\ov{k}},\mc M)
 \hookrightarrow H^1(U_{\ov{k}},\mathcal L)] =
 \operatorname{Im}[
 H_c^1(U_{\ov{k}},\mathcal L)\to
 H^1(U_{\ov{k}},\mathcal L)],
\]
where the injection on the left comes from \eqref{eq:parabolic-etale}.
Similarly,
\[
\on{Im}[H^1(B_{\ov{k}},\mc M^\dagger)\hookrightarrow H^1(U_{\ov{k}},\mathcal L^\vee(1))]
 =\operatorname{Im}[ 
 H_c^1(U_{\ov{k}},\mathcal L^\vee(1))\to
 H^1(U_{\ov{k}},\mathcal L^\vee(1))].
\]
By Poincar\'e duality with compact supports, the two maps
$H_c^1\to H^1$ occurring here are transposes of one another, and hence the pairing \eqref{eq:geometric-parabolic-pairing} is perfect; see, for example, \cite[V.2.2]{milne1980etale} or \cite[\S7]{schoen1995computation}.

Since $G_k$ acts trivially on both factors of
\eqref{eq:geometric-parabolic-pairing}, cup product gives
\[
 H^1(k,H^1(B_{\ov{k}},\mc M))\times
 H^1(B_{\ov{k}},\mc M^\dagger)
 \longrightarrow H^1(k,\F_\ell).
\]
Evaluation at $\operatorname{Frob}_k$ identifies $H^1(k,\F_\ell)$ with $\F_\ell$.
Moreover, since $G_k$ acts trivially on $H^1(B_{\ov{k}},\mc M)$,
evaluating cocycles at $\operatorname{Frob}_k$ gives
\[H^1(k,H^1(B_{\ov{k}},\mc M))\xlongrightarrow{\sim} H^1(B_{\ov{k}},\mc M).\]
We conclude that the pairing $\langle -,-\rangle_{\mathrm{par}}$ is perfect. This proves (2).

(3) Let $x\in U(k)$. Choose a Frobenius element $g_x\in G$ at $x$, and
let $\gamma_x\in\Gamma$ be its image under the surjection $G\to \Gamma$. By
Assumption~\ref{schoen-situation-projective}(8) and Poincar\'e duality, the group $H^2(B_{\ov{k}},\mc M)$ is dual to
$H^0(B_{\ov{k}},\mathcal M^\vee(1))=(M^\vee(1))^\Gamma=0$.
Since $x\in U(k)$, purity and
Assumption~\ref{schoen-situation-projective}(2) give
\begin{equation}\label{eq:local-source-fixed-space}
 H^2_{\ov{x}}(B_{\ov{k}},\mc M)_0^{G_k}
 =H^2_{\ov{x}}(B_{\ov{k}},\mc M)^{G_k}
 \simeq M^{\langle\gamma_x\rangle}.
\end{equation}
In particular, if $\gamma_x$ is conjugate to $\xi$, then
Assumption~\ref{schoen-situation-projective}(10) shows that the group in
\eqref{eq:local-source-fixed-space} is one-dimensional. This proves (3).

(4) Assumption~\ref{schoen-situation-projective}(8) gives $H^0(B_{\ov{k}},\mc M^\dagger)
=(M^\vee(1))^\Gamma=0$. Since $\on{cd}(k)=1$, the
Hochschild--Serre spectral sequence gives an isomorphism
\begin{equation}\label{eq:descend-v}
H^1(B_k,\mc M^\dagger)
\xlongrightarrow{\sim}
H^1(B_{\ov{k}},\mc M^\dagger)^{G_k}=H^1(B_{\ov{k}},\mc M^\dagger),
\end{equation}
where we have used that $G_k$ acts trivially on $H^1(B_{\ov{k}},\mc M^\dagger)$. Let $\widetilde v\in H^1(B_k,\mc M^\dagger)$ be the unique class whose restriction to $B_{\ov{k}}$ is $v$.

We next describe the restriction of $\widetilde v$ to $x$. Over $U$ the
cover $\breve C\to B_k$ gives a quotient
\[
q\colon\pi_1(U)\longrightarrow G.
\]
Inflating $c_v$ along $q$ gives a class on $U$ whose geometric restriction is $v|_{U_{\ov{k}}}$. Since the restrictions of $c_v$ to the inertia groups at $B_{\ov{k}}\setminus U_{\ov{k}}$ are trivial in cohomology, the localization sequence for $U\hookrightarrow B_k$ shows that this class extends to a class in $H^1(B_k,\mc M^\dagger)$. By the uniqueness in \eqref{eq:descend-v}, the resulting class is $\widetilde v$. 

Let
\[
s_x\colon G_k\longrightarrow\pi_1(U)
\]
be the section associated with the chosen geometric point
$\widetilde x$ above $x$. By the definition of $g_x$,
$q(s_x(\on{Frob}_k))=g_x$. It follows that the pullback $x^*\widetilde v\in H^1(k,M^\vee(1))$ 
is represented by the crossed homomorphism $\sigma\mapsto c_v(q(s_x(\sigma)))$, and in particular
\begin{equation}\label{eq:restriction-v-frob}
(x^*\widetilde v)(\on{Frob}_k)=c_v(g_x).
\end{equation}

It remains to compare this restriction with the connecting map
$\delta_x$. As observed in the proof of (3), Assumption
\ref{schoen-situation-projective}(8) and Poincar\'e duality give $H^2(B_{\ov{k}},\mc M)=0.$ In particular,
\[
H^2_{\ov{x}}(B_{\ov{k}},\mc M)_0
=
H^2_{\ov{x}}(B_{\ov{k}},\mc M),
\]
so the map $\delta_x$ of \eqref{coboundary-situation} is the connecting
homomorphism considered in
\Cref{lem:localization-hochschild-serre}. Since $x\in U(k)$, the sheaf
$\mc M$ is locally constant in a neighborhood of $x$, and hence that
lemma applies. Let $\widetilde m\in H^2_x(B_k,\mc M)$ be the unique class such that $\epsilon_x(\widetilde m)=m$. 
By \Cref{lem:localization-hochschild-serre}(3), we have $\epsilon_B(i_{x*}\widetilde m)=\delta_x(m)$. Thus, under the isomorphism
\[
H^2(B_k,\mc M)
\xlongrightarrow{\sim}
H^1(k,H^1(B_{\ov{k}},\mc M)),
\]
coming from the spectral sequence, the class $i_{x*}\widetilde m$ corresponds to $\delta_x(m)$.

We can now compute the pairing. 
Applying $R\Gamma(k,-)$ to the trace morphism $
R\Gamma(B_{\ov{k}},\F_\ell(1))\to \F_\ell[-2]$ gives a morphism
\[
R\Gamma(k,R\Gamma(B_{\ov{k}},\F_\ell(1)))
\longrightarrow
R\Gamma(k,\F_\ell)[-2].
\]
Using the canonical identification $R\Gamma(B_k,\F_\ell(1))
\simeq R\Gamma(k,R\Gamma(B_{\ov{k}},\F_\ell(1)))$, we obtain in degree $3$ a map
\[\on{Tr}_{B/k}\colon 
H^3(B_k,\F_\ell(1))
\longrightarrow
H^1(k,\F_\ell).
\]
By the compatibility of the Hochschild--Serre spectral sequence with
cup products and the functoriality of its edge maps with respect to the
trace morphism, we have
\[
\delta_x(m)\cup v
=
\on{Tr}_{B/k}(i_{x*}\widetilde m\cup\widetilde v\bigr)
\quad\text{in }H^1(k,\F_\ell).
\]
The projection formula for the Gysin map \cite[VI.6.5]{milne1980etale} gives
\[
i_{x*}\widetilde m\cup\widetilde v
=
i_{x*}(\widetilde m\cup x^*\widetilde v) \qquad\text{in }H^3(B_k,\F_\ell(1)).
\]
Compatibility of the trace with the Gysin map therefore gives
\[
\langle\delta_x(m),v\rangle_{\mathrm{par}}
=
\operatorname{ev}_{\on{Frob}_k}
(\widetilde m\cup x^*\widetilde v).
\]
The latter cup product is the usual evaluation pairing
between $M$ and $M^\vee$. Using \eqref{eq:restriction-v-frob}, we obtain
\[
\langle\delta_x(m),v\rangle_{\mathrm{par}}
=
\langle m,(x^*\widetilde v)(\on{Frob}_k)\rangle_M
=
\langle m,c_v(g_x)\rangle_M,
\]
which proves the claim.

(5) Let $v\in\operatorname{Inf}(\mathcal K)$. Thus $v=\operatorname{Inf}(\alpha)$
for some $\alpha\in\mathcal K\subset H^1(\Gamma,M^\vee(1))$. Let $x\in U(k)$ and suppose
that $\gamma_x$ is conjugate to $\xi$. Since
$|\langle\gamma_x\rangle|$ is prime to $\ell$, we have $H^1(\langle\gamma_x\rangle,M^\vee)=0$. Therefore the restriction of $\alpha$ to
$\langle\gamma_x\rangle$ is a coboundary. 
This implies that for every $m\in 
H^2_{\ov{x}}(B_{\ov{k}},\mc M)_0^{G_k}\simeq
M^{\langle\gamma_x\rangle}$, we have $\langle m,c_v(g_x)\rangle_M=0$: 
indeed, after choosing a representative of $\alpha$, we may write
$c_v(g_x)=\gamma_x n-n$ for some $n\in M^\vee$, and then
\[
 \langle m,c_v(g_x)\rangle_M=\langle m,\gamma_xn-n\rangle_M
 =\langle\gamma_x^{-1}m,n\rangle_M-\langle m,n\rangle_M=0.
\]
Therefore, by \eqref{eq:schoen-evaluation} we have $\langle \delta_x(m),v\rangle=0$,
and every element of $\operatorname{Im}(\delta_x)$ thus annihilates
$\operatorname{Inf}(\mathcal K)$. Hence
\[
 \bigcup_{\substack{x\in U(k)\\
 \kappa_k(\operatorname{Frob}_x)\sim\xi}}
 \operatorname{Im}(\delta_x)
 \subset
 (\operatorname{Inf}(\mathcal K))^\perp.
\]

It remains to prove the reverse inclusion. Let $e$ be the order of $\xi$, and let $G_\xi\subset G$ be the inverse image of $\langle\xi\rangle$. Since $A$ is an $\ell$-group and $(e,\ell)=1$, the subgroup $A\subset G_\xi$ is a normal Hall subgroup. By the Schur--Zassenhaus theorem, it admits a complement: there exists a subgroup $S\subset G_\xi$ such that $G_\xi=A\rtimes S$, $A\cap S=1$, and the quotient map $G_\xi\to\langle\xi\rangle$ restricts to an isomorphism $S\xrightarrow{\sim}\langle\xi\rangle$. Choosing $s\in S$ mapping to $\xi$, we obtain an element $s\in G$ of order $e$ mapping to $\xi$. Fix $0\neq m\in M^{\langle\xi\rangle}\simeq \F_\ell$.

For every $v\in H^1(B_{\ov{k}},\mc M^\dagger)$, the restriction of
$c_v$ to $\langle s\rangle$ is a coboundary, since
$H^1(\langle s\rangle,M^\vee)=0$. Thus
$c_v(s)=sn-n$ for some $n\in M^\vee$, and, since $s$ acts on $M$ as
$\xi$, we have
\begin{equation}\label{eq:schoen-s-zero}
 \langle m,c_v(s)\rangle_M=\langle m, sn-n \rangle_M= \langle s^{-1}m,n\rangle_M - \langle m,n\rangle_M=0.
\end{equation}

By \eqref{eq:KN-exact}, we have $\operatorname{Ker}(r)=\operatorname{Inf}(\mathcal K)$. Consider the linear map
\begin{equation}\label{eq:T-map}
 A\longrightarrow
 (
 H^1(B_{\ov{k}},\mc M^\dagger)/\operatorname{Inf}(\mathcal K)
 )^\vee,
 \qquad
 a\longmapsto
 (\bar v\longmapsto\langle m,r(v)(a)\rangle_M).
\end{equation}
We claim that \eqref{eq:T-map} is surjective. It is enough to show that
its transpose is injective. Let $\bar v\neq0$. Then $r(v)\neq0$ because
$\operatorname{Ker}(r)=\operatorname{Inf}(\mathcal K)$. The image of
the $\Gamma$-equivariant map $r(v)\colon A\to M^\vee$ is therefore a non-zero $\Gamma$-submodule of $M^\vee$. By Assumption~\ref{schoen-situation-projective}(9), $M^\vee$ is
irreducible, so $r(v)$ is surjective. Since $m\neq0$, there is
$\lambda\in M^\vee$ with $\lambda(m)\neq0$, and hence there is
$a\in A$ such that
$\langle m,r(v)(a)\rangle_M\neq0$. This proves the claim.

Now let
\[
 \epsilon\in(\operatorname{Inf}(\mathcal K))^\perp
 \subset H^1(k,H^1(B_{\ov{k}},\mc M)).
\]
The functional $v\mapsto\langle\epsilon,v\rangle_{\mathrm{par}}$ on $H^1(B_{\ov{k}},\mc M^\dagger)$ vanishes on $\operatorname{Inf}(\mathcal K)$. By the surjectivity of
\eqref{eq:T-map}, there exists $a\in A$ such that
\begin{equation}\label{eq:epsilon-via-a}
 \langle\epsilon,v\rangle_{\mathrm{par}}
 =\langle m,r(v)(a)\rangle_M
\end{equation}
for every $v\in H^1(B_{\ov{k}},\mc M^\dagger)$. Set $g=as\in G$. By definition, $r(v)=\operatorname{Res}_A(\iota_{M^\vee(1)}(v))$. Since $A$ lies in the kernel of $G\to\Gamma$, it acts trivially on $M^\vee$. Hence $H^1(A,M^\vee)=\operatorname{Hom}(A,M^\vee)$ and the restriction of the class represented by $c_v$ is represented by
the homomorphism $c_v|_A$. In particular, $r(v)(a)=c_v(a)$. The cocycle identity and \eqref{eq:schoen-s-zero} therefore give
\[\langle m,c_v(g)\rangle_M =\langle m,c_v(a)+a c_v(s)\rangle_M =\langle m,r(v)(a)\rangle_M =\langle\epsilon,v\rangle_{\mathrm{par}}.\]
By Assumption~\ref{schoen-situation-projective}(6), there exists
$x\in U(k)$ whose Frobenius conjugacy class in $G$ is the conjugacy
class of $g$. Choosing the point above $x$ used to specify the
Frobenius representative, we may take $g_x=g$. Then
$\gamma_x=\xi$, so $m$ belongs to the local source through
\eqref{eq:local-source-fixed-space}. By \eqref{eq:schoen-evaluation}
and \eqref{eq:epsilon-via-a},
\[
 \langle\delta_x(m),v\rangle_{\mathrm{par}}
 =\langle\epsilon,v\rangle_{\mathrm{par}}
\]
for every $v\in H^1(B_{\ov{k}},\mc M^\dagger)$. By (2), we conclude that $\delta_x(m)=\epsilon$. Therefore
\[
 (\operatorname{Inf}(\mathcal K))^\perp
 \subset
 \bigcup_{\substack{x\in U(k)\\
 \kappa_k(\operatorname{Frob}_x)\sim\xi}}
 \operatorname{Im}(\delta_x),
\]
and this proves \eqref{eq:union-images-delta}.
\end{proof}

\begin{cor}\label{schoen-unified-cor}
Suppose that Assumptions~\ref{schoen-situation-projective}(1)--\ref{schoen-situation-projective}(10) are satisfied, and suppose further that $\mc K=0$. Then, for every $\epsilon\in H^1(k,H^1(B_{\ov{k}},\mathcal M))$, there exists $x\in U(k)$ such that $\kappa_k(\operatorname{Frob}_x)$ is conjugate to $\xi$ and $\epsilon\in\on{Im}(\delta_x)$.
\end{cor}

We conclude this section by showing that Assumptions \ref{schoen-situation-projective}(1)--\ref{schoen-situation-projective}(6) are automatically satisfied after replacing $k_0$ by a suitable finite extension.

\begin{lem}\label{lang-estimate}
Assume that $\breve C$ is geometrically connected over $k$ and that $k(\breve C)/k(B)$ is Galois (this holds under \Cref{schoen-situation-projective}(1,4,5)). Let $S\subset B$ be a proper closed subset of $B$, and let $F\subset\operatorname{Gal}(k(\breve C)/k(B))$ be a non-empty conjugacy class. For every positive integer $N$, there exists a positive integer $N'$ such that, for every finite extension $k'/k$ with $[k':k]\geq N'$, the set
\[\{x'\in(B\setminus S)(k'):\operatorname{Frob}_{x'}=F\}\]
has cardinality greater than $N$.
\end{lem}

\begin{proof}
    This is due to Lang, and it is a consequence of the Riemann hypothesis for Artin L-functions for function fields of transcendence degree $1$ over a finite field proved by Weil; see \cite[p. 393]{lang1956series} or \cite[Proposition 9.9]{schoen1995computation}.
\end{proof}

\begin{prop}\label{3.2.1-3.2.6}
Fix a dense open subscheme $U\subset\dot B_{k_0}\setminus\Sigma$.
    There exists a finite extension $k_1/k_0$ such that Assumptions \ref{schoen-situation-projective}(1)-\ref{schoen-situation-projective}(6) hold for all finite extensions $k/k_1$ (with $U$ replaced by $U_k$).
\end{prop}

\begin{proof}
    Replacing $k_0$ by its algebraic closure in $\overline{k_0(B)}^{\on{Ker}\kappa}$, we may assume that (1) holds for every finite extension $k/k_0$. Increasing $k_0$, we may assume that (2), (3) and (5) also hold for every finite extension $k/k_0$. For (4), let $J\coloneqq \on{Pic}^0(C)$ be the Jacobian of $C$. The short exact sequence
\[0\longrightarrow J[\ell]\longrightarrow J\xlongrightarrow{\times\ell} J\longrightarrow 0\]
induces, for every finite field extension $k/k_0$, an isomorphism
\begin{equation}\label{eq:j-mod-ell}
    J(k)/\ell \simeq H^1(k, J[\ell])\simeq H^1(k,H^1(C_{\ov{k}},\Z/\ell))=\on{Hom}(G_k,H^1(C_{\ov{k}},\Z/\ell)),
\end{equation}
where the last identification holds by (5). The isomorphism (\ref{eq:j-mod-ell}) is functorial in $k$. Since $H^1(C_{\ov{k}},\Z/\ell)$ is $\ell$-torsion, it follows that every element of $J(k_0)/\ell$ becomes zero in $J(k_0')/\ell$, where $k_0\subset k_0'\subset \overline{k}_0$ is the extension of degree $\ell$. Replacing $k_0$ by $k_0'$, we obtain that (4) holds for all finite extensions $k/k_0$. Finally, applying \Cref{lang-estimate} with $S=B\setminus U$, increasing if necessary $k_0$, we may assume that (6) holds over every finite extension $k/k_0$.
\end{proof}

\section{Proofs of the conditional results}\label{sec:4}

\subsection{Proof of Theorem~\ref{main-conditional-liftable}}

\begin{lem}\label{surjective-rationally-integrally}
Let $k$ be a finite field, and let $X$ and $Y$ be smooth projective varieties of dimensions $d_X$ and $d_Y$ over $k$,
respectively. Let $i,j,n$ be integers, with $i,j\geq 0$, and let
$z\in Z^j(X\times Y)$ be a cycle. Suppose that there exist an integer
$N>0$ and a cycle
$
w\in Z^{d_X+d_Y-j}(Y\times X)
$
such that
$
z_*\circ w_*=N\cdot \id
$
on
$
H^{i+2(j-d_X)}(Y_{\ov{k}},\Q_\ell(n+j-d_X))
$
for every sufficiently large prime $\ell$. Then the map
\[
z_*\colon H^i(X_{\ov{k}},\Z_\ell(n))
\longrightarrow  H^{i+2(j-d_X)}(Y_{\ov{k}},\Z_\ell(n+j-d_X))
\]
is surjective for all sufficiently large primes $\ell$.
\end{lem}

\begin{proof}
By Gabber's theorem on torsion in $\ell$-adic cohomology
\cite{gabber1983torsion}, the group $H^{i+2(j-d_X)}(Y_{\ov{k}},\Z_\ell(n+j-d_X))$
is torsion-free for all sufficiently large primes $\ell$.
For such $\ell$, we have $z_*\circ w_*=N\cdot \id$ on $H^{i+2(j-d_X)}(Y_{\ov{k}},\Z_\ell(n+j-d_X))$. Indeed, it holds after tensoring with $\Q_\ell$, and the target is
torsion-free. After discarding the finitely many primes dividing $N$, multiplication by $N$ is an automorphism of this $\Z_\ell$-module. Hence, for every $y\in H^{i+2(j-d_X)}(Y_{\ov{k}},\Z_\ell(n+j-d_X))$, we have $y=z_*(N^{-1}w_*(y))$. Thus $z_*$ is surjective for all sufficiently large $\ell$.
\end{proof}

\begin{thm}\label{tate-aj-liftable-threefold}
Let $k_0$ be a finite field of characteristic $p>2$, and assume that \Cref{tate-conjecture}(2) holds for surfaces. Let $X/k_0$ be a smooth projective geometrically integral threefold which, after possibly replacing $k_0$ by a finite extension, admits a smooth projective
lifting to the ring of Witt vectors of $k_0$.

For all but finitely many primes $\ell\neq p$, for every finite extension $k/k_0$, the Abel--Jacobi map
\[
\alpha^2_{X_k,\ell}\colon
Z^2_{\mathrm{hom},\ell}(X_k)
\longrightarrow
H^1(k,H^3(X_{\ov{k}},\Z_\ell(2)))
\]
is surjective.
\end{thm}

\begin{proof}
\Cref{lem:Abel--Jacobi-corestriction} allows us to replace $k_0$ by a finite extension throughout the proof. We may assume that $X$ admits a smooth projective lifting $\mathcal X\to\operatorname{Spec}W(k_0)$. If $b_2(X)$ is odd, set $Y=X$ and $\mathcal Y=\mathcal X$. If $b_2(X)$ is even, after a further finite extension choose a point $x\in X(k_0)$. Since $\mathcal X$ is smooth over $W(k_0)$, the point $x$ lifts to a section $\widetilde x$ of $\mathcal X$. Set $Y=\operatorname{Bl}_xX$ and $\mathcal Y=\operatorname{Bl}_{\widetilde x}\mathcal X$. Then $\mathcal Y$ is a smooth projective lifting of $Y$, and $b_2(Y)=b_2(X)+1$ is odd. The blow-up formula gives $H^3(Y_{\ov{k}_0},\Z_\ell(2))\simeq H^3(X_{\ov{k}_0},\Z_\ell(2))$, and compatibility of Abel--Jacobi maps with proper pushforward shows that it suffices to prove the assertion for $Y$.

Choose a relatively ample line bundle $\mathcal L_0$ on $\mathcal Y$ and let $L_0$ be its restriction to $Y$. Thus the pair $(Y,L_0)$ lifts to $(\mathcal Y,\mathcal L_0)$ over $W(k_0)$. We fix an embedding $W(k_0)\hookrightarrow\C$, so that $\mathcal Y$ gives rise to a complex projective manifold $Y^{an}$ and $\mathcal L_0$ gives rise to a line bundle $L_0^{an}$ on $Y^{an}$. By \cite[Lemma~9.5.2]{schoen1999image}, up to replacing $k_0$ by a finite extension if necessary, there exists a positive integer $r_0$ such that the conclusion of \cite[Lemma~9.3.2]{schoen1999image} is satisfied, that is, for every $r\geq r_0$:\begin{itemize}
\item[--] the line bundle $L_0^{\otimes r}$ is very ample,
\item[--] $H^i(Y,L_0^{\otimes r})=0$ for all $i>0$, and
\item[--] there exists a Lefschetz pencil $B\subset\P H^0(Y_{k_{0,r}},L_0^{\otimes r})$, for some finite field extension $k_{0,r}/k_0$,
\end{itemize}
and moreover the conclusion of \cite[Lemma~9.5.2]{schoen1999image} is satisfied, that is,
\begin{itemize}
\item[--] the pencil $B\subset\P H^0(Y,L_0^{\otimes r_0})$ is defined over $k_0$,
\item[--] the integer $r_0$ is even,
\item[--] there is a complex hypersurface on $Y^{an}$ belonging to the linear system $|(L_0^{an})^{\otimes r_0}|$ whose singular locus consists of a unique closed point $P$ given, locally in the analytic topology, by the equation
\[
z_1^3+z_2^3+z_3^4=0,
\]
\item[--] the rank of the vanishing cohomology $E$ is at least $5$, and
\item[--] the signature $(a,b)$ of the intersection product restricted to $E$ satisfies $a>1$ and $b>1$.
\end{itemize}
Set $L\coloneqq L_0^{\otimes r_0}$, and let $W\to B$ be the corresponding Lefschetz pencil over $k_0$; see \cite[(9.1.1)]{schoen1999image}. The morphism $\sigma\colon W\to Y$ is the blow-up of $Y$ along the base locus of the pencil. By the blow-up formula, the pushforward map
\[
\sigma_*\colon H^3(W_{\ov{k}},\Z_\ell(2))\longrightarrow H^3(Y_{\ov{k}},\Z_\ell(2))
\]
is surjective for every prime $\ell\neq p$ and every finite extension $k/k_0$, and hence so is the induced map
\[
H^1(k,H^3(W_{\ov{k}},\Z_\ell(2)))\longrightarrow H^1(k,H^3(Y_{\ov{k}},\Z_\ell(2))).
\]
It therefore suffices to prove the surjectivity of $\alpha^2_{W,\ell}$.

Let $k/k_0$ be a finite extension. As in \Cref{paragraph-relative-aj}, the Leray spectral sequence for $f_{\ov{k}}\colon W_{\ov{k}}\to B_{\ov{k}}$ gives rise to a filtration $L^2\subset L^1\subset L^0$ on $H^3(W_{\ov{k}},\Z_\ell(2))$:
\begin{align*}
L^0&=H^3(W_{\ov{k}},\Z_\ell(2)),\\
L^1&=\operatorname{Ker}[H^3(W_{\ov{k}},\Z_\ell(2))\to H^0(B_{\ov{k}},R^3(f_{\ov{k}})_*\Z_\ell(2))],\\
L^2&=\operatorname{Ker}[L^1\to H^1(B_{\ov{k}},R^2(f_{\ov{k}})_*\Z_\ell(2))]\\
&=\operatorname{Im}[H^2(B_{\ov{k}},R^1(f_{\ov{k}})_*\Z_\ell(2))\to H^3(W_{\ov{k}},\Z_\ell(2))].
\end{align*}
For $i=0,1,2$, set
\[
L^iH^1(k,H^3(W_{\ov{k}},\Z_\ell(2)))\coloneqq
\operatorname{Im}[H^1(k,L^i)\to H^1(k,H^3(W_{\ov{k}},\Z_\ell(2)))].
\]
It suffices to check conditions \textup{(i)--(iii)} of \Cref{conditions} for $f$ with $m=1$.

First, after replacing $k_0$ by a finite extension, choose a rational point $x\in\dot B(k_0)$ and set $V\coloneqq f^{-1}(x)$. Then $V$ is a smooth projective surface over $k_0$. By \Cref{rmk:AJ_properties}(7), the map
\[
\alpha^1_{V_k,\ell}\colon Z^1_{\mathrm{hom},\ell}(V_k)\longrightarrow H^1(k,H^1(V_{\ov{k}},\Z_\ell(1)))
\]
is surjective for every finite extension $k/k_0$. This proves condition \textup{(i)}.

Next, we verify condition \textup{(iii)}. Consider the composite
\[
H^1(V_{\ov{k}},\Z_\ell(1))\longrightarrow H^3(Y_{\ov{k}},\Z_\ell(2))
\longrightarrow H^3(W_{\ov{k}},\Z_\ell(2))\longrightarrow L^0/L^1.
\]
Since $V$, $Y$ and $W$ lift to characteristic zero, for all but finitely many $\ell$ the cokernel of this map is torsion-free. In the proof of \cite[Lemma~(10.2.4)]{schoen1999image}, it is shown that the cokernel is torsion for every $\ell\neq p$. Thus the map is surjective for all but finitely many primes $\ell$. Passing to Galois cohomology and using the surjectivity of $\alpha^1_{V,\ell}$ together with the compatibility of Abel--Jacobi maps with proper pushforward \Cref{rmk:AJ_properties}(2), we conclude that the composite
\[
Z^2_f(W)\xlongrightarrow{\alpha^2_{W,\ell}}H^1(k,L^0)\longrightarrow H^1(k,L^0/L^1)
\]
is surjective. Hence condition \textup{(iii)} holds.

It remains to verify condition \textup{(ii)}. Write $j\colon\dot B\hookrightarrow B$ for the smooth locus of the pencil and $\dot f\colon\dot W\to\dot B$ for the restriction of $f$. For every prime $\ell$ outside the finite set excluded in \cite[(9.5.3)]{schoen1999image}, let $E_\ell\subset H^2(W_{\bar\eta},\Z_\ell(1))$ be the vanishing-cohomology lattice at a geometric generic point, let $E_r=E_\ell/\ell^rE_\ell$, let $\dot{\mathcal E}_r$ be the corresponding locally constant sheaf on $\dot B$, and set
\[
\mathcal E_r=j_*\dot{\mathcal E}_r,\qquad
\mathcal E=\{\mathcal E_r\}_{r\geq1}.
\]
The compatible orthogonal decompositions of \cite[(9.5.3)--(9.5.4)]{schoen1999image}, cf. \cite[(4.3)]{deligne1980weil}, give
\[
R^2f_*\Z_\ell(1)\simeq\mathcal E\oplus\mathcal E^\perp.
\]
Set $\mathcal H=R^2f_*\Z_\ell(2)$, and let
\[
q\colon\mathcal H(-1)=R^2f_*\Z_\ell(1)\longrightarrow\mathcal E
\]
be the resulting projection. After twisting,
\[
\mathcal H\simeq\mathcal E(1)\oplus\mathcal E^\perp(1).
\]
Geometric monodromy is trivial on $\mathcal E^\perp$, and hence $H^1(B_{\bar k},\mathcal E^\perp(1))=0$ because $B_{\bar k}\simeq\P^1_{\bar k}$. Thus the Leray identification and $q$ give an isomorphism
\begin{equation}\label{eq:q-star-isomorphism}
q_*\colon H^1(k,L^1/L^2)\xlongrightarrow{\sim}
H^1(k,H^1(B_{\bar k},\mathcal E(1))).
\end{equation}

Set $M=(\dot{\mathcal E}_1(1))_{\bar\eta}$; then, 
in the notation of \Cref{sec:connecting-map}
and in addition letting $\dot g \colon \eta\to \dot B$ be the inclusion of the generic point, we have 
\[\mathcal E_1(1)=j_*\dot{\mathcal E}_1(1)=j_*(\dot g_* \dot g^*\mathcal E_1(1))=j_*\dot g_*M=g_*M,\]
where the second equality holds because $\dot{\mathcal E}_1(1)$ is locally constant and the third equality holds by the definition of $M$.
The sheaves $\mathcal E_r$ are flat over $\Z/\ell^r$, their restrictions to $\dot B$ are locally constant, and $H^0(\dot B_{\bar k},\mathcal E_1^\vee)=0$. In the notation of Assumption~\ref{schoen-situation-projective}, we take $U\coloneqq \dot B\setminus\Sigma$.
After the extension provided by \Cref{3.2.1-3.2.6}, Assumptions \ref{schoen-situation-projective}(1)--\ref{schoen-situation-projective}(6) hold over every further finite extension. Lemma~(9.5.5) of \cite{schoen1999image} gives Assumptions~\ref{schoen-situation-projective}(7)--\ref{schoen-situation-projective}(9), while \cite[Proposition~(9.5.6)]{schoen1999image} gives Assumption~\ref{schoen-situation-projective}(10) and the vanishing of $H^1(\Gamma,M^\vee(1))$, hence the vanishing of $\mathcal K$, for all but finitely many $\ell$. Let 
\[
\epsilon_1,\ldots,\epsilon_t \in H^1(k,H^1(B_{\bar k},\mathcal E_1(1)))
\]
be an $\F_\ell$-basis. By \Cref{schoen-unified-cor}, for each $i$ there are $k$-points $x_i\in\dot B(k)$ and classes $\bar e_i\in H^2_{\bar x_i}(B_{\bar k},\mathcal E_1(1))_0^{G_k}$ such that $H^2_{\bar x_i}(B_{\bar k},\mathcal E_1(1))_0^{G_k}$ is one-dimensional generated by $\bar e_i$, and $\delta_{\mathcal E_1,x_i}(\bar e_i)=\epsilon_i$. 

The lattice $E_\ell$ has odd rank because $b_2(Y)$ is odd and $r_0$ is even; see \cite[Lemma~(9.2.1)(i)]{schoen1999image}. Therefore \cite[Lemma~(8.2.2)]{schoen1999image} gives
\[
H^2_{\bar x_i}(B_{\bar k},\mathcal E(1))_0^{G_k}\simeq\Z_\ell,\qquad
H^2_{\bar x_i}(B_{\bar k},\mathcal E(1))_0^{G_k}/\ell
\xlongrightarrow{\sim}
H^2_{\bar x_i}(B_{\bar k},\mathcal E_1(1))_0^{G_k}.
\]
In particular, we may choose lifts $e_i\in H^2_{\bar x_i}(B_{\bar k},\mathcal E(1))_0^{G_k}$ of $\bar e_i$. 

Since $m=1$, the Tate conjecture for divisors on the smooth projective fibers of $f$, together with \cite[Lemma~(8.2.3)]{schoen1999image}, implies Hypothesis~(9.6.1) of \cite{schoen1999image}; see \cite[Remark~(9.6.2)(i)]{schoen1999image}. Thus the $\Z_\ell$-linear extension of the fiberwise cycle-class map appearing in the composite defining $q_\#$ in \eqref{q-sharp} is surjective.

Moreover, since
\[
q(1)\colon\mathcal H\longrightarrow\mathcal E(1)
\]
is the projection onto a direct summand, the last map in that composite is split surjective. It follows that the $\Z_\ell$-linear extension of
\[
q_\#\colon Z_f^2(W)\longrightarrow Z(\mathcal E)
\]
is surjective. Applying this to the element of $Z(\mathcal E)$ whose $x_i$-component is $e_i$ and whose other components are zero, we obtain, for every $i=1,\dots,t$, an element $\xi_i\in Z^1(f^{-1}(x_i))_0\otimes\Z_\ell$ such that $q_\#(\xi_i)=e_i$. In particular, $q_\#(\xi_i)$ reduces to $\bar e_i$ modulo $\ell$.

For every $i=1,\dots,t$, write
\[\xi_i=\sum_j a_{ij}z_{ij},\qquad a_{ij}\in\Z_\ell,\quad z_{ij}\in Z^1(f^{-1}(x_i))_0.\]
Choosing integers $n_{ij}\equiv a_{ij}\pmod\ell$ and setting $d_i=\sum_jn_{ij}z_{ij}\in Z^1(f^{-1}(x_i))_0$, we deduce that $q_\#(d_i)$ reduces to $\bar e_i$ modulo $\ell$. The classes $q_\#(d_i)$ now satisfy the hypotheses of \Cref{schoen-2.4.3-finer}, so their $\alpha_{\mathcal E}$-images generate $H^1(k,H^1(B_{\bar k},\mathcal E(1)))$. By \Cref{lem:comp_f_E} and \eqref{eq:q-star-isomorphism}, the classes $\alpha_{f,\ell}(d_i)$ generate $H^1(k,L^1/L^2)$. Thus condition \textup{(ii)} holds. Conditions \textup{(i)} and \textup{(iii)} were proved above, and \Cref{conditions} shows that $\alpha^2_{W,\ell}$ is surjective. The surjectivity of the pushforward from $W$ to $Y$, and then from $Y$ to $X$ in the blow-up case, proves the theorem.
\end{proof}

\begin{lem}\label{lem:geometric-cycle-map-large-ell}
Let $k_0$ be a finite field of characteristic $p$,
and assume that \Cref{tate-conjecture}(2) holds for surfaces.
Let $X/k_0$ be a smooth projective geometrically integral variety of dimension $d$ which admits a smooth projective lifting to the ring of Witt vectors of $k_0$. For all but finitely many primes $\ell\neq p$, for every finite extension $k/k_0$, the cycle map
\[
CH^{d-1}(X_k)_{\Z_\ell}\longrightarrow H^{2d-2}(X_{\ov{k}},\Z_\ell(d-1))^{G_k}
\]
is surjective.
\end{lem}

\begin{proof}
The conclusion is immediate if $d=1$; from now on, we assume $d\geq 2$. By assumption, $X$ admits a smooth projective lifting $\mathcal X\to \on{Spec}W(k_0)$. Choose a relative ample line bundle $\mathcal L_0$ on $\mathcal X$ and let $L_0$ be its restriction to $X$. The pair $(X,L_0)$ lifts to $(\mathcal X, \mathcal L_0)$ over $W(k_0)$. We fix an embedding $W(k_0)\hookrightarrow \C$ so that $\mathcal X$ gives rise to a complex projective manifold $X^{an}$ and $\mathcal L_0$ gives rise to a line bundle $L_0^{an}$ on $X^{an}$. Hard Lefschetz shows that the cup product with $c_1(L_0^{an})^{d-2}$ induces a map 
\[H^2(X^{an},\Z)\longrightarrow H^{2d-2}(X^{an},\Z)\]
with finite kernel and cokernel,  hence by smooth proper base change, for all but finitely many $\ell\neq p$, the cup product with $c_1(L_0)^{d-2}$ induces an isomorphism of continuous $G_{k_0}$-modules
\[
H^2(X_{\ov{k}_0},\Z_\ell(1))\xlongrightarrow{\sim} H^{2d-2}(X_{\ov{k}_0},\Z_\ell(d-1)).
\]
Let $k/k_0$ be a finite extension.
Since the Tate conjecture for divisors on surfaces implies the conjecture for divisors on smooth projective varieties of any dimension \cite{morrow2019variational}, the cycle map
\[
CH^1(X_k)_{\Z_\ell}\longrightarrow H^2(X_{\ov{k}},\Z_\ell(1))^{G_k}
\]
is surjective, hence so is the cycle map
\[
CH^{d-1}(X_k)_{\Z_\ell}\longrightarrow H^{2d-2}(X_{\ov{k}},\Z_\ell(d-1))^{G_k},
\]
as desired.
\end{proof}

\begin{proof}[Proof of \Cref{main-conditional-liftable}]
The theorem is immediate when $d=1$, for all primes $\ell\neq p$; we now assume $d\geq 2$. Fix a prime $\ell\neq p$ outside the finite exceptional sets in
\Cref{tate-aj-liftable-threefold} and \Cref{lem:geometric-cycle-map-large-ell}. 
By \Cref{lem:geometric-cycle-map-large-ell}, for every finite extension $k/k_0$,
the map
\[
CH^{d-1}(X_k)_{\Z_\ell}
\longrightarrow
H^{2d-2}(X_{\bar k},\Z_\ell(d-1))^{G_k}
\]
is surjective.

We next prove that, for every finite extension $k/k_0$, the Abel--Jacobi map
\[
\alpha^{d-1}_{X_k,\ell}\colon
Z^{d-1}_{\mathrm{hom},\ell}(X_k)
\longrightarrow
H^1(k,H^{2d-3}(X_{\bar k},\Z_\ell(d-1)))
\]
is surjective.
By \Cref{lem:Abel--Jacobi-corestriction}, it is enough to establish the surjections after replacing $k_0$ by a finite extension.
If $d=2$, this
follows from \Cref{rmk:AJ_properties}(7), while if $d=3$ it follows
from \Cref{tate-aj-liftable-threefold}. Suppose that $d\geq4$. 
Since $X$ is liftable, after replacing $k_0$ by a finite extension, we may choose a smooth threefold complete intersection $Y\subset X$ which is liftable.
By \Cref{tate-aj-liftable-threefold}, the map
\[
\alpha^2_{Y_k,\ell}\colon
Z^2_{\mathrm{hom},\ell}(Y_k)
\longrightarrow
H^1(k,H^3(Y_{\bar k},\Z_\ell(2)))
\]
is surjective. By weak Lefschetz
and Poincar\'e duality, the Gysin map
\[
H^3(Y_{\bar k},\Z_\ell(2))
\longrightarrow
H^{2d-3}(X_{\bar k},\Z_\ell(d-1))
\]
is surjective. Since a finite field has $\ell$-cohomological dimension
one, the induced map on $H^1(k,-)$ is also surjective. Compatibility of
Abel--Jacobi maps with proper pushforward
\Cref{rmk:AJ_properties}(2) now gives the surjectivity of
$\alpha^{d-1}_{X_{k},\ell}$.

Therefore, for every finite extension $k/k_0$, both
the maps
\[
CH^{d-1}(X_k)_{\Z_\ell}
\longrightarrow
H^{2d-2}(X_{\bar k},\Z_\ell(d-1))^{G_k}
\]
and
\[
\alpha^{d-1}_{X_k,\ell}\colon
Z^{d-1}_{\mathrm{hom},\ell}(X_k)
\longrightarrow
H^1(k,H^{2d-3}(X_{\bar k},\Z_\ell(d-1)))
\]
are surjective. By the Hochschild--Serre exact sequence and
\Cref{rmk:AJ_properties}(1), the cycle map
\[
CH^{d-1}(X_k)_{\Z_\ell}
\longrightarrow
H^{2d-2}(X_k,\Z_\ell(d-1))
\]
is therefore surjective. The final assertion for $d=3$ follows from the rest and \cite[Th\'eor\`eme~2.2]{colliot2013cycles}.
\end{proof}

\subsection{A variant of Theorem~\ref{main-conditional-liftable}}

Assuming the following stronger form of the Tate Conjecture, we may drop the liftability assumption in \Cref{main-conditional-liftable}.

\begin{conj}\label{tate-conjecture}
    Let $k_0$ be a finite field of characteristic $p>0$, and let $\ell\neq p$ be a prime. For every finite extension $k/k_0$ and every smooth projective $k$-variety $X$:
    \begin{enumerate}
        \item[(1)] the Frobenius element $\phi\in G_k$ acts semi-simply on $H^*(X_{\ov{k}},\Q_\ell)$;
        \item[(2)] for all $i\geq 0$, the cycle class map $CH^i(X)_{\Q_\ell}\to H^{2i}(X_{\ov{k}},\Q_\ell(i))^{G_k}$ is surjective.
    \end{enumerate}
\end{conj}

\begin{lem}\label{lem:tate-independence}
Assume that \Cref{tate-conjecture} holds over $k_0$ for one prime $\ell_0\ne p$. Then, for every finite extension $k/k_0$, every smooth projective $V/k$, every $r$, and every prime $\ell\ne p$, numerical and $\ell$-adic homological equivalence agree on $CH^r(V)_\Q$. Moreover, the $\ell$-adic cycle-class map $CH^r(V)_{\Q_\ell}\to H^{2r}(V_{\bar k},\Q_\ell(r))^{G_k}$ is surjective.
\end{lem}

\begin{proof}
By \cite[Theorem~2.9]{tate1994conjectures}, the assumptions at
$\ell_0$ imply that the rank of $CH^r(V)_\Q$ modulo numerical equivalence
equals the order of the pole of $\zeta(V,s)$ at $s=r$. Let this common
integer be $a$. It is also the algebraic multiplicity of $|k|^r$ as an
eigenvalue of Frobenius on $H^{2r}(V_{\bar k},\Q_\ell)$, and is therefore
independent of $\ell$. For an arbitrary $\ell\ne p$, $\ell$-adic
homological equivalence is finer than numerical equivalence. Consequently
the image of the cycle-class map has dimension at least $a$. It is
contained in the Frobenius-fixed subspace, whose dimension is at most the
algebraic multiplicity $a$. All these dimensions are therefore equal.
This proves both assertions.
\end{proof}

\begin{prop}\label{tate-implies-aj-surjective}
    Let $k_0$ be a finite field of characteristic $p>2$, and assume that \Cref{tate-conjecture} holds over $k_0$ for some prime $\ell_0$.
    Let $X$ be a smooth projective $k_0$-variety of dimension $d$.
    For all but finitely many primes $\ell\neq p$, for every finite extension $k/k_0$, the Abel--Jacobi maps
    \[\alpha^2_{X_k,\ell}\colon Z^2_{\mathrm{hom},\ell}(X_k)\longrightarrow H^1(k,H^3(X_{\ov{k}},\Z_\ell(2)))\] and \[\alpha^{d-1}_{X_k,\ell}\colon Z^{d-1}_{\mathrm{hom},\ell}(X_k)\longrightarrow H^1(k,H^{2d-3}(X_{\ov{k}},\Z_\ell(d-1)))
    \]
    are surjective.
\end{prop}

\begin{proof}
By \Cref{lem:Abel--Jacobi-corestriction}, it is enough to establish the surjections after replacing $k_0$ by a finite extension.
If $d\leq 1$, both targets vanish and there is nothing to prove. Suppose that $d=2$.
The map
\[
\alpha^1_{X_k,\ell}\colon Z^1_{\mathrm{hom},\ell}(X_k)
\longrightarrow H^1(k,H^1(X_{\bar k},\Z_\ell(1)))
\]
is surjective by \Cref{rmk:AJ_properties}(7). After replacing $k_0$ by a
finite extension, choose a smooth projective geometrically connected
hyperplane section $i\colon C\hookrightarrow X$ defined over $k_0$.
Integral weak Lefschetz and Poincar\'e duality give a surjective Gysin map
\[
i_*\colon H^1(C_{\bar k},\Z_\ell(1))\relbar\joinrel\twoheadrightarrow
H^3(X_{\bar k},\Z_\ell(2)).
\]
Since a finite field has $\ell$-cohomological dimension one, the induced map on $H^1(k,-)$ is also surjective. The codimension-one Abel--Jacobi map on $C_k$ is surjective by \Cref{rmk:AJ_properties}(7), and compatibility with
proper pushforward therefore proves the surjectivity of $\alpha^2_{X_k,\ell}$. This proves the proposition for $d=2$.

Suppose now that $d=3$. As explained in
\cite[\S7.1.2]{schoen1996varieties}, the Tate conjecture together with
semisimplicity of Frobenius on
$H^3(X_{\bar k},\Q_{\ell_0})$ gives smooth projective geometrically connected
curves $T_1,T_2,T_3$ over $k_0$ and a cycle
\[
z\in Z^3(T\times X),\qquad T=T_1\times T_2\times T_3,
\]
such that $z_*\colon H^3(T_{\bar k},\Q_{\ell_0})\to
H^3(X_{\bar k},\Q_{\ell_0})$ is surjective. By
\cite[Lemma~(10.2.1)]{schoen1999image}, the same map is surjective for every
prime $\ell\ne p$.

Using the semisimplicity of Frobenius, we choose a $G_{k_0}$-equivariant section
\[s\colon H^3(X_{\bar k},\Q_{\ell_0}) \longrightarrow H^3(T_{\bar k},\Q_{\ell_0})\]
of $z_*$. By Poincar\'e duality and the K\"unneth decomposition, $s$
corresponds to a Tate class in the degree-$(3,3)$ summand of $H^6((X\times T)_{\bar k},\Q_{\ell_0}(3))$. Hence the Tate conjecture for $X\times T$ gives $w_{\ell_0}\in CH^3(X\times T)_{\Q_{\ell_0}}$ such that the induced map $(w_{\ell_0})_*\colon H^3(X_{\bar k},\Q_{\ell_0}) \to H^3(T_{\bar k},\Q_{\ell_0})$ is equal to $s$.

We now show that $w_{\ell_0}$ may be replaced, for our purposes, by a
correspondence with rational coefficients. For a smooth projective variety
$Y$, write $N^3(Y)_\Q$ for the quotient of $CH^3(Y)_\Q$ by numerical equivalence. 
Let $\pi_X^3$ be the Katz--Messing degree-three projector, which is a
rational linear combination of graphs of powers of Frobenius; see
\cite{katz1974consequences}. Consider the $\Q$-linear map
\[
F\colon N^3(X\times T)_\Q
\longrightarrow N^3(X\times X)_\Q,
\qquad
w\longmapsto
\pi_X^3\circ z\circ w\circ\pi_X^3.
\]
Since $w_{\ell_0}$ induces the section $s$, the correspondence $\pi_X^3\circ z\circ w_{\ell_0}\circ\pi_X^3-\pi_X^3$ acts trivially on $\ell_0$-adic cohomology. Under the Tate conjecture, $\ell_0$-adic homological equivalence and numerical equivalence agree
\cite[Theorem~2.9]{tate1994conjectures}. Therefore $\pi_X^3$ belongs to $\operatorname{Im}(F\otimes_\Q\Q_{\ell_0})$.
Since $\Q_{\ell_0}$ is faithfully flat over $\Q$ and $\pi_X^3$ is defined over $\Q$, it follows that $\pi_X^3\in\operatorname{Im}(F)$, that is, there exists a class $\ov{w}\in N^3(X\times T)_\Q$ such that $\pi_X^3\circ z\circ \ov{w}\circ\pi_X^3=\pi_X^3$ modulo numerical equivalence. Choose a representative $w\in CH^3(X\times T)_\Q$ of $\ov{w}$. Thus the cycle $\pi_X^3\circ(z\circ w-\Delta_X)\circ\pi_X^3$ is numerically trivial, and hence by \Cref{lem:tate-independence}, its $\ell$-adic realization is zero for every prime $\ell\ne p$. It follows that
\[
z_*\circ w_*=\operatorname{id}
\quad\text{on }H^3(X_{\bar k},\Q_\ell)
\]
for every $\ell\ne p$. (A similar argument is used in the proof of \cite[Lemma~(10.2.1)]{schoen1999image}.)

Choose $N>0$ and
$w_N\in Z^3(X\times T)$ such that
$z_*\circ(w_N)_*=N\operatorname{id}$ on $H^3(X_{\bar k},\Q_\ell)$ for every
$\ell\ne p$. By \Cref{surjective-rationally-integrally},
\[
z_*\colon H^3(T_{\bar k},\Z_\ell)\relbar\joinrel\twoheadrightarrow
H^3(X_{\bar k},\Z_\ell)
\]
is surjective for all but finitely many $\ell$. It remains surjective after
applying $H^1(k,-)$ because $H^2(k,-)=0$. Since $T$ lifts to characteristic zero, compatibility of Abel--Jacobi maps with correspondences and \Cref{tate-aj-liftable-threefold} now prove the
assertion for $X$. This completes the proof for $d=3$.

Finally, suppose that $d\geq 4$. Let $T \subset X$ be a smooth threefold complete intersection of ample divisors (it exists after replacing $k_0$ by a suitable finite extension). By the Lefschetz hyperplane theorem \cite[Chapter VI, Theorem 7.1]{milne1980etale}, the Gysin map
\[
H^3(T_{\ov{k}}, \Z_\ell(2)) \longrightarrow H^{2d-3}(X_{\ov{k}}, \Z_\ell(d-1))
\]
is surjective. Thus, for every finite extension $k/k_0$, the induced map
\[
H^1(k,H^3(T_{\ov{k}}, \Z_\ell(2))) \longrightarrow H^1(k,H^{2d-3}(X_{\ov{k}}, \Z_\ell(d-1)))
\]
is surjective. Since $\dim(T)=3$, by the case $d=3$ proved above, for all but finitely many primes $\ell$ and every finite extension $k/k_0$, the map $\alpha_{T_k,\ell}^2$ is surjective. By \Cref{rmk:AJ_properties}(2), it now follows that for all finite extensions $k/k_0$ the map $\alpha_{X_k,\ell}^{d-1}$ is surjective.

By \cite[Theorem 4.1(1)]{kleiman1994standard}, \Cref{tate-conjecture}(2) for $X\times X$
implies the Lefschetz standard conjecture for $X$, and hence the existence of an inverse Lefschetz correspondence $z\in CH^3(X\times X)_{\Q}$ such that the induced map
\[
z_* \colon  H^{2d-3}(X_{\ov{k}}, \Q_{\ell_0}(d-1)) \longrightarrow
H^3(X_{\ov{k}}, \Q_{\ell_0}(2))
\]
is inverse to the hard-Lefschetz isomorphism. Let
\[
L^{d-3}_*\colon H^3(X_{\ov{k}},\Q_{\ell}(2))\longrightarrow
H^{2d-3}(X_{\ov{k}},\Q_{\ell}(d-1))
\]
be the map induced by intersecting with $d-3$ hyperplane sections. If
$\pi_X^3$ again denotes the Katz--Messing degree-three projector, then
\[
\pi_X^3\circ(z\circ L^{d-3}-\Delta_X)\circ\pi_X^3
\]
acts trivially on $\ell_0$-adic cohomology. Homological and numerical
equivalence agree under the Tate conjecture
\cite[Theorem~2.9]{tate1994conjectures}, so this numerical correspondence
is zero; by \Cref{lem:tate-independence}, it has zero realization for every
$\ell\ne p$. There exist an integer $N>0$ and an integral cycle, still denoted
$z\in Z^3(X\times X)$, such that
\[
z_*\circ L^{d-3}_*=N\cdot \id
\]
on $H^3(X_{\ov{k}},\Q_\ell(2))$ for every prime $\ell\neq p$. By
\Cref{surjective-rationally-integrally}, with $i=2d-3$, $j=3$, and
$n=d-1$, this implies that
\[
z_* \colon  H^{2d-3}(X_{\ov{k}}, \Z_{\ell}(d-1)) \longrightarrow
H^3(X_{\ov{k}}, \Z_{\ell}(2))
\]
is surjective for all but finitely many primes $\ell$. Thus the induced map
\[
z_*\colon
H^1(k,H^{2d-3}(X_{\ov{k}},\Z_\ell(d-1)))
\longrightarrow
H^1(k,H^3(X_{\ov{k}},\Z_\ell(2)))
\]
is also surjective. Hence, by \Cref{rmk:AJ_properties}(2), for all but finitely many primes $\ell$ and every finite extension $k/k_0$, the map
\[
\alpha^2_{X_k,\ell}\colon
Z^2_{\mathrm{hom},\ell}(X_k)
\longrightarrow
H^1(k,H^3(X_{\ov{k}},\Z_\ell(2)))
\]
is surjective, as desired.
\end{proof}

\begin{thm}\label{main-conditional-theorem}
Let $k_0$ be a finite field of characteristic $p>2$, and assume that \Cref{tate-conjecture} holds over $k_0$ for some prime $\ell_0$. Let $X/k_0$ be a smooth projective geometrically integral variety of dimension $d$. For all but finitely many primes $\ell\neq p$, for every finite extension $k/k_0$, the group $H^3_{\mathrm{nr}}(k(X)/k,\Q_\ell/\Z_\ell(2))$ is divisible and the cycle maps
\[
CH^{2}(X_k)_{\Z_\ell}\longrightarrow H^{4}(X_k,\Z_\ell(2)),
\qquad
CH^{d-1}(X_k)_{\Z_\ell}
\longrightarrow H^{2d-2}(X_k,\Z_\ell(d-1))
\]
are surjective.
\end{thm}

\begin{proof}
When $d=1$, we have $H^{4}(X_k,\Z_\ell(2))=0$, the cycle map is an isomorphism in codimension $0$, and $H^3(k(X),\Q_\ell/\Z_\ell(2))=0$ because $\mathrm{cd}_\ell(k(X))=2$. From now on, we assume $d\geq 2$.

We first claim that for all but finitely many primes $\ell\neq p$, for every finite extension $k/k_0$, the maps
\begin{equation}\label{eq:geometric-map}
CH^i(X_k)_{\Z_\ell}\longrightarrow H^{2i}(X_{\ov{k}},\Z_\ell(i))^{G_k}\qquad (i=2,d-1).
\end{equation}
are surjective.

By \cite[Theorem~2.9]{tate1994conjectures}, \Cref{tate-conjecture} implies that homological and numerical equivalence coincide for cycles with rational coefficients; see \cite[\S 2.6]{tate1994conjectures} for the definition of numerical equivalence. Let $k/k_0$ be a finite extension. For $j=1,2, d-2,d-1$, let $N^j(X_k)$ denote the $\Q$-vector space of codimension-$j$-cycles on $X_k$ modulo numerical equivalence.
Choose finitely many cycles defined over $k$ whose numerical classes form $\Q$-bases of these four spaces.
By \Cref{lem:tate-independence}, for every $\ell\neq p$, the cycle map induces an isomorphism
\[
N^j(X_k)_\Q\otimes_{\Q}\Q_\ell\xlongrightarrow{\sim} H^{2j}(X_{\ov{k}},\Q_\ell(j))^{G_k}.
\]
Moreover, the intersection pairings
\[
N^2(X_k)_\Q\times N^{d-2}(X_k)_\Q\longrightarrow\Q,\qquad N^{d-1}(X_k)_\Q\times N^1(X_k)_\Q\longrightarrow \Q
\]
are non-degenerate by the definition of numerical equivalence.

It follows that the determinants of their matrices in the chosen bases are non-zero rational numbers. The entries of these matrices are intersection numbers of the chosen cycles, and hence the non-zero determinants are independent of $\ell$. Let $D_k$ be a positive integer such that every numerator and denominator of these finitely many determinants divides $D_k$. By a theorem of Gabber \cite{gabber1983torsion}, integral $\ell$-adic cohomology of $X_{\ov{k}}$ is torsion-free for all but finitely many primes $\ell$. For every prime $\ell\nmid D_k$ for which moreover the $\ell$-adic integral cohomology groups are torsion-free, the cokernel of each map in  \eqref{eq:geometric-map} is torsion-free. Indeed, if $\ell x$ belongs to the image of the cycle map, pairing with the chosen basis in complementary codimension and using that the corresponding intersection matrix is invertible over $\Z_\ell$ shows that $x$ itself belongs to the image. Therefore, under \Cref{tate-conjecture}(2), the maps \eqref{eq:geometric-map} are surjective. 
As $k/k_0$ runs over all finite extensions, only finitely many subspaces $N^j(X_k)\subset N^j(X_{\ov{k}_0})$ occur. We fix bases for each of these finitely many possibilities, and define $D_k$ using these bases. Then only finitely many values of $D_k$ occur.
Hence for all but finitely many primes $\ell\neq p$, for every finite extension $k/k_0$, \eqref{eq:geometric-map} is surjective.

By the Hochschild--Serre spectral sequence, we deduce from the above claim and \Cref{tate-implies-aj-surjective} that for all but finitely many primes $\ell\neq p$, for every finite extension $k/k_0$, the cycle maps
\[
CH^i(X_k)_{\Z_\ell}\longrightarrow H^{2i}(X_k,\Z_\ell(i)),\qquad (i=2,d-1)
\]
are surjective.
Finally, by \cite[Th\'eor\`eme 2.2]{colliot2013cycles}, we get
\[
H^3_{\nr}(k(X)/k,\Q_\ell/\Z_\ell(2))/H^3_{\nr}(k(X)/k,\Q_\ell/\Z_\ell(2))_{\on{div}}=0,
\]
as desired.
\end{proof}

\section{Self-fiber-products of semistable elliptic surfaces}\label{sec:5}

\subsection{The projector and complex multiplication cycles}

Let $k_0$ be a finite field of characteristic $p$, let
$\pi\colon Y\to B$ be a non-isotrivial relatively minimal semistable elliptic
surface with a section, and let $j\colon\dot B\hookrightarrow B$ be its smooth
locus. Let
\[
 \sigma\colon W\longrightarrow Y\times_B Y
\]
be the blow-up of the reduced singular locus, and let
$f\colon W\to B$ be the composite $W\xrightarrow{\sigma}Y\times_{B}Y\to B$. Let $\dot \pi\colon \dot Y\to \dot B$ and $\dot f\colon \dot W\to \dot B$ be the base changes of $\pi$ and $f$ along $j$ respectively. Write $m_\pi$ for the least common multiple of the integers $m$ such that there exists a singular fiber of $\pi$ of Kodaira type $I_m$. For a prime $\ell\nmid 2pm_\pi$, set
\begin{equation}\label{eq:e-m}
 \mathcal E=j_*(\operatorname{Sym}^2R^1\dot\pi_*\Z_\ell)(1),
 \qquad
 \mathcal M=\mathcal E_1(1)
 =j_*(\operatorname{Sym}^2R^1\dot\pi_*(\Z/\ell))(2).
\end{equation}

\begin{lem}\label{11.2.5}
There is a cycle
$P\in Z^3(W\times W)\otimes\Z[1/2]$ with the following properties.
\begin{enumerate}
\item[(i)] Its class in $CH^3(W\times W)\otimes\Z[1/2]$ is idempotent.
\item[(ii)] On $\dot W\times_{\dot B}\dot W$ it induces an idempotent relative
correspondence.
\item[(iii)] If $\ell\nmid2p$, then
\[
P_*R^2\dot f_*\Z_\ell(2)\simeq
(\operatorname{Sym}^2R^1\dot\pi_*\Z_\ell)(2).
\]
\item[(iv)] If $\ell\nmid2pm_\pi$, then
$H^0(B_{\bar k},j_*\operatorname{Sym}^2R^1\dot\pi_*(\Z/\ell^n))=0$
for every $n\geq1$.
\item[(v)] Under the same assumption,
\[
P_*H^3(W_{\bar k},\Z_\ell(2))\simeq
H^1(B_{\bar k},\mathcal E(1)),
\]
and this module is torsion-free.
\item[(vi)] The correspondence $P$ fixes every CM cycle.
\item[(vii)] The $\ell$-adic Abel--Jacobi map induces a surjection
\[
(1-P)_*CH^2_{\mathrm{alg}}(W_{\bar k})\longrightarrow
(1-P)_*J^2_\ell(W),
\]
where
$J^2_\ell(W)=\varinjlim_{k'/k_0}H^1(k',H^3(W_{\bar k},\Z_\ell(2))/\mathrm{tors})$.
\end{enumerate}
\end{lem}

\begin{proof}
These are exactly the seven assertions of
\cite[Lemma~(11.2.5)]{schoen1999image}. In particular, (v) has no residual
$\ell^n$ on its right-hand side, and (vii), rather than a statement about the
diagonal alone, is the complementary-projector assertion proved in
\cite[Corollary~10.8]{schoen1995computation}.
\end{proof}

A $k'$-point $x\in\dot B(k')$ is a \emph{CM point} if the elliptic
curve $\pi^{-1}(x)_{\bar k}$ over $\bar k$ is ordinary. If $x$ is a CM point, then
$\operatorname{NS}(f^{-1}(x)_{\bar k})$ has rank four. A \emph{CM cycle} is
a divisor on $f^{-1}(x)$ whose N\'eron--Severi class generates the orthogonal
complement of
\[
 \pi^{-1}(x)\times s(x),\qquad s(x)\times\pi^{-1}(x),\qquad\Delta_x.
\]
These are \cite[Definitions~(11.2.1) and (11.2.3)]{schoen1999image}.
For $\ell\nmid2pm_\pi$, a CM cycle is homologically trivial on $W$, and its
class in
$\operatorname{Sym}^2H^1(\pi^{-1}(x)_{\bar k},\Z/\ell)(1)\simeq H^2_{\ov{x}}(B_{\ov{k}},\mathcal M)$ is non-zero;
see \cite[Lemma~(11.2.6)(i),(ii)]{schoen1999image}.

\subsection{The auxiliary threefold}\label{sec:W}

Assume $p\neq 2,3$. Let
$Y\subset\P^1_{\F_p}\times\P^2_{\F_p}$ be given by
\[
u_0(x_0^3+x_1^3+x_2^3)-u_1x_0x_1x_2=0,
\]
and let $\pi\colon Y\to B=\P^1_{\F_p}$ be the first projection. It has the
section $(0:-1:1)$ and four geometric singular fibers, all of type $I_3$,
over
\[
 Z_{\bar{\F}_p}=\{(0:1)\}\cup
 \{(1:3\zeta):\zeta\in\mu_3(\bar{\F}_p)\}.
\]
Thus $m_\pi=3$. We will sometimes refer to the closed points of $Z_{\bar{\F}_p}$ as the cusps of the fibration. 

Let $W$ be the blow-up of the reduced singular locus of
$Y\times_B Y$. Then $W$ is a smooth projective threefold over $\F_p$.

\begin{thm}\label{w-aj-surjective}
Let $\ell\nmid 6p$ be a prime. Then, for every finite extension $k/\F_p$, the Abel--Jacobi map
\[
\alpha^2_{W_k,\ell}\colon CH^2_{\mathrm{hom},\ell}(W_k)
\longrightarrow H^1(k,H^3(W_{\bar k},\Z_\ell(2)))
\]
is surjective.
\end{thm}

\begin{proof}
Let $\mathcal M$ be as in \eqref{eq:e-m}, let $\rho\colon C\to B=\P^1_{\F_p}$ be its arithmetic monodromy cover, and write $\Gamma_{\mathrm{arith},\F_p}$ and $\Gamma_{\mathrm{geom}}$ for its arithmetic and geometric monodromy groups, respectively, as in \Cref{sec:connecting-map}. It is proved in \cite[Lemma~(11.2.6)(iii),(iv)]{schoen1999image} that
\[
\Gamma_{\mathrm{geom}}\simeq\operatorname{SL}_2(\F_\ell)/\{\pm1\},
\]
and the representation is tamely ramified.

Let $\ddot B\subset \dot B=B\setminus Z$ be the non-empty open subset parametrizing ordinary fibers of $\pi\colon Y\to B$, i.e. the complement in $\dot B$ of the supersingular points; see \cite[Lemma~(11.2.2)(i)]{schoen1999image}. In the notation of Assumption~\ref{schoen-situation-projective}, we take $U=\ddot B$. By \Cref{3.2.1-3.2.6}, after replacing $\F_p$ by a finite extension $k_1$, we may assume that Assumptions \ref{schoen-situation-projective}(1)--\ref{schoen-situation-projective}(6) hold over every finite extension of $k_1$. Assumption \ref{schoen-situation-projective}(7) was verified above, and Assumptions \ref{schoen-situation-projective}(8)--\ref{schoen-situation-projective}(10) follow from \cite[Lemma~(11.2.7)]{schoen1999image}.

As usual, let $M$ be the Galois module corresponding to the stalk $\mathcal{M}_{\bar{\eta}}=\dot{\mathcal E}_1(1)_{\bar\eta}$, so that $M^\vee(1)$ corresponds to the stalk $(\mathcal E_1)^\vee_{\bar\eta}$. By \cite[Lemma~(11.2.6)(iii)]{schoen1999image}, the geometric monodromy group is
\[\Gamma\simeq\operatorname{PSL}_2(\F_\ell).\]
Let $b$ be one of the four cusps. By the local monodromy calculation in the proof of \cite[Lemma~(11.2.6)]{schoen1999image}, the image $I_b\subset\Gamma$ of the inertia subgroup at $b$ is generated by a nontrivial unipotent element. Hence $|I_b|=\ell$. Since $|\operatorname{PSL}_2(\F_\ell)|=\frac{\ell(\ell^2-1)}{2}$, we deduce that $[\Gamma:I_b]$ is prime to $\ell$. A restriction-corestriction argument now shows that $H^1(\Gamma,M^\vee(1))\to H^1(I_b,M^\vee(1))$ is injective. Therefore 
\[H^1(\Gamma,M^\vee(1))\longrightarrow\prod_{b\in B_{\ov{k}}\setminus U_{\ov{k}}} H^1(I_b,M^\vee(1))\]
is injective, and hence $\mathcal K_{M^\vee(1)}=0$.

Since the natural homomorphism $\pi_1(\ddot B_{\bar k})\to\pi_1(\dot B_{\bar k})$ is surjective, the image of the homomorphism $\pi_1(\ddot B_{\bar k})\to\operatorname{GL}(M)$ is also equal to $\Gamma_{\mathrm{geom}}$. Therefore $H^0(\ddot B_{\bar k},\mathcal E_1^\vee)\simeq(M^\vee(1))^{\Gamma_{\mathrm{geom}}}$. Since the cyclotomic character is trivial on geometric monodromy, Assumption \ref{schoen-situation-projective}(8) gives $(M^\vee(1))^{\Gamma_{\mathrm{geom}}}=(M^\vee)^{\Gamma_{\mathrm{geom}}}=0$. Thus $H^0(\ddot B_{\bar k},\mathcal E_1^\vee)=0$.

Let $K/k_1$ be a finite extension, and let $\epsilon_1,\ldots,\epsilon_t\in H^1(K,H^1(B_{\bar k},\mathcal M))$ be an $\F_\ell$-basis. By \Cref{schoen-unified-cor}, for every $i=1,\dots,t$ there exist $x_i\in U(K)=\ddot B(K)$ and $\bar e_i\in H^2_{\bar x_i}(B_{\bar k},\mathcal M)_0^{G_K}$ such that $\delta_{\mathcal M,x_i}(\bar e_i)=\epsilon_i$ and $\kappa_K(\operatorname{Frob}_{x_i})$ is conjugate to $\xi$. By \Cref{schoen-unified}(3), the space $H^2_{\bar x_i}(B_{\bar k},\mathcal M)_0^{G_K}$ is one-dimensional.

Let $\mathcal E$ be as in \eqref{eq:e-m}, so that $\mathcal E_1(1)=\mathcal M$, and let $q\colon R^2f_*\Z_\ell(1)\to\mathcal E$ be the projection induced by $P$. Since $x_i\in\ddot B(K)$, it is a CM point, and hence we may choose a CM cycle $z_i$ supported on $f^{-1}(x_i)$. By \Cref{11.2.5}(vi) and \cite[Lemma~(11.2.6)(i),(ii)]{schoen1999image}, the cycle $z_i$ is homologically trivial on $W_K$ and the reduction of $q_\#(z_i)$ is non-zero in $H^2_{\bar x_i}(B_{\bar k},\mathcal M)_0^{G_K}$. Since this space is one-dimensional, after multiplying $z_i$ by an integer prime to $\ell$, we may assume that the reduction of $q_\#(z_i)$ is $\bar e_i$.

The hypotheses of \Cref{schoen-2.4.3-finer} are satisfied for $\mathcal E$ and the open $U=\ddot B$: the restrictions $\mathcal E_n|_U$ are locally constant, the sheaves $\mathcal E_n$ are flat over $\Z/\ell^n$, and we have shown above that $H^0(U_{\bar k},\mathcal E_1^\vee)=0$. Hence \Cref{schoen-2.4.3-finer} shows that the classes $\alpha_{\mathcal E}(q_\#(z_i))$ generate $H^1(K,H^1(B_{\bar k},\mathcal E(1)))$.

Since $q$ is the projection induced by $P$, it follows from \Cref{lem:comp_f_E,relative-aj-compatibility} and \Cref{11.2.5}(v) that, under the identification
\[P_*H^1(K,H^3(W_{\bar k},\Z_\ell(2)))\simeq H^1(K,H^1(B_{\bar k},\mathcal E(1))),\]
the class $P_*\alpha^2_{W_K,\ell}(z_i)$ corresponds to $\alpha_{\mathcal E}(q_\#(z_i))$. It follows that the composite
\[CH^2_{\mathrm{hom},\ell}(W_K)\xlongrightarrow{\alpha^2_{W_K,\ell}}H^1(K,H^3(W_{\bar k},\Z_\ell(2)))\xlongrightarrow{P_*}P_*H^1(K,H^3(W_{\bar k},\Z_\ell(2)))\]
is surjective.

The correspondence $P$ acts as the identity on $H^3(W_{\bar k},\Z_\ell(2))$. This follows from \cite[Lemma~(10.3)(iii)]{schoen2002complex} in characteristic zero, but then extends to characteristic prime to $6\ell$ by a spreading-out argument and the compatibility of smooth proper base change with correspondences. Thus, for every finite extension $K/k_1$, we have
\[P_*H^1(K,H^3(W_{\bar k},\Z_\ell(2)))=H^1(K,H^3(W_{\bar k},\Z_\ell(2))),\]
and hence $\alpha^2_{W_K,\ell}$ is surjective.

Finally, let $K$ be the compositum of $k$ and $k_1$. Then $K/k_1$ is finite, so the Abel--Jacobi map is surjective over $K$ by the preceding paragraph. By \Cref{lem:Abel--Jacobi-corestriction}, surjectivity over $K$ descends to $k$. This proves the result.
\end{proof}

\section{Proof of Theorem \ref{main-fermat-cubic-thm}}\label{sec:6}

\subsection{The supersingular case}

We begin with two consequences of Shioda's theorem for the cohomology of supersingular abelian varieties.

\begin{lem}\label{lem:supersingular-divisors}
Let $k$ be the algebraic closure of a finite field of characteristic $p$, let $X/k$ be a supersingular abelian variety, and let $\ell\ne p$ be a prime. Then the cycle map
\[CH^1(X)_{\Z_\ell} \longrightarrow H^2(X,\Z_\ell(1))\]
is surjective.
\end{lem}

\begin{proof}
Shioda's theorem \cite[Appendix]{shioda1975algebraic} gives $\operatorname{rank}\operatorname{NS}(X)=b_2(X)$. The Kummer exact sequence gives
\[
0\longrightarrow
\operatorname{NS}(X)_{\Z_\ell}
\longrightarrow
H^2(X,\Z_\ell(1))
\longrightarrow
T_\ell\operatorname{Br}(X)
\longrightarrow0.
\]
The last term is torsion-free, while the rank equality above forces its
rank to be zero. Hence it vanishes, and the result follows.
\end{proof}

\begin{lem}\label{lem:supersingular-gysin}
Let $k$ be the algebraic closure of a finite field of characteristic $p$,
let $X/k$ be a supersingular abelian variety of dimension $g$, and let
$\ell\ne p$ be a prime. For every $1\leq r\leq g$, the $\Z_\ell$-module $H^{2r-1}(X,\Z_\ell(r))$ is generated by the images of Gysin maps
\[
i_*\colon
H^1(Y,\Z_\ell(1))
\longrightarrow
H^{2r-1}(X,\Z_\ell(r)),
\]
where $i\colon Y\hookrightarrow X$ ranges over smooth complete
intersections of codimension $r-1$.
\end{lem}

\begin{proof}
The case $r=1$ is immediate. Assume $2\leq r\leq g$.
By \Cref{lem:supersingular-divisors}, divisor classes generate $H^2(X,\Z_\ell(1))$. 
Since $H^*(X,\Z_\ell)$ is an exterior algebra on $\bigwedge H^1(X,\Z_\ell)$, the group $H^{2r-1}(X,\Z_\ell(r))$ is generated by classes of the form
\[
a\cup D_1\cup\cdots\cup D_{r-1},
\]
where $a\in H^1(X,\Z_\ell(1))$ and $D_i\in H^2(X,\Z_\ell(1))$ are divisor classes.

Write each $D_i$ as a difference of sufficiently ample divisor classes
and expand the products. By Bertini's theorem, the resulting
intersections may be represented by smooth complete intersections $i\colon Y\hookrightarrow X$ of codimension $r-1$. By the projection formula, $i_*i^*a=a\cup[Y]$.
Thus every generator above belongs to the image of a Gysin map $i_*$, as desired.
\end{proof}

\begin{thm}\label{thm:supersingular-Abel--Jacobi}
Let $k$ be a finite field of characteristic $p$, let $\ell\ne p$ be a prime, and let $X/k$ be a supersingular abelian variety. Then, for every integer $r\geq 0$, the Abel--Jacobi map
\[
\alpha^r_{X,\ell}\colon
Z^r_{\mathrm{hom},\ell}(X)
\longrightarrow
H^1(k,H^{2r-1}(X_{\bar k},\Z_\ell(r)))
\]
is surjective.
\end{thm}

\begin{proof}
The cases $r=0$ and $r>\dim X$ are immediate, while the case $r=1$ follows from \Cref{rmk:AJ_properties}(7). 

Assume $2\leq r\leq\dim X$. By \Cref{lem:supersingular-gysin} and \cite[Lemma~2.12]{scavia2023coniveau}, there exist smooth projective
geometrically connected curves $C_j$ over $k$ and correspondences $\Gamma_j\in CH^r(C_j\times X)$ such that the map
\[
\bigoplus_j H^1(C_{j,\bar k},\Z_\ell(1))
\xlongrightarrow{\ \sum_j(\Gamma_j)_*\ }
H^{2r-1}(X_{\bar k},\Z_\ell(r))
\]
is surjective. Since $k$ is a finite field, this implies that the induced map
\[
\bigoplus_j
H^1(k,H^1(C_{j,\bar k},\Z_\ell(1)))
\longrightarrow
H^1(k,H^{2r-1}(X_{\bar k},\Z_\ell(r)))
\]
is surjective. The codimension-one Abel--Jacobi maps of the curves
$C_j$ are surjective by \Cref{rmk:AJ_properties}(7), and compatibility
of Abel--Jacobi maps with correspondences
\Cref{rmk:AJ_properties}(2) proves the result.
\end{proof}

\begin{thm}\label{thm:supersingular-all-codimension}
Let $X/k_0$ be a supersingular abelian variety of dimension $g$ over a
finite field of characteristic $p$. There exists a finite extension
$k_1/k_0$, depending only on $X$, such that, for every finite extension
$k/k_1$, every prime $\ell\ne p$, and every $0\leq r\leq g$, the cycle
map
\[
CH^r(X_k)_{\Z_\ell}
\longrightarrow
H^{2r}(X_k,\Z_\ell(r))
\]
is surjective.
\end{thm}

\begin{proof}
Choose divisors $D_1,\ldots,D_s$ whose classes form a $\Z$-basis of $\operatorname{NS}(X_{\bar k_0})$, and choose a finite extension $k_1/k_0$ over which all of them are defined. Observe that the extension $k_1/k_0$ is chosen independently of
$r$ and $\ell$. By \Cref{lem:supersingular-divisors}, for every prime $\ell\ne p$ their
classes generate $H^2(X_{\bar k_0},\Z_\ell(1))$. Set $V\coloneqq H^1(X_{\bar k_0},\Z_\ell)$. Since $H^*(X_{\bar k_0},\Z_\ell)$ is an exterior algebra on $V$, the cup-product map
\[
H^2(X_{\bar k_0},\Z_\ell(1))^{\otimes r}
\longrightarrow
H^{2r}(X_{\bar k_0},\Z_\ell(r))
\]
is surjective for every $r$. Indeed, if
$e_1,\ldots,e_{2g}$ is a basis of $V$, every basis monomial of degree
$2r$ can be written as
\[
e_{i_1}\wedge\cdots\wedge e_{i_{2r}}
=
(e_{i_1}\wedge e_{i_2})\cup\cdots\cup
(e_{i_{2r-1}}\wedge e_{i_{2r}}).
\]

It follows that, for every finite extension $k/k_1$, products of the
divisors $D_i$ give a surjection
\[
CH^r(X_k)_{\Z_\ell}
\longrightarrow
H^{2r}(X_{\bar k},\Z_\ell(r)).
\]
In particular, since $G_k$ acts trivially on $CH^r(X_k)$, it acts trivially on $H^{2r}(X_{\bar k},\Z_\ell(r))$.

The Hochschild--Serre spectral sequence gives an exact sequence
\[
0\longrightarrow
H^1(k,H^{2r-1}(X_{\bar k},\Z_\ell(r)))
\longrightarrow
H^{2r}(X_k,\Z_\ell(r))
\longrightarrow
H^{2r}(X_{\bar k},\Z_\ell(r))^{G_k}
\longrightarrow0.
\]
The right-hand term is contained in the image of the cycle map by the
preceding paragraph, while the left-hand term is contained in its image
by \Cref{thm:supersingular-Abel--Jacobi} and
\Cref{rmk:AJ_properties}(1). Hence the cycle map
\[
CH^r(X_k)_{\Z_\ell}
\longrightarrow
H^{2r}(X_k,\Z_\ell(r))
\]
is surjective. 
\end{proof}

\begin{cor}\label{cor:supersingular-threefold}
Let $k$ be a finite field of characteristic $p$, let $\ell\ne p$ be a
prime, and let $X/k$ be a supersingular abelian threefold which is a
product of Jacobians of smooth projective geometrically connected
curves over $k$. Then
\[
CH^2(X)_{\Z_\ell}
\longrightarrow
H^4(X,\Z_\ell(2))
\]
is surjective and
\[
H^3_{\nr}(k(X)/k,\Q_\ell/\Z_\ell(2))=0.
\]
\end{cor}

\begin{proof}
By \cite[Theorem~1.2]{scavia2025direct}, the cycle map
\[
CH^2(X)_{\Z_\ell}
\longrightarrow
H^4(X_{\bar k},\Z_\ell(2))^{G_k}
\]
is surjective. On the other hand,
\Cref{thm:supersingular-Abel--Jacobi} gives the surjectivity of
\[
\alpha^2_{X,\ell}\colon
Z^2_{\mathrm{hom},\ell}(X)
\longrightarrow
H^1(k,H^3(X_{\bar k},\Z_\ell(2))).
\]
Using \Cref{rmk:AJ_properties}(1) and the Hochschild--Serre exact
sequence, we conclude that the cycle map
\[
CH^2(X)_{\Z_\ell}
\longrightarrow
H^4(X,\Z_\ell(2))
\]
is surjective. 
By \cite[Th\'eor\`eme~2.2]{colliot2013cycles}, this implies that $H^3_{\nr}(k(X)/k,\Q_\ell/\Z_\ell(2))$ is divisible. Since $X$ is an abelian variety,
it belongs to the class $B(k)$ of
\cite[\S3.4]{colliot2013cycles}. As $\dim X=3$,
\cite[Proposition~3.15(b)(ii)]{colliot2013cycles} implies that
$X\in B_{\mathrm{Tate}}(k)$; see also
\cite[Property~3.14(vi)]{colliot2013cycles}. Therefore
\cite[Th\'eor\`eme~3.18(c)]{colliot2013cycles} shows that $H^3_{\nr}(k(X)/k,\Q_\ell/\Z_\ell(2))$ is finite. Since a finite divisible group is zero, the result follows.
\end{proof}

\subsection{The Fermat cubic curve}

Let $k$ be a field of characteristic different from $3$, and consider the projective plane $\P^2_k$ with homogeneous coordinates $x,y,z$. The \emph{Fermat cubic curve} over $k$ is the smooth projective cubic curve $E\subset \P^2_k$ given by
\begin{equation}\label{eq:fermat}E\coloneqq \{x^3+y^3+z^3=0\}\subset \P^2_k.\end{equation} We will always consider $E$ as an elliptic curve with origin $O\coloneqq (1:-1:0)\in E(k)$.

The group $\mu_3$ acts on $E$ by
\begin{equation}\label{eq:mu_3-action}\mu_3\times_k E\to E,\qquad \omega\cdot (x:y:z)=(x:y:\omega z).\end{equation}

\begin{lem}\label{lem:fermat-schoen-isogenous}
Let $k$ be a field of characteristic different from $2$ and $3$. Let
$E/k$ be the Fermat cubic of \eqref{eq:fermat}, and let
\[
E'\colon (x')^3-(y')^2z'+(z')^3/4=0
\]
be the elliptic curve considered by Schoen \cite[(13.1)]{schoen1995computation}, with origin $O'=(0:1:0)$. There is a $3$-isogeny $E'\to E$ which is obtained by base change from a $3$-isogeny between the elliptic curves over $\Z[1/6]$ given by the same equations as $E'$ and $E$. Its dual $E\to E'$ has the same property. 

Moreover, if $\zeta_3\in k$, these $3$-isogenies are equivariant for the order-three automorphisms given, on the corresponding Weierstrass models, by multiplication of the $x$-coordinate by $\zeta_3$.
\end{lem}

\begin{proof}
We first recall a Weierstrass model of the Fermat cubic. Let
\[
E_W\colon v^2w=u^3-432w^3.
\]
Since $\operatorname{char}(k)\neq 2,3$, the change of coordinates
\[
[u:v:w]=[-12z:36(x-y):x+y]
\]
is defined over $k$ and induces a $k$-isomorphism $E\xrightarrow{\sim} E_W$.
On the affine chart $z'\neq 0$, the curve $E'$ is
\[
(y')^2=(x')^3+1/4.
\]
We have a $3$-isogeny
\[
E'\longrightarrow E'',\qquad
E''\colon \eta^2=\xi^3-27/4,
\]
given by
\[
\xi=\frac{(x')^3+1}{(x')^2},\qquad
\eta=\frac{y'((x')^3-2)}{(x')^3}.
\]
Its kernel is the order-$3$ subgroup
\[
\{O',\,(0:1/2:1),\,(0:-1/2:1)\}\subset E'(k).
\]
Finally, the change of variables $u=4\xi$, $v=8\eta$ identifies $E''$ with $E_W$.
Composing
\[
E'\longrightarrow E''\xlongrightarrow{\sim}E_W\xlongleftarrow{\sim}E
\]
gives a $k$-rational $3$-isogeny $E'\to E$. The dual isogeny gives a
$k$-rational $3$-isogeny $E\to E'$. The displayed formula is equivariant for
$x'\mapsto\zeta_3x'$ and $\xi\mapsto\zeta_3\xi$; under the Weierstrass
change of coordinates, the action \eqref{eq:mu_3-action} is
$u\mapsto\zeta_3u$. Hence the isogeny, and therefore also its dual, is
$C_3$-equivariant.
\end{proof}

\begin{prop}\label{prop:fermat-facts}
    Let $k$ be a field of characteristic different from $2$ and $3$, and let $E$ be the Fermat cubic curve over $k$ defined in (\ref{eq:fermat}), with origin $O=(1:-1:0)$.

    \begin{enumerate}
     \item The Weierstrass form of $E$ is given by $v^2=u^3-2^4\cdot 3^3$.
    \item The $j$-invariant of $E$ is equal to $0$.
    \item Suppose that $\mathrm{char}(k)=p>0$. Then $E$ is ordinary if $p\equiv 1\pmod 3$, and is supersingular if $p\equiv 2\pmod 3$.
     \item If $k$ contains a primitive third root of unity $\zeta_3$ and $E$ is ordinary, then the $\mu_3$-action of (\ref{eq:mu_3-action}) induces an isomorphism \[\Z[\zeta_3]\xlongrightarrow{\sim}\mathrm{End}_{\ov{k}}(E),\] and every endomorphism is defined over $k$.

     \item Let $\ell$ be a prime not equal to $3$ or $\on{char}(k)$, let 
     $R=\Z_\ell[\zeta_3]$, let $C_3$ be a cyclic group of order $3$, and let $\omega\in C_3$ be a generator.  Let $L$
     (respectively $\overline L$) be the rank-one $R[C_3]$-module on which
     $\omega$ acts by $\zeta_3$ (respectively $\zeta_3^2$). Then
     \[
     H^1(E_{\ov{k}},\Z_\ell)\otimes_{\Z_\ell}R
     \simeq L\oplus\overline L
     \]
     as $R[C_3]$-modules.
    \item If $\mathrm{char}(k)\neq 2$, the curve $E$ is $3$-isogenous to the elliptic curve of equation $x^3-y^2z+z^3/4=0$ of \cite[(13.1)]{schoen1995computation}.
    \end{enumerate}
\end{prop}

\begin{proof}
(1) The change of homogeneous coordinates
\[
[u:v:w]=[-12z:36(x-y):x+y]
\]
identifies $E$ with the Weierstrass cubic $v^2w=u^3-432w^3$, and hence on the affine chart $w=1$ with $v^2=u^3-2^4\cdot3^3$.

(2) This follows from (1) and the formula for the $j$-invariant \cite[\S III.1, p.~42]{silverman2009arithmetic}.

(3) This is \cite[Example~V.4.4]{silverman2009arithmetic}; equivalently it follows from the complex multiplication criterion in \cite[Chapter~13, Theorem~12]{lang1987elliptic}.

(4) Since $E$ is ordinary, by \cite[Theorem~V.3.1]{silverman2009arithmetic} the ring $\on{End}_{\ov{k}}(E)$ is an order in an imaginary quadratic field. The automorphism in \eqref{eq:mu_3-action} has order $3$ and gives an injective ring homomorphism $\Z[\zeta_3]\hookrightarrow \on{End}_{\ov{k}}(E)$. We deduce the equality $\on{End}_{\ov{k}}(E)\otimes_{\Z}\Q=\Q(\zeta_3)$. Since $\Z[\zeta_3]$ is the maximal order of $\Q(\zeta_3)$, we obtain $\on{End}_{\ov{k}}(E)=\Z[\zeta_3]$. Since by assumption, $\zeta_3\in k$, the automorphism in \eqref{eq:mu_3-action} is defined over $k$. As this automorphism generates $\Z[\zeta_3]$ as a ring, every geometric endomorphism of $E$ is defined over $k$.

(5) Since $\ell\ne3$, the difference
$\zeta_3-\zeta_3^2$ is a unit of $R=\Z_\ell[\zeta_3]$. The two
idempotents associated with the roots of $T^2+T+1$ therefore split
$H^1(E_{\ov{k}},\Z_\ell)\otimes_{\Z_\ell}R$ into the rank-one character
modules $L$ and $\overline L$.

(6) This is \Cref{lem:fermat-schoen-isogenous}.
\end{proof}

\begin{lem}\label{medium-tate}
Let $E$ be an elliptic curve over a finite field $k$ and $\ell$ be a prime number invertible in $k$.
Then $CH^2(E^3)_{\Z_\ell}\rightarrow H^{4}(E^3_{\ov{k}},\Z_\ell(2))^{G_k}$ is surjective.
\end{lem}
\begin{proof}
This is \cite[Theorem~1.2]{scavia2025direct}.
\end{proof}

\subsection{Proof of Theorem \ref{main-fermat-cubic-thm}}

Let $\Delta\subset (\mu_3)^3\subset \mathrm{Aut}(E_{\ov{\mathbb{F}}_p})$ be the diagonal subgroup of $(\mu_3)^3$, where each factor $\mu_3$ acts on the corresponding copy of $E_{\ov{\mathbb{F}}_p}$ as in (\ref{eq:mu_3-action}). For every $\delta\in \Delta$, let $\Gamma_\delta$ be the graph of $\delta$, and consider the correspondences $Q\coloneqq\sum_{\delta\in \Delta}\Gamma_\delta$ and $Q'\coloneqq 3I-Q$. We have
\[Q\circ Q=3Q,\qquad Q'\circ Q'=3Q',\qquad Q\circ Q'=0=Q'\circ Q.\]
Thus $Q$ and $Q'$ are not idempotents. Whenever $3$ is invertible in the coefficient ring, the correspondences $e=Q/3$ and $e'=Q'/3$ are complementary idempotents. We use $Q_*A$ and $Q'_*A$ only to denote the images of the integral correspondences on a group $A$ (this is not necessarily an integral direct-sum decomposition of the Chow group).

Following \cite[\S13]{schoen1995computation}, let $X(3)$ denote the compactified modular curve over $\Q$ associated with the level-$3$ moduli problem considered there, let $V(3)\to X(3)$ be the corresponding projective relatively minimal elliptic surface, and let $W(3)$ be the blow-up of the reduced singular locus of $V(3)\times_{X(3)}V(3)$. 
We write $X(3)^\circ\subset X(3)$ for the open modular curve over which $V(3)$ is smooth, and $V(3)^\circ\longrightarrow X(3)^\circ$ for the corresponding universal elliptic curve.

\begin{lem}\label{lem:hesse-V3}
Let $B_0=\P^1_{\Q}$, and let $Y_0\subset\P^1_{\Q}\times\P^2_{\Q}$ be the surface defined by
\[
u_0(x_0^3+x_1^3+x_2^3)-u_1x_0x_1x_2=0.
\]
Let $\pi_0\colon Y_0\to B_0$ be the first projection, and let $W_0$ be the blow-up of the reduced singular locus of $Y_0\times_{B_0}Y_0$. Then there is a commutative diagram
\[
\begin{tikzcd}
Y_0 \arrow[r,"\sim"] \arrow[d,"\pi_0"'] & V(3) \arrow[d] \\
B_0 \arrow[r,"\sim"] & X(3)
\end{tikzcd}
\]
over $\Q$. This induces an isomorphism
\[
Y_0\times_{B_0}Y_0
\xlongrightarrow{\sim}
V(3)\times_{X(3)}V(3),
\]
and hence, after blow-up of the reduced singular locus of $Y_0\times_{B_0}Y_0$, an isomorphism
\[
W_0\xlongrightarrow{\sim}W(3).
\]
\end{lem}

\begin{rem}\label{rmk:reduce-mod-p}
    By construction, for every prime $p\neq2,3$, the varieties $Y$ and $W$ of \S\ref{sec:W} are obtained by reducing the equation of \Cref{lem:hesse-V3} modulo $p$ and performing the same self-fiber-product and blow-up construction.
\end{rem}

\begin{proof}[Proof of \Cref{lem:hesse-V3}]
Let $U\subset B_0$ be the open subset over which $\pi_0$ is smooth.
Recall that \cite[\S13]{schoen1995computation} considers the moduli
problem over $\Q$ which associates to a $\Q$-scheme $S$ the
isomorphism classes of elliptic curves over $S$ equipped with a point
of order $3$ and a disjoint subgroup scheme isomorphic to $\mu_3$. The restriction $\pi_{0,U}\colon (Y_0)_U\to U$ is a family of elliptic curves with zero section $O=(1:-1:0)$. The nine base points of the pencil are flex points of every smooth member and hence are $3$-torsion with respect to the group law with origin $O$. In particular, $P=(0:1:-1)$ has order $3$. Moreover, the intersection with the line $x_2=0$ gives the finite flat subgroup scheme
\[
H\coloneqq \{(1:-\zeta:0):\zeta\in\mu_3\}\subset (Y_0)_U[3]
\]
which is isomorphic to $\mu_3$. The modular description gives a classifying morphism
\[
\phi\colon U\longrightarrow X(3)^\circ
\]
over $\Q$ and an isomorphism
\[
(Y_0)_U
\xlongrightarrow{\sim}
U\times_{X(3)^\circ}V(3)^\circ
\]
of elliptic schemes over $U$. Over $\C$, Beauville's classification \cite[p.~658, first row of the table]{beauville1982elliptic} identifies the Hesse pencil with the elliptic modular family corresponding to $\Gamma(3)$. Any two full level-$3$ structures differ by a change of level structure, and such a change induces an automorphism of $X(3)^\circ_{\C}$. It follows from the fine moduli property that the classifying morphism $\phi_{\C}\colon U_{\C}\to X(3)^\circ_{\C}$ associated with $(P,H)$ is an isomorphism. It follows that $\phi$ is already an isomorphism over $\Q$, and hence extends uniquely to an isomorphism of smooth projective compactifications
\[
B_0\xlongrightarrow{\sim}X(3).
\]
Under this identification, the universal-property isomorphism above identifies the generic fiber of $\pi_0\colon Y_0\to B_0$ with the generic fiber of $V(3)\to X(3)$. Both $Y_0\to B_0$ and $V(3)\to X(3)$ are relatively minimal regular elliptic models of this generic elliptic curve. By uniqueness of the relatively minimal regular model \cite[Theorem~IV.1.1]{silverman1994advanced}, the isomorphism of generic fibers extends to an isomorphism
\[
Y_0\xlongrightarrow{\sim}V(3)
\]
over the identified bases. It follows that
\[
Y_0\times_{B_0}Y_0
\xlongrightarrow{\sim}
V(3)\times_{X(3)}V(3),
\]
and this isomorphism identifies their reduced singular loci. Blowing
up these loci gives
\[
W_0\xlongrightarrow{\sim}W(3).
\]

Finally, the equation defining $Y_0$ has coefficients in $\Z[1/6]$. For every $p\neq2,3$, its reduction modulo $p$ is precisely the equation defining the surface $Y$ in \Cref{sec:W}. The threefold $W$ of that section is, by definition, obtained from this reduced surface by taking the self-fiber-product over $\P^1_{\F_p}$ and blowing up its reduced singular locus. This completes the proof.
\end{proof}

\begin{lem}\label{lem:dominant-rational-E3-W}
Let $k$ be a field of characteristic different from $2$ and $3$, let $E$ be the Fermat cubic curve, and let $W$ be the threefold over $k$ defined in \Cref{sec:W}. There exists a dominant rational map $f\colon E^3\dashrightarrow W$ of degree $d$ a power of $3$ with the following properties. If $\Gamma\subset E^3\times W$ denotes the closure of the graph of $f$, then:
\begin{enumerate}
\item[(i)] for every prime $\ell$ invertible in $k$ with $\ell\nmid6$, we have $\Gamma_*\Gamma^*=d\cdot \mathrm{id}$ on $H^3(W_{\ov{k}},\Z_\ell(2))$;
\item[(ii)] if $k$ contains a primitive cube root of unity and $Q$ is
the correspondence introduced above, then $Q_*\Gamma^*=3\Gamma^*$ on $H^3(W_{\ov{k}},\Z_\ell(2))$ for every prime $\ell$ invertible in $k$ with $\ell\nmid6$.
\end{enumerate}
\end{lem}

\begin{proof}
Let $E'/k$ be the elliptic curve of equation $(x')^3-(y')^2z'+(z')^3/4=0$, as in  \cite[(13.1)]{schoen1995computation} and \Cref{prop:fermat-facts}(6). By
\Cref{lem:fermat-schoen-isogenous}, there is a $3$-isogeny $\varphi\colon E\to E'$ over $k$. Let $E_0$ and $E'_0$ be the elliptic curves over $\Q$
defined by the same equations as $E$ and $E'$, respectively. \Cref{lem:fermat-schoen-isogenous} gives a $3$-isogeny $\varphi_0\colon E_0\to E'_0$ over $\Q$, whose base change to any field of characteristic different from $2$ and $3$ is the isogeny $\varphi$ above. Let
\[
h_0\colon (E'_0)^3\dashrightarrow W(3)
\]
be the dominant rational map of
\cite[Theorem~13.2]{schoen1995computation}. By \Cref{lem:hesse-V3} and \Cref{rmk:reduce-mod-p}, we identify $W(3)$ with $W_0$, where $W_0$ is the
threefold over $\Q$ obtained from the same equation and blow-up
construction as the threefold $W$ of \Cref{sec:W}. Set
\[
f_0\coloneqq h_0\circ\varphi_0^3
\colon E_0^3\dashrightarrow W_0.
\]
By \cite[Proposition~10.2]{schoen2002complex}, the degree of $h_0$ is
a power of $3$. Since $\varphi_0^3$ also has degree a power of $3$,
the degree of $f_0$ is a power of $3$.

The rational map in the proof of
\cite[Theorem~13.2]{schoen1995computation} is given by explicit
formulas, and \cite[Lemma~13.6]{schoen1995computation} observes that
its graph spreads out over $\Z[1/6]$. The equation defining
$Y_0$, the blow-up construction defining $W_0$, and the isogeny
$\varphi_0$ of \Cref{lem:fermat-schoen-isogenous} are likewise defined
over $\Z[1/6]$. Reducing these constructions in positive
characteristic, or base changing them in characteristic zero, gives
for every field $k$ of characteristic different from $2$ and $3$ a
rational map $f\colon E^3\dashrightarrow W$. The explicit factorization in the proof of
\cite[Theorem~13.2]{schoen1995computation} shows that this map remains
dominant and has degree a power of $3$. We let $d$ be the degree of $f$.

We prove (i) first over $\C$. We resolve the indeterminacies of $f_0$ to
obtain a diagram
\[
(E_0)^3_{\C}\xlongleftarrow{\pi}\widetilde X\xlongrightarrow{g}(W_0)_{\C},
\]
where $\pi$ is a composite of blow-ups along smooth centers and $g$ is
generically finite of degree $d$. The proof of
\cite[Lemma~13.5]{schoen1995computation} gives
\[
H^3((W_0)_{\C},\C)
=
H^{3,0}((W_0)_{\C})\oplus H^{0,3}((W_0)_{\C}).
\]
The only new summands in the third cohomology introduced by blowing up
a smooth threefold arise from curve centers and are of the form $H^1(C_i,\Q)(-1)$, hence have Hodge types $(2,1)$ and $(1,2)$. It follows from the compatibility of the blow-up decomposition with Hodge structures that
\[
g^*H^3((W_0)_{\C},\Q)
\subset
\pi^*H^3((E_0)^3_{\C},\Q).
\]
Therefore $\pi^*\pi_*g^*=g^*$ on $H^3((W_0)_{\C},\Q)$. Since the graph correspondence $\Gamma_0$ of $f_0$ satisfies $\Gamma_0^*=\pi_*g^*$ and $(\Gamma_0)_*=g_*\pi^*$, we obtain $(\Gamma_0)_*\Gamma_0^*=
g_*\pi^*\pi_*g^*=g_*g^*=d\cdot \mathrm{id}$ on $H^3((W_0)_{\C},\Q)$.

By the comparison theorem, the same equality holds on
$H^3((W_0)_{\C},\Q_\ell(2))$. The group
$H^3((W_0)_{\C},\Z_\ell(2))$ is torsion-free by
\cite[Lemma~10.3(i)]{schoen2002complex}, so the equality holds
integrally on $H^3((W_0)_{\C},\Z_\ell(2))$. Finally, \cite[Lemma~13.6]{schoen1995computation} and the spreading
described above allow us to specialize the graph correspondence.
Compatibility of smooth proper base change with correspondences gives $\Gamma_*\Gamma^*
=d\cdot\mathrm{id}$ on $H^3(W_{\ov{k}},\Z_\ell(2))$, proving (i).

For (ii), assume that $k$ contains a primitive cube root of unity.
The construction in the proof of
\cite[Theorem~13.2]{schoen1995computation} factors the rational map
$h_0$ through a quotient
\[
(E'_0)^3\dashrightarrow N\backslash(E'_0)^3,
\]
where the diagonal subgroup of order $3$ is contained in $N$; see also
\cite[Proposition~10.2]{schoen2002complex}. Hence $h_0$ is invariant
under the diagonal order-three action on $(E'_0)^3$. By
\Cref{lem:fermat-schoen-isogenous}, the isogeny
$\varphi_0\colon E_0\to E'_0$ is equivariant for the corresponding
order-three automorphisms. Thus $f_0$ is invariant under the diagonal subgroup $\Delta\subset(\mu_3)^3$. This invariance specializes to the rational map $f$ over $k$. Thus, $(\Gamma_\delta)_*\Gamma^*=\Gamma^*$ for every $\delta\in\Delta$. Summing over the three elements $\delta\in \Delta$ gives
\[
Q_*\Gamma^*
=
\sum_{\delta\in\Delta}(\Gamma_\delta)_*\Gamma^*
=
3\Gamma^*,
\]
which proves (ii).
\end{proof}

\begin{lem}\label{lemma:cup-correspondence}
Let $k$ be a field of characteristic different from $2$ and $3$
containing a primitive cube root of unity, let $E/k$ be the Fermat cubic,
and let $\ell>3$ be a prime invertible in $k$. Let $f\colon E^3\dashrightarrow W$ be a rational map as in \Cref{lem:dominant-rational-E3-W}, and let
$\Gamma\subset E^3\times W$ be the closure of its graph.
\begin{enumerate}
\item[(i)] The cup product induces a surjective map
\[
H^1(E^3_{\ov{k}},\Z_\ell)\otimes_{\Z_\ell}
Q_*H^2(E^3_{\ov{k}},\Z_\ell)
\xlongrightarrow{\cup}
Q'_*H^3(E^3_{\ov{k}},\Z_\ell).
\]

\item[(ii)] If $k$ is a finite field of characteristic
$p\equiv1\pmod3$, then
\[
Q_*H^2(E^3_{\ov{k}},\Z_\ell(1))
=
H^2(E^3_{\ov{k}},\Z_\ell(1))^{G_k}.
\]

\item[(iii)] The map
\[
\Gamma^*\colon
H^3(W_{\ov{k}},\Z_\ell(2))
\xlongrightarrow{\sim}
Q_*H^3(E^3_{\ov{k}},\Z_\ell(2))
\]
is an isomorphism.
\end{enumerate}
\end{lem}

\begin{proof}
(i) Our argument is inspired by the proof of
\cite[Theorem~14.4 and Remark~14.9]{schoen1995computation}. Since cup
product gives an isomorphism
\[
\bigwedge H^1(E^3_{\ov{k}},\Z_\ell)
\xlongrightarrow{\sim}
H^*(E^3_{\ov{k}},\Z_\ell),
\]
the cup-product map
\[
H^1(E^3_{\ov{k}},\Z_\ell)\otimes_{\Z_\ell}
H^2(E^3_{\ov{k}},\Z_\ell)
\longrightarrow
H^3(E^3_{\ov{k}},\Z_\ell)
\]
is surjective.

Let $C_3$ be a cyclic group of order $3$, acting on $E_{\ov{k}}$
through the automorphism \eqref{eq:mu_3-action}. Let
$R=\Z_\ell[\zeta_3]$. Since $\ell\neq3$, we have
\begin{equation}\label{eq:decompose-h1}
H^1(E_{\ov{k}},\Z_\ell)\otimes_{\Z_\ell}R
\simeq L\oplus\overline L,
\end{equation}
where $L$ and $\overline L$ are the rank-one $R[C_3]$-modules on which
a generator of $C_3$ acts by multiplication by $\zeta_3$ and $\zeta_3^2$, respectively.
Thus
\[
H^1(E^3_{\ov{k}},\Z_\ell)\otimes_{\Z_\ell}R=V\oplus\overline V,
\]
where $V=L^{\oplus3}$ and  $\overline V=\overline L^{\oplus3}$. For all $i\geq j\geq0$, set
\[
V_{i,j}
\coloneqq
\bigwedge^{i-j}V\otimes_R\bigwedge^j\overline V.
\]
Then
\[
H^i(E^3_{\ov{k}},\Z_\ell)\otimes_{\Z_\ell}R
=
\bigoplus_{j=0}^iV_{i,j}.
\]
The action of $C_3$ on $V_{i,j}$ is given by multiplication by $\zeta_3^{i+j}$. Therefore $Q$ acts on $V_{i,j}$ as multiplication by $3$ if $i+j\equiv 0\pmod 3$ and as zero otherwise.

In degree $2$, the $Q$-part is therefore $V_{2,1}$, while in degree
$3$ the $Q'$-part is $V_{3,1}\oplus V_{3,2}$. The cup-product map in
the statement becomes after extension of scalars to $R$ the sum of the
natural wedge maps
\[
(V\oplus\overline V)\otimes_RV_{2,1}
\longrightarrow
V_{3,1}\oplus V_{3,2},
\]
which is surjective. Since $R$ is faithfully flat over $\Z_\ell$, this
proves (i).

Notice also that in degree $3$ the $Q$-part is $V_{3,0}\oplus V_{3,3}$. Hence
$Q_*H^3(E^3_{\ov{k}},\Z_\ell)$ is a free $\Z_\ell$-module of rank $2$.

(ii) Let $q\coloneqq |k|$. The geometric Frobenius commutes with the order-$3$
automorphism \eqref{eq:mu_3-action} and hence preserves the decomposition \eqref{eq:decompose-h1}. Its eigenvalues on $L$ and $\ov{L}$ are algebraic numbers
$\pi$ and $\overline\pi$ such that $\pi\overline\pi=q$. 

Since $E$ is ordinary, one of $\pi$ and $\bar\pi$ is a $p$-adic unit, while the other has $p$-adic valuation $v_p(q)$. Hence $\pi/\bar\pi$ has non-zero $p$-adic valuation and is therefore not a root of unity. Therefore, the Frobenius-fixed part of
$H^2(E^3_{\ov{k}},\Z_\ell(1))$ is exactly the mixed-character
summand. Thus
\[
Q_*H^2(E^3_{\ov{k}},\Z_\ell(1))
\otimes_{\Z_\ell}\Q_\ell(\zeta_3)
=
V\otimes_R\overline V(1)\otimes_{\Z_\ell}\Q_\ell=
H^2(E^3_{\ov{k}},\Z_\ell(1))^{G_k}
\otimes_{\Z_\ell}\Q_\ell(\zeta_3).
\]
Since both $Q_*H^2(E^3_{\ov{k}},\Z_\ell(1))$ and $H^2(E^3_{\ov{k}},\Z_\ell(1))^{G_k}$ 
are saturated submodules of $H^2(E^3_{\ov{k}},\Z_\ell(1))$, they are equal to each other.

(iii) Let $d\coloneqq \deg(f)$. By \Cref{lem:dominant-rational-E3-W}(ii), we have $Q_*\Gamma^*=3\Gamma^*$, so
\[\operatorname{Im}(\Gamma^*)\subset Q_*H^3(E^3_{\ov{k}},\Z_\ell(2)).\]
By \cite[Lemma~10.3(i) and Proposition~4.15]{schoen2002complex} and smooth proper base change, $H^3(W_{\ov{k}},\Z_\ell(2))$ is free of rank $2$. By the character calculation in the proof of (i), $Q_*H^3(E^3_{\ov{k}},\Z_\ell(2))$ is also free of rank $2$. Finally, by \Cref{lem:dominant-rational-E3-W}(i), we have $\Gamma_*\Gamma^*=d\cdot \mathrm{id}$ on $H^3(W_{\ov{k}},\Z_\ell(2))$.
Since $d$ is a power of $3$ and $\ell>3$, it is a unit in $\Z_\ell$.
Hence $d^{-1}\Gamma_*$ is a left inverse to $\Gamma^*$, so
$\Gamma^*$ is a split injection. Since its source and target are free
$\Z_\ell$-modules of rank $2$, it is an isomorphism.
\end{proof}


\begin{proof}[Proof of \Cref{main-fermat-cubic-thm}]

Let $X\coloneqq E^3$. Suppose first that $p\equiv2\pmod3$. We claim that $E$ is
supersingular. If $p>2$, this follows from \Cref{prop:fermat-facts}. If $p=2$, then 
\[
E(\F_2)=\{(1:1:0),(1:0:1),(0:1:1)\},
\]
so $|E(\F_2)|=3$. Thus the trace of Frobenius is $a_2=2+1-|E(\F_2)|=0$, and therefore $E$ is supersingular; see \cite[Theorem~V.3.1]{silverman2009arithmetic}. It follows that $X_k=E^3_k$ is a supersingular abelian threefold for every finite extension $k/\F_p$. Applying \Cref{cor:supersingular-threefold}, we obtain, for every prime
$\ell\neq p$,
\[H^3_{\nr}(k(X)/k,\Q_\ell/\Z_\ell(2))=0\]
and the surjectivity of
\[CH^2(X_k)_{\Z_\ell}\longrightarrow H^4(X_k,\Z_\ell(2)).\]

Suppose now that $p\equiv 1\pmod 3$. In particular $p\neq 2$, and $\F_p$ contains a primitive cube root of unity, so the correspondences $Q$ and $Q'=3I-Q$ introduced above are defined over $\F_p$. Let $k/\F_p$ be an arbitrary finite extension.

We claim first that the Abel--Jacobi map
\[
\alpha^2_{X_k,\ell}\colon
CH^2_{\mathrm{hom},\ell}(X_k)\longrightarrow
H^1(k,H^3(X_{\bar{k}},\Z_\ell(2)))
\]
is surjective. Let $\Gamma\subset X\times W$
be the closure of the graph of the rational map $X\dashrightarrow W$ from
\Cref{lem:dominant-rational-E3-W}. By functoriality of the Abel--Jacobi map
with respect to correspondences, we have a commutative diagram
\[
\begin{tikzcd}
CH^2_{\mathrm{hom},\ell}(W_k)_{\Z_\ell}
  \arrow[r,"(\alpha^2_{W_k,\ell})_{\Z_\ell}"]
  \arrow[d,"\Gamma^*"]
&
H^1(k,H^3(W_{\bar k},\Z_\ell(2)))
  \arrow[d,"\Gamma^*"]
\\
CH^2_{\mathrm{hom},\ell}(X_k)_{\Z_\ell}
  \arrow[r,"(\alpha^2_{X_k,\ell})_{\Z_\ell}"]
&
H^1(k,H^3(X_{\bar k},\Z_\ell(2))).
\end{tikzcd}
\]
By \Cref{lemma:cup-correspondence}(iii), the right vertical map identifies $H^1(k,H^3(W_{\bar k},\Z_\ell(2)))$ with $Q_*H^1(k,H^3(X_{\bar k},\Z_\ell(2)))$. By \Cref{w-aj-surjective}, the top horizontal map is surjective. Therefore the image of the Abel--Jacobi map for $X_k$ contains $Q_*H^1(k,H^3(X_{\bar k},\Z_\ell(2)))$.

We now prove the surjectivity of the $Q'$-part. Set
\[
A_\ell\coloneqq Q_*H^2(X_{\bar{k}},\Z_\ell(1)).
\]
By \Cref{lemma:cup-correspondence}(ii), we have
$
A_\ell=H^2(X_{\bar{k}},\Z_\ell(1))^{G_k}.
$
In particular $G_k$ acts trivially on $A_\ell$. By Tate's theorem for
divisors on abelian varieties over finite fields \cite{tate1966endomorphisms}, the cycle class map
\[
CH^1(X_k)_{\Z_\ell}\longrightarrow A_\ell
\]
is surjective. Moreover, by \Cref{rmk:AJ_properties}(7), the Abel--Jacobi map
in codimension $1$,
\[
\alpha^1_{X_k,\ell}\colon
CH^1_{\mathrm{hom},\ell}(X_k)\longrightarrow
H^1(k,H^1(X_{\bar{k}},\Z_\ell(1))),
\]
is surjective.

By \Cref{lemma:cup-correspondence}(i), cup product induces a surjection
\[
H^1(X_{\bar{k}},\Z_\ell(1))\otimes_{\Z_\ell} A_\ell
\relbar\joinrel\twoheadrightarrow
Q'_*H^3(X_{\bar{k}},\Z_\ell(2)).
\]
Passing to $H^1(k,-)$ therefore gives a surjection
\[
H^1(k,H^1(X_{\bar{k}},\Z_\ell(1)))\otimes_{\Z_\ell} A_\ell
\relbar\joinrel\twoheadrightarrow
H^1(k,Q'_*H^3(X_{\bar{k}},\Z_\ell(2))).
\]

Now let $\eta\in
H^1(k,Q'_*H^3(X_{\bar k},\Z_\ell(2)))$. By the above, we may write $\eta$ as a finite sum of the images
of tensors
\[
\beta_i\otimes a_i,\qquad
\beta_i\in H^1(k,H^1(X_{\bar k},\Z_\ell(1))),
\quad
a_i\in A_\ell.
\]
Since $\alpha^1_{X_k,\ell}$ is surjective, there exist $z_i\in CH^1_{\mathrm{hom},\ell}(X_k)$ such that $\alpha^1_{X_k,\ell}(z_i)=\beta_i$. By Tate's theorem for
divisors, there exist $D_i\in CH^1(X_k)_{\Z_\ell}$ whose cycle class is $a_i$. Then
\[
\sum_i z_i\cdot D_i
\in CH^2_{\mathrm{hom},\ell}(X_k)_{\Z_\ell},
\]
and the compatibility of Abel--Jacobi maps with products
(see \Cref{rmk:AJ_properties}(2), applied, for each $i$, to the
correspondence $\Delta_{X*}D_i$) shows that the Abel--Jacobi map sends $\sum_i z_i\cdot D_i$ to $\eta$. Hence the image $(\alpha^2_{X_k,\ell})_{\Z_\ell}$ contains the $Q'$-part of the codomain. Since it also contains the $Q$-part, the
decomposition of the target induced by the complementary idempotents
$Q/3$ and $Q'/3$ (recall that $\ell\neq 3$) shows that
\[
(\alpha^2_{X_k,\ell})_{\Z_\ell}\colon CH^2_{\mathrm{hom},\ell}(X_k)_{\Z_\ell}
\longrightarrow
H^1(k,H^3(X_{\bar k},\Z_\ell(2)))
\]
is surjective. Since $H^1(k,H^3(X_{\bar k},\Z_\ell(2)))$ is finite, this implies the surjectivity of
\[
\alpha^2_{X_k,\ell}\colon
CH^2_{\mathrm{hom},\ell}(X_k)\longrightarrow
H^1(k,H^3(X_{\bar{k}},\Z_\ell(2))).
\]

We now prove the surjectivity of the cycle class map
\[
CH^2(X_k)_{\Z_\ell}\longrightarrow H^4(X_k,\Z_\ell(2)).
\]
The Hochschild--Serre exact sequence gives
\[
0\to H^1(k,H^3(X_{\bar{k}},\Z_\ell(2)))
\to H^4(X_k,\Z_\ell(2))
\to H^4(X_{\bar{k}},\Z_\ell(2))^{G_k}\to 0.
\]
By \Cref{medium-tate}, the cycle map
\[
CH^2(X_k)_{\Z_\ell}\longrightarrow
H^4(X_{\bar{k}},\Z_\ell(2))^{G_k}
\]
is surjective. By the surjectivity of the Abel--Jacobi proved above and the Hochschild--Serre spectral sequence, we deduce that the cycle map
\[
CH^2(X_k)_{\Z_\ell}\longrightarrow H^4(X_k,\Z_\ell(2))
\]
is surjective.

As noted in the proof of \cite[Theorem~5.6]{kahn2012classes}, $E^3$ belongs to
$B_{\mathrm{Tate}}$, and hence so does $(E^3)_k$ for every finite extension $k/\F_p$. Finally, by \cite[Th\'eor\`eme 3.18]{colliot2013cycles}, the group
$H^3_{\nr}(k(X)/k,\Q_\ell/\Z_\ell(2))$ is finite for $X=E^3$, and by
\cite[Th\'eor\`eme 1.1]{kahn2012classes} it is isomorphic to the torsion
subgroup of the cokernel of
\[
CH^2(X_k)_{\Z_\ell}\longrightarrow H^4(X_k,\Z_\ell(2)).
\]
Since this cokernel is zero, we obtain
\[
H^3_{\nr}(k(X)/k,\Q_\ell/\Z_\ell(2))=0.
\]
The assertion over $\overline{\F}_p$ follows from the rest and the standard \Cref{lem:unramified-direct-limit} below.
\end{proof}

\begin{lem}\label{lem:unramified-direct-limit}
    Let $k$ be a field, let $X$ be a smooth $k$-variety, let $i$ be an integer, and let $A$ be a torsion discrete $G_k$-module. If $\on{char}(k)=p>0$, assume that $A[p]=0$. Then pullback induces a group isomorphism
   \begin{equation}\label{direct-limit}\varinjlim_{k\subset K\subset\ov{k}}H^i_{\nr}(K(X)/K,A)\xlongrightarrow{\sim}H^i_{\nr}(\ov{k}(X)/\ov{k},A),\end{equation}
    where $K/k$ ranges over all finite subextensions of $k$ contained in $\ov{k}$.
\end{lem}

\begin{proof}
    We know that
    \[\varinjlim_{k\subset K\subset \ov{k}} H^i(K(X),A)\xlongrightarrow{\sim} H^i(\ov{k}(X),A),\]
    see \cite[Tag 09YQ]{stacks-project}. In particular, (\ref{direct-limit}) is injective. We now prove surjectivity. Given $\alpha\in H^i_{\on{nr}}(\ov{k}(X)/\ov{k},A)$, up to replacing $k$ by a finite extension, we may assume that $\alpha$ is the image of some $\alpha_0\in H^i(k(X),A)$. Since \'etale cohomology commutes with direct limits, there exists a dense open subscheme $U\subset X$ such that $\alpha_0$ extends to an element of $H^i(U,A)$. In particular, $\alpha_0$ is unramified at every codimension $1$ point of $X$ contained in $U$. The complement $X_{\ov{k}}\setminus U_{\ov{k}}$ contains at most finitely many codimension $1$ points of $X_{\ov{k}}$ and hence, replacing $k$ by a finite extension, we may assume that all of these finitely many points are defined over $k$. Let $P$ be a codimension $1$ point of $X$ not contained in $U$. Since $\alpha$ is unramified at $P_{\ov{k}}$, it comes from $H^i(\mc{O}_{X_{\ov{k}},P_{\ov{k}}},A)$. We have $\mc{O}_{X_{\ov{k}},P_{\ov{k}}}=\varinjlim_K \mc{O}_{X_K,P_K}$, and hence by \cite[Tag 09YQ]{stacks-project} pullback gives an isomorphism
    \[\varinjlim_{k\subset K\subset \ov{k}} H^i(\mc{O}_{X_K,P_K},A)\xlongrightarrow{\sim} H^i(\mc{O}_{X_{\ov{k}},P_{\ov{k}}},A),\]
    where $K/k$ ranges over all finite subextensions of $k$ contained in
    $\ov{k}$. Choose such a local lift over some $K$. Its restriction
    to $K(X)$ and the restriction of $\alpha_0$ have the same image in
    $H^i(\ov{k}(X),A)$. By the injectivity of
    \[
    \varinjlim_K H^i(K(X),A)\longrightarrow H^i(\ov{k}(X),A),
    \]
    after enlarging $K$ once more these two generic restrictions agree.
    Therefore, up to replacing $k$ by a finite extension, we may assume
    that $\alpha_0$ comes from $H^i(\mc{O}_{X,P},A)$, that is, $\alpha_0$ is
    unramified at $P$. We may repeat this process for all codimension $1$
    points of $X$ not in $U$, and hence, up to replacing $k$ by a finite
    extension, conclude that $\alpha_0$ is unramified. Thus
    (\ref{direct-limit}) is also surjective.
\end{proof}

\section{Consequences of Theorem \ref{main-fermat-cubic-thm}}\label{sec:7}

\subsection{Coniveau and strong coniveau filtrations}\label{subsec:coniveau}

We begin by determining the coniveau-one part of $H^3(E^3)$ and comparing the coniveau and strong coniveau filtrations.

\begin{lem}\label{lem:transcendental-E3}
Let $k$ be the algebraic closure of a finite field of characteristic
$p\neq3$, let $E/k$ be the Fermat cubic, and let $\ell\neq p$ be a prime.
If $p\equiv1\pmod3$, assume moreover that $\ell>3$.

If $p\equiv1\pmod3$, then
\[
N^1H^3(E^3,\Z_\ell(2))
=
Q'_*H^3(E^3,\Z_\ell(2)),
\]
and hence
\[
H^3_{\mathrm{tr}}(E^3,\Z_\ell(2))
\coloneqq
H^3(E^3,\Z_\ell(2))/N^1H^3(E^3,\Z_\ell(2))
\simeq
Q_*H^3(E^3,\Z_\ell(2)),
\]
which is free of rank $2$.

If $p\equiv2\pmod3$, then
\[
N^1H^3(E^3,\Z_\ell(2))
=
H^3(E^3,\Z_\ell(2)),
\]
and hence
\[
H^3_{\mathrm{tr}}(E^3,\Z_\ell(2))=0.
\]
\end{lem}

\begin{proof}
Suppose first that $p\equiv2\pmod3$. Then $E$ is supersingular.
By \Cref{lem:supersingular-gysin},
\[
H^3(E^3,\Z_\ell(2))
=
\widetilde N^1H^3(E^3,\Z_\ell(2))
\subset
N^1H^3(E^3,\Z_\ell(2)),
\]
and the result follows.

Suppose now that $p\equiv1\pmod3$. By
\Cref{lemma:cup-correspondence}(i),(ii) and Tate's theorem for divisors,
we have
\[
Q'_*H^3(E^3,\Z_\ell(2))
\subset
N^1H^3(E^3,\Z_\ell(2)).
\]

We claim that
\[
N^1H^3(E^3,\Z_\ell)
\cap
Q_*H^3(E^3,\Z_\ell)
=0.
\]
By \Cref{lemma:cup-correspondence}(iii), we have $Q_*H^3(E^3,\Z_\ell)
\simeq H^3(W,\Z_\ell)$. The geometric-Frobenius polynomial on $H^3(W,\Z_\ell)$ is $T^2-b_pT+p^3$ by \cite[Proposition~12.6]{schoen1995computation}, and
$p\nmid b_p$ by \cite[Example~12.8(4)]{schoen1995computation}.

Every geometric-Frobenius eigenvalue on
$N^1H^3(E^3,\Q_\ell)$ is divisible by $p$ as an algebraic integer.
Indeed, after a finite constant extension defining the supports and
choosing alterations of their divisor components,
$N^1H^3(E^3,\Q_\ell)$ is generated by Gysin images of groups $H^1(\widetilde D,\Q_\ell)(-1)$. Thus, if $\lambda$ is such an eigenvalue, then
$\lambda^m/p^m$ is an algebraic integer for some $m$, and hence
$\lambda/p$ is an algebraic integer.

On the other hand, neither root of $T^2-b_pT+p^3$ is divisible by $p$
as an algebraic integer. Indeed, the two roots have complex absolute
value $p^{3/2}$, hence are not rational, and so the polynomial is irreducible over $\Q$. If one
root were divisible by $p$, then so would its conjugate be, and hence
their sum $b_p$ would be divisible by $p$, a contradiction.
This proves the claim.

Since $\ell\neq 3$, the idempotents $Q/3$ and
$Q'/3$ give
\[
H^3(E^3,\Z_\ell)
=
Q_*H^3(E^3,\Z_\ell)
\oplus
Q'_*H^3(E^3,\Z_\ell).
\]
Moreover, these correspondences preserve the coniveau filtration. Thus
\[
N^1H^3(E^3,\Z_\ell)
\subset
Q'_*H^3(E^3,\Z_\ell).
\]
Together with the reverse inclusion proved above, this gives
\[
N^1H^3(E^3,\Z_\ell)
=
Q'_*H^3(E^3,\Z_\ell).
\]
Finally,
\[
Q_*H^3(E^3,\Z_\ell(2))
\simeq
H^3(W,\Z_\ell(2))
\]
is free of rank $2$ by \Cref{lemma:cup-correspondence}(iii) and
\cite[Lemma~10.3(i)]{schoen2002complex}.
\end{proof}

\begin{prop}\label{c-sc-E3}
Let $k$ be the algebraic closure of a finite field of characteristic
$p\neq3$, let $E/k$ be the Fermat cubic, and let $\ell\neq p$ be a prime.
If $p\equiv1\pmod3$, assume moreover that $\ell>3$. Then
\[
N^1H^3(E^3,\Z_\ell)=\widetilde N^1H^3(E^3,\Z_\ell).
\]
If $p\equiv1\pmod3$ (resp. $p\equiv2\pmod3$), these are free
$\Z_\ell$-modules of rank $18$ (resp. $20$).
\end{prop}

\begin{proof}
Suppose first that $p\equiv2\pmod3$. Then $E$ is supersingular, so
\Cref{lem:supersingular-gysin} gives $H^3(E^3,\Z_\ell) = \widetilde N^1H^3(E^3,\Z_\ell)$.
Hence
\[
N^1H^3(E^3,\Z_\ell)
=
\widetilde N^1H^3(E^3,\Z_\ell)
=
H^3(E^3,\Z_\ell)
\]
is $\Z_\ell$-free of rank $20$.

Suppose now that $p\equiv1\pmod3$. By
\Cref{lemma:cup-correspondence}(i),(ii) and Tate's theorem for divisors,
the module $Q'_*H^3(E^3,\Z_\ell)$ is generated by cup products of classes in $H^1(E^3,\Z_\ell)$ with divisor classes. Writing a divisor
class as a difference of sufficiently ample classes and applying
Bertini, the projection formula shows that cup product of a class in $H^1(E^3,\Z_\ell)$ with a divisor
class factors through a Gysin map from a smooth divisor. Hence
\[
Q'_*H^3(E^3,\Z_\ell)
\subset
\widetilde N^1H^3(E^3,\Z_\ell).
\]
By \Cref{lem:transcendental-E3}, we have $N^1H^3(E^3,\Z_\ell)= Q'_*H^3(E^3,\Z_\ell)$.
Therefore
\[
Q'_*H^3(E^3,\Z_\ell)
\subset
\widetilde N^1H^3(E^3,\Z_\ell)
\subset
N^1H^3(E^3,\Z_\ell)
=
Q'_*H^3(E^3,\Z_\ell),
\]
so all three modules are equal.

By \Cref{lem:transcendental-E3}, the complement $Q_*H^3(E^3,\Z_\ell)$ is free
of rank $2$. Since $H^3(E^3,\Z_\ell)$ is free of rank $20$, the common
module above is free of rank $18$.
\end{proof}

\subsection{Galois descent of codimension-\texorpdfstring{$2$}{2} cycles}

The vanishing of degree-$3$ unramified cohomology obtained in \Cref{main-fermat-cubic-thm} and the results on the coniveau and strong coniveau filtrations of \S\ref{subsec:coniveau} have concrete Galois descent consequences
for codimension-$2$ cycles, as we now show.

\begin{cor}\label{cor:CTK}
Let $p\neq3$ be a prime, let $E\subset\P^2_{\F_p}$ be the Fermat cubic,
and let $\ell\neq p$ be a prime. If $p\equiv 1 \pmod 3$, assume moreover that $\ell>3$.
Then, for every finite extension $k/\F_p$, base change induces an isomorphism
\[
CH^2(E^3_k)_{\Z_\ell}
\xlongrightarrow{\sim}
CH^2(E^3_{\ov{\F}_p})^{G_k}_{\Z_\ell}.
\]
\end{cor}

\begin{proof}
Set $X=E^3_k$. Recall that
$H^3(X_{\ov{k}},\Z_\ell(2))$ is torsion-free. In the notation of
\cite[Th\'eor\`eme~6.8]{colliot2013cycles}, the finite Galois module $M\coloneqq \oplus_q H^3(X_{\ov{k}},\Z_q(2))\{q\}$ therefore satisfies $M\{\ell\}=H^3(X_{\ov{k}},\Z_\ell(2))\{\ell\}=0$. 
Hence $H^1(k,M)\{\ell\}=0$. By \Cref{main-fermat-cubic-thm}, we have
\[
H^3_{\nr}(k(X)/k,\Q_\ell/\Z_\ell(2))=0.
\]
Now from the exact sequence of \cite[Th\'eor\`eme~6.8]{colliot2013cycles} we deduce that both the kernel and the cokernel of
\[
CH^2(X)\longrightarrow CH^2(X_{\ov{k}})^{G_k}
\]
have trivial $\ell$-primary components, that the kernel is finite, and that the cokernel is torsion. Tensoring with $\Z_\ell$ therefore gives the desired isomorphism.
\end{proof}

For a field $k$ and a $k$-variety $X$, write $CH^i_{k-\alg}(X)$ for the subgroup of classes algebraically equivalent to zero over $k$ in the sense of Fulton (see \cite[\S 4.1]{scavia2023coniveau}), and write $CH^i_{\alg}(X)$ for the subgroup of classes which become algebraically trivial over a finite separable extension of $k$, equivalently over $\bar k$.

\begin{cor}\label{cor:waj}
Let $p\neq3$ be a prime, let $E\subset\P^2_{\F_p}$ be the Fermat cubic,
let $k/\F_p$ be a finite extension, and let $\ell\neq p$ be a prime.
If $p\equiv1\pmod3$, assume moreover that $\ell>3$. Then the following
hold.
\begin{enumerate}
\item[(i)] The algebraic Abel--Jacobi map induces an isomorphism
\[
\cl_a\colon CH^2_{\alg}(E^3_k)_{\Z_\ell}
\xlongrightarrow{\sim}
H^1(k,N^1H^3(E^3_{\ov{\F}_p},\Z_\ell(2))).
\]

\item[(ii)] Base change induces an isomorphism
\[
CH^2_{\alg}(E^3_k)_{\Z_\ell}
\xlongrightarrow{\sim}
CH^2_{\alg}(E^3_{\ov{\F}_p})^{G_k}_{\Z_\ell}.
\]

\item[(iii)] A class in $CH^2(E^3_k)\{\ell\}$ belongs to
$CH^2_{k-\alg}(E^3_k)$ if and only if it belongs to
$CH^2_{\alg}(E^3_k)$.

\item[(iv)] If $p\equiv1\pmod3$, then the abelian groups $CH^2_{\alg}(E^3_k)_{\Z_\ell}$ and $(E(k)\{\ell\})^9$ have the same order.
\end{enumerate}
\end{cor}

\begin{proof}
By \Cref{c-sc-E3}, coniveau and strong coniveau agree in degree $3$.
Parts (1) and (3) of \cite[Theorem~1.2]{scavia2023coniveau} give (i) and
(iii), and part (2) gives the surjectivity in (ii). The second isomorphism
in part (1) of that theorem identifies the $\ell$-primary kernel of
$CH^2(E^3_k)\to CH^2(E^3_{\ov{\F}_p})$ with $H^1(k,H^3(E^3_{\ov{\F}_p},\Z_\ell(2))_{\tors})$, 
which is zero because the integral cohomology of $E^3$ is torsion-free.
Thus the map in (ii) is also injective. For (iv), set
\[
M=N^1H^3(E^3_{\ov{\F}_p},\Z_\ell(2)),
\qquad V=H^1(E_{\ov{\F}_p},\Z_\ell)(1).
\]
\Cref{c-sc-E3} gives $M=Q'_*H^3(E^3_{\ov{\F}_p},\Z_\ell(2))$. After adjoining a third root of unity to $\Q_\ell$, write
\[
H^1(E,\Q_\ell)\otimes_{\Q_\ell}\Q_\ell(\zeta_3)
=
L\oplus\bar L
\]
for the two non-trivial $C_3$-characters $\chi$ and $\bar\chi$
(cf. the proof of \Cref{lemma:cup-correspondence} for the notation).
The six K\"unneth summands of type $(2,1,0)$ give six copies
of $V_{\Q_\ell}\otimes_{\Q_\ell}\Q_\ell(\zeta_3)$.
Within the $(1,1,1)$-K\"unneth summand, the two ``pure'' character terms
$L^{\otimes3}$ and $\bar L^{\otimes3}$ form the $Q$-summand, while the
``mixed'' terms form the $Q'$-summand. Since $L\otimes\bar L\simeq\Q_\ell(\zeta_3)(-1)$ 
the ``mixed'' terms give three further copies of $V_{\Q_\ell}\otimes_{\Q_\ell}\Q_\ell(\zeta_3)$. Therefore, after extending scalars to $\Q_\ell(\zeta_3)$, the
Frobenius characteristic polynomial of $M_{\Q_\ell}$ is the ninth
power of that of $V_{\Q_\ell}$. Since characteristic polynomials are
unchanged by scalar extension, we have
\[
\det(\operatorname{Frob}_k-1\mid M_{\Q_\ell})
=
\det(\operatorname{Frob}_k-1\mid V_{\Q_\ell})^9.
\]

For a finite free $\Z_\ell$-module $T$ on which $1$ is not a Frobenius
eigenvalue,
\[
H^1(k,T)=\operatorname{Coker}(\operatorname{Frob}_k-1),
\qquad
|H^1(k,T)|
=
|\det(\operatorname{Frob}_k-1\mid T_{\Q_\ell})|_\ell^{-1};
\]
see \cite[Lemma~5.5]{scavia2023coniveau}. Hence $|H^1(k,M)|=|H^1(k,V)|^9$. The principal polarization and the Weil
pairing identify $V$ with $T_\ell E$, while Kummer theory and Lang's
theorem give $H^1(k,T_\ell E)\simeq E(k)\{\ell\}$. Now (iv) follows from (i).
\end{proof}

For a $k$-variety $X$, set $\mathrm{A}^i(X)\coloneqq CH^i(X)/CH^i_{k-\alg}(X)$.

\begin{cor}
Let $p\neq3$ be a prime, let $E\subset\P^2_{\F_p}$ be the Fermat cubic, and let $\ell\neq p$ be a prime. 
If $p\equiv 1\pmod 3$, assume moreover that $\ell>3$.
Then, for every finite extension $k/\F_p$, base change induces an isomorphism
\[\mathrm{A}^2(E^3_k)_{\Z_\ell}\xlongrightarrow{\sim}\mathrm{A}^2(E^3_{\ov{\F}_p})^{G_k}_{\Z_\ell}.\]
\end{cor}

\begin{proof}
Set $C\coloneqq CH^2_{\alg}(E^3_{\ov{\F}_p})$ and 
$F\coloneqq \operatorname{Frob}_k$. We first show that $H^1(k,C)=0$. Indeed,
every $c\in C$ is obtained from a degree-zero class on the Jacobian $J$ of a suitable smooth projective curve by a correspondence, all defined over some finite extension
$k_m/k$ of degree $m$. By Lang's theorem, $F^m-1$ is surjective on $J(\ov{k})$.
Thus $c\in(F^m-1)C\subset(F-1)C$. Thus $C=(F-1)C$ and hence 
$H^1(k,C)=C/(F-1)C=0$. Taking invariants therefore gives the short exact
sequence
\[
0\longrightarrow C^{G_k}\longrightarrow
CH^2(E^3_{\ov{\F}_p})^{G_k}\longrightarrow
\mathrm{A}^2(E^3_{\ov{\F}_p})^{G_k}\longrightarrow0.
\]
Now compare this sequence with the defining sequence for
$\mathrm{A}^2(E^3_k)$. The middle vertical map becomes an isomorphism after
tensoring with $\Z_\ell$ by \Cref{cor:CTK}. The left vertical map becomes
surjective by \Cref{cor:waj}(ii), while part (iii) permits us to replace
geometric algebraic equivalence on the source by algebraic equivalence over
$k$. Since $\Z_\ell$ is flat over $\Z$, the two rows remain exact after
tensoring. The induced map on the quotients is therefore an isomorphism.
\end{proof}

\begin{rem}
    For related results on the behavior of algebraic equivalence for one-cycles under field extension, see \cite[Theorem~7]{kollar2025stable}.
\end{rem}

\subsection{Griffiths group}

Using \Cref{main-fermat-cubic-thm}, we complete the computation of the Griffiths group $\on{Griff}^2(E^3)$ initiated by Schoen \cite[Theorem 0.1]{schoen1995computation}.

\begin{cor}\label{griffiths}
Let $k$ be the algebraic closure of a finite field of characteristic
$p\neq3$, let $E/k$ be the Fermat cubic, and let $\ell\neq p$ be a prime.
If $p\equiv1\pmod3$, assume moreover that $\ell>3$. Then
\begin{itemize}
    \item[(i)] If $p\equiv1\pmod3$, then $\Griff^2(E^3)\{\ell\}\simeq(\Q_\ell/\Z_\ell)^2$.
    \item[(ii)] If $p\equiv2\pmod3$, then $\Griff^2(E^3)\{\ell\}=0$.
\end{itemize}
\end{cor}

When $p\equiv1\pmod3$ and $\ell\equiv2\pmod3$, part (i) was proved by
Schoen in \cite[Theorem~14.4(3)]{schoen1995computation}. For
$p\equiv2\pmod3$ and $\ell>3$, part (ii) is
\cite[Theorem~14.4(1)]{schoen1995computation}. By
\cite[Remark~14.9]{schoen1995computation}, the elliptic curve considered by Schoen may
be replaced by any isogenous elliptic curve, in particular by the Fermat
cubic. Thus the new cases are $p\equiv1\pmod3$ and
$\ell\equiv1\pmod3$, and, in the supersingular case, the primes
$\ell\leq3$ with $\ell\neq p$.

\begin{proof}[Proof of \Cref{griffiths}]
By \Cref{main-fermat-cubic-thm}, we have
\[
H^3_{\nr}(k(E^3)/k,\Q_\ell/\Z_\ell(2))=0.
\]
Let $X_0/\F_p$ be the triple self-product of the Fermat cubic over $\F_p$. Since $X_0$ is an abelian threefold, it belongs to $B(\F_p)$ and hence to $B_{\mathrm{Tate}}(\F_p)$ by
\cite[Property~3.14(vi)]{colliot2013cycles}. Combining \cite[Theorem~5.2(b), Proposition~5.4, Proposition~5.5]{kahn2012classes}, we obtain an isomorphism
\[
\Griff^2(E^3)_{\Z_\ell}\simeq \Griff^2(E^3,\Z_\ell)
\simeq
H^3_{\mathrm{tr}}(E^3,\Z_\ell(2))
\otimes_{\Z_\ell}\Q_\ell/\Z_\ell.
\]
The conclusion now follows from \Cref{lem:transcendental-E3}.
\end{proof}

\subsection{Unramified cohomology of the auxiliary threefold}

Finally, the dominant rational map $E^3\dashrightarrow W$ allows us to prove an analogue of \Cref{main-fermat-cubic-thm} for the threefold $W/\F_p$ of \S\ref{sec:W}.

\begin{thm}
Let $p>3$ be a prime, let $W$ be the smooth projective threefold over
$\F_p$ introduced in \S\ref{sec:W}, and let $\ell\neq p$ be a prime.
If $p\equiv 1\pmod 3$ (resp. $p\equiv 2\pmod 3$), assume moreover that $\ell>3$ (resp. $\ell\neq 3$).
Then, for every finite extension $k/\F_p$,
\[
H^3_{\nr}(k(W)/k,\Q_\ell/\Z_\ell(2))=0
\]
and the cycle map
\[
CH^2(W_k)_{\Z_\ell}\longrightarrow H^4(W_k,\Z_\ell(2))
\]
is surjective. In particular,
\[
H^3_{\nr}(\ov{\F}_p(W)/\ov{\F}_p,\Q_\ell/\Z_\ell(2))=0.
\]
\end{thm}

\begin{proof}
The dominant rational map of
\Cref{lem:dominant-rational-E3-W} gives a finite separable extension
$k(W)\subset k(E^3)$ of degree a power of $3$. Restriction on unramified
cohomology is injective: its composite with corestriction is multiplication
by this degree, which is an automorphism of $\Q_\ell/\Z_\ell$. By
\Cref{main-fermat-cubic-thm}, the unramified cohomology of $E^3_k$
vanishes, and hence so does that of $W_k$.

We next verify rational surjectivity of the cycle map for $W$.
For Schoen's threefold $W(3)$,
\cite[Example~1.5]{schoen1986cycles} gives $h^2(W(3),\mathcal O)=0$, and
\cite[Lemma~1.6]{schoen1986cycles} then gives
$b_2(W(3))=\operatorname{rank}\operatorname{NS}(W(3))$. Thus
$H^2(W(3)_{\C},\Q_\ell(1))$ is generated by divisor classes. 
By \Cref{lem:hesse-V3}, this is the same threefold as the resolution used here. Choose finitely many divisor generators and
spread them over a number field. This construction and its blow-up
form a smooth proper model over $\Z[1/6]$. By the proper base change in $\ell$-adic cohomology, we deduce that $H^2(W_{\bar k},\Q_\ell(1))$ is also generated by divisor classes.
Hard Lefschetz, using an ample divisor defined over the ground field, then
shows that $H^4(W_{\bar k},\Q_\ell(2))$ is generated by codimension-two
cycle classes. Taking norms from fields of definition of these cycles to $k$ shows that the rational cycle map surjects onto
$H^4(W_{\bar k},\Q_\ell(2))^{G_k}$. The left-hand term of
\eqref{eq:hochschild-serre-short-exact} is finite, so it disappears after
tensoring with $\Q_\ell$; hence
\[
CH^2(W_k)_{\Q_\ell}\longrightarrow H^4(W_k,\Q_\ell(2))
\]
is surjective.

Let $C$ be the cokernel of the integral cycle map in the statement.
The preceding rational surjectivity says that $C$ is torsion, that is, $C=C_{\tors}$. On the
other hand, \cite[Th\'eor\`eme~2.2]{colliot2013cycles} identifies
$C_{\tors}$ with the quotient of
$H^3_{\nr}(k(W)/k,\Q_\ell/\Z_\ell(2))$ by its maximal divisible subgroup.
The unramified group has just been shown to vanish, so $C=C_{\tors}=0$.
Thus $C=0$, proving surjectivity of the integral cycle map. The final assertion follows
from \Cref{lem:unramified-direct-limit}.
\end{proof}

\section{Potential vanishing versus absolute vanishing of unramified cohomology}\label{sec:8}

Over $\C$, let $X:=X_{n,2n-1}$ in \cite[Definition 2.13]{ottem2024fano}. The variety $X$ is a smooth Fano variety of dimension $2n-6$, and it has Picard rank $1$ and $H^3(X,\Z)=\Z/2$; see \cite[Theorem~4.1]{ottem2024fano}.

\begin{thm}\label{thm:IHC-for-Fano}
For $n\geq 6$, the cycle map $CH^2(X)\rightarrow H^4(X,\Z)$ is surjective.
\end{thm}

According to \cite[\S 3.3]{ottem2024fano}, $X$ approximates $BGO(4)^\circ$.
More precisely, by \cite[Proposition 3.5]{ottem2024fano}, for $n\geq 6$, we have $H^4(BGO(4)^\circ,\Z)\xrightarrow{\sim}H^4(X,\Z)$.

Thus \Cref{thm:IHC-for-Fano} follows from:

\begin{thm}\label{thm:IHC-for-BGO}
The cycle map $CH^2(BGO(4)^\circ)\rightarrow H^4(BGO(4)^\circ,\Z)$ is surjective.
\end{thm}

\Cref{thm:IHC-for-BGO} is perhaps surprising, because for the closely related group $SO(4)$, the corresponding statement fails, which is a source of examples with non-zero Griffiths group in \cite{totaro1997torsion}.

\begin{proof}[Proof of \Cref{thm:IHC-for-BGO}]
The exact sequence
\[
1\longrightarrow SO(4)\longrightarrow GO(4)^\circ
\mathop{\longrightarrow}^{\chi}\C^*\longrightarrow1
\]
gives a $\C^*$-fibration
$\pi\colon BSO(4)\to BGO(4)^\circ$. The relevant part of its (topological) Gysin
sequence is
\[
H^2(BGO(4)^\circ,\Z)\mathop{\longrightarrow}^{c_1(L)}
H^4(BGO(4)^\circ,\Z)\mathop{\longrightarrow}^{\pi^*}
H^4(BSO(4),\Z)\mathop{\longrightarrow}^{\pi_*}
H^3(BGO(4)^\circ,\Z),
\]
where $L$ is associated with $\chi$. The cycle map
$CH^1(BGO(4)^\circ)\to H^2(BGO(4)^\circ,\Z)$ is surjective by
\cite[Corollary~3.5]{totaro1999chow}. Moreover,
\[
H^4(BSO(4),\Z)=\Z e\oplus\Z p,
\]
where $e$ and $p$ are the Euler and Pontryagin classes, and
\cite[pp.~483--484]{totaro1997torsion} and
\cite{pandharipande1998equivariant} identify the image
of $CH^2(BSO(4))$ with $\Z(2e)\oplus\Z p$. The localization sequence gives an
exact sequence
\[
CH^1(BGO(4)^\circ)\mathop{\longrightarrow}^{c_1(L)}
CH^2(BGO(4)^\circ)\longrightarrow CH^2(BSO(4))\longrightarrow0;
\]
see for example \cite[\S2.2]{bhaumik2015chow}.

By \cite[Proposition~3.1]{ottem2024fano}, we have $H^3(BGO(4)^\circ,\Z)\simeq\Z/2$. We claim that $\pi_*e$ is the non-zero element of $H^3(BGO(4)^\circ,\Z)$. Modulo $2$, the Euler class is
$w_4$, so it is enough to prove
\[
\pi^*\pi_*w_4=w_3\ne0
\quad\text{in }H^3(BSO(4),\Z/2).
\]
For the extension
\[
1\longrightarrow O(4)\longrightarrow GO(4)
\mathop{\longrightarrow}^{\chi}\C^*\longrightarrow1,
\]
\cite[Proposition~4.21]{holla2001characteristic} gives the corresponding
identity
\[
\pi^*\pi_*w_4=w_3
\quad\text{in }H^3(BO(4),\Z/2).
\]
The square obtained from the inclusions
$SO(4)\subset O(4)$ and $GO(4)^\circ\subset GO(4)$ is Cartesian.
By naturality of the Gysin maps, we therefore obtain the claim.

Moreover, $p=-c_2(V)$, where $V$ is the tautological $4$-dimensional complex representation of $SO(4)$. Since $V$ is the restriction of a $GO(4)^\circ$-representation, we see that $p$ is the pullback of the corresponding Chern class on
$BGO(4)^\circ$, and exactness gives $\pi_*p=0$. It follows that
\[
\ker(\pi_*\colon H^4(BSO(4),\Z)\to
H^3(BGO(4)^\circ,\Z))=\Z(2e)\oplus\Z p,
\]
which is exactly the algebraic lattice. Given
$y\in H^4(BGO(4)^\circ,\Z)$, the class $\pi^*y$ lies in this kernel.
Lift an algebraic representative of $\pi^*y$ through the Chow localization
sequence. The difference between $y$ and the class of this lift belongs
to $\ker\pi^*$, hence by the Gysin sequence is $c_1(L)$ times a class in
$H^2(BGO(4)^\circ,\Z)$. That $H^2$-class is algebraic, hence so is $y$.
\end{proof}

\begin{proof}[Proof of \Cref{thm:potential-vs-absolute}]
Let $X_{\C}=X_{6,11}$ be the Fano sixfold of Ottem--Rennemo \cite{ottem2024fano}. Thus
$X_{6,11}$ is a general codimension-$11$ linear section of the double
symmetric determinantal variety
\[
W_{4,6}\longrightarrow Z_{4,6},
\]
where $Z_{4,6}\subset \P(\Sym^2(\C^6)^\vee)$ parametrizes quadrics of rank at most $4$. By
\cite[Theorem~4.1 and Corollary~3.6]{ottem2024fano}, the variety
$X_{\C}$ is smooth and Fano, $H^3(X_{\C},\Z)\simeq \Z/2$ and, if $u_{\C}\in H^3(X_{\C},\Z)$ is the non-zero class, then the mod-$2$ reduction $\overline{u}_{\C}$ satisfies $\overline{u}_{\C}^{2}\neq 0$ in $H^6(X_{\C},\Z/2)$.

By \Cref{thm:IHC-for-Fano}, choose finitely many
classes $z_j\in CH^2(X_{\C})$ whose integral cycle classes generate
$H^4(X_{\C},\Z)$. Let
$0\ne u_{\C}\in H^3(X_{\C},\Z_2)\simeq\Z/2$. Since $X_{\C}$ is rationally
connected, \cite[Theorem~1]{bloch1983remarks} gives an integer $N>0$, a
point $x\in X_{\C}$,
a divisor $D\subset X_{\C}$, and a relation
\begin{equation}\label{eq:diagonal-relation-complex}
    N[\Delta_{X_{\C}}]=N[X_{\C}\times x]+\Gamma
\quad\text{in }CH^6(X_{\C}\times X_{\C}),
\end{equation}
where $\Gamma$ is supported on $D\times X_{\C}$.

We spread out $X_\C$, the cycles $z_j$, and the above relation over
an integral scheme of finite type over $\Z[1/2]$, and shrink the base so
that the family $f\colon\mathcal X\to S$ is smooth and projective. After shrinking $S$ and a suitable finite \'etale base change of $S$, the class $u_{\C}$ extends, by smooth proper
base change, as a section of the $2$-torsion subsheaf of
$R^3f_*\Z_2(2)$. Choose a closed point of odd residue
characteristic and write $X/\kappa$ for the resulting fiber. Smooth
proper base change,
comparison, and specialization of the $z_j$ give a surjection
\[
CH^2(X_{\bar\kappa})_{\Z_2}
\relbar\joinrel\twoheadrightarrow H^4(X_{\bar\kappa},\Z_2(2)).
\]
Moreover, since here $n=6$ and $c=11=4n-13$,
\cite[Corollary~3.6]{ottem2024fano} says that the square of the mod-$2$
reduction of $u_{\C}$ is non-zero. Compatibility of specialization
with Steenrod operations, as in
\cite[Lemma~2.1]{scavia2022cohomology}, therefore gives a class
\[
u\in H^3(X_{\bar\kappa},\Z_2(2)),\qquad
\operatorname{Sq}^3(\overline u)=\overline u^{\,2}\ne0.
\]
After replacing $\kappa$ by a finite extension $k_0$, we may assume that the following conditions hold:
\begin{itemize}
\item[--] $\mu_4\subset k_0$;
\item[--] $G_{k_0}$ acts trivially on $H^6(X_{\bar k_0},\Z/2)$;
\item[--] the specializations of the $z_j$ are defined over $k_0$, so that we have a surjection
\[
CH^2(X_{k_0})_{\Z_2}\relbar\joinrel\twoheadrightarrow H^4(X_{\ov{k}_0},\Z_2(2));\]
\item[--] the specialization of \eqref{eq:diagonal-relation-complex} is defined over $k_0$: for the integer $N>0$ in \eqref{eq:diagonal-relation-complex}, there exist a point $x\in X(k_0)$, a divisor $D\subset X_{k_0}$, a $6$-cycle $\Gamma$ supported on $D\times X_{k_0}$, and a relation
\begin{equation}\label{eq:diagonal-relation}
N[\Delta_{X_{k_0}}]=N[X_{k_0}\times x] + \Gamma \in CH^6(X_{k_0}\times X_{k_0}).
\end{equation}
\end{itemize}

Let
\[
C\coloneqq \operatorname{Coker}(
CH^2(X_{\bar k_0})_{\Z_2}\to
H^4(X_{\bar k_0},\Z_2(2))).
\]
The preceding surjectivity gives $C=0$. The exact sequence of
\cite[Theorem~1.1]{kahn2012classes} then shows that
$H^3_{\nr}(\bar k_0(X)/\bar k_0,\Q_2/\Z_2(2))$ is divisible. On the other
hand, by base change of \eqref{eq:diagonal-relation} over $\ov{k}_0$,
we see that this group is annihilated by $N$. Since a divisible group of finite exponent is zero, we conclude
\[
H^3_{\nr}(\bar k_0(X)/\bar k_0,\Q_2/\Z_2(2))=0.
\]

Now let $k/k_0$ be any finite extension. We apply
\cite[Proposition~3.8]{scavia2022cohomology} over $k$, with
$\ell=i=2$, $P=\operatorname{Sq}^3$, and the geometric class $u$ above.
The field $k$ still contains $\mu_4$, the $G_k$-action on
$H^6(X_{\bar k},\Z/2)$ is trivial, and
$\operatorname{Sq}^3(\overline u)\ne0$. The proposition produces a
geometrically trivial, non-algebraic class $\alpha_k\in H^4(X_k,\Z_2(2))$. 
Because $u$ has order $2$, so does $\alpha_k$. Its image is therefore a
non-zero element of 
the torsion subgroup of
\[
C'\coloneqq \operatorname{Coker}(
CH^2(X_k)_{\Z_2}\to H^4(X_k,\Z_2(2))).
\]
This cokernel class has order $2$. In fact, this is the only non-zero class in $C'$ because $CH^2(X_k)_{\Z_2}\to H^4(X_{\ov{k}},\Z_\ell(2))$ is surjective, we have $H^1(k,H^3(X_{\ov{k}},\Z_2(2)))\simeq H^3(X_{\ov{k}},\Z_2(2))\simeq \langle u\rangle$, and by the Hochschild--Serre spectral sequence.
By \cite[Th\'eor\`eme~2.2]{colliot2013cycles}, we have
\[\frac{H^3_{\nr}(k(X)/k,\Q_2/\Z_2(2))}{H^3_{\nr}(k(X)/k,\Q_2/\Z_2(2))_{\on{div}}}
\xlongrightarrow{\ \sim\ }C'\simeq \Z/2\]
for every finite extension $k/k_0$. Moreover, by base change of \eqref{eq:diagonal-relation} over $k$, we have $H^3_{\nr}(k(X)/k,\Q_2/\Z_2(2))_{\on{div}}=0$ and hence $H^3_{\nr}(k(X)/k,\Q_2/\Z_2(2))\simeq\Z/2$ for every such $k/k_0$. This completes the proof for $d=6$. For $d>6$, it suffices to consider the product $X\times\P^{d-6}_{k_0}$.
\end{proof}

\begin{rem}\label{rem:schreieder}
Stefan Schreieder has pointed out that for a finite field $k_0$, an
integer $i\geq2$, and a smooth projective variety $X/k_0$, if
$H^i_{\nr}(k(X)/k,\Q_\ell/\Z_\ell(i-1))$ is non-zero for every finite
extension $k/k_0$, then $X_{\bar k_0}$ is not stably rational. Indeed, a
stable-rationality construction over $\bar k_0$ descends to some finite extension
$k/k_0$. Then $H^i_{\nr}(k(X)/k,\Q_\ell/\Z_\ell(i-1))\simeq H^i(k,\Q_\ell/\Z_\ell(i-1))$, which is zero for $i>1$ because
$\on{cd}_\ell(k)=1$.

In particular, for the Fano variety $X/k_0$ of
\Cref{thm:potential-vs-absolute}, $X_{\ov{k}_0}$ is not stably rational,
even though $H^3_{\nr}$ is absolutely vanishing. The same non-zero-square
class also shows on the special fiber that
\[
N^1H^3(X_{\ov{k}_0},\Z_2)\neq
\widetilde{N}^1H^3(X_{\ov{k}_0},\Z_2).
\]
Indeed, apply \cite[Proposition~3.4]{scavia2023coniveau} with
$m=3$, $c=1$, and $j=2$. In this case, $S_2=\operatorname{Sq}^3$, so
$\operatorname{Sq}^3(\bar u)=\bar u^2\neq0$ implies
$u\notin\widetilde N^1H^3$. On the other hand, a torsion $\ell$-adic
class belongs to $N^1H^3$; see the proof of
\cite[Theorem~1.1(2)]{scavia2023coniveau} and
\cite[Lemma~3.12]{kahn2012classes}. This already obstructs stable
rationality of $X_{\ov{k}_0}$. It would be interesting to find examples
with $H^3_{\nr}$ that is not potentially vanishing, for which no other
known obstructions to stable rationality apply.
\end{rem}

\section*{Acknowledgments}

The results of this paper have been obtained in the Spring of 2025 and announced in January 2026 during the Workshop ``Arithmetic of Algebraic Cycles'' at the Lodha Mathematical Sciences Institute in Mumbai. Part of this work was written during our stay at the LMSI; we thank the Institute for excellent working conditions.

We thank Stefan Schreieder for \Cref{rem:schreieder}. We thank Yang Cao for asking us about \Cref{cor:local-global-fermat} and for bringing \cite{gazaki2024weak} to our attention. We thank John Christian Ottem for helpful comments about \Cref{thm:potential-vs-absolute}.

The first author thanks Jean-Louis Colliot-Thélène for introducing him to this research area and for asking him about the unramified cohomology of triple products of elliptic curves over finite fields in the Fall of 2019.

The second author was funded by National Key R\&D Program of China (grant no. 2025YFA1017300 and 2025YFA1017301) and by the ERC grant ``RationAlgic'' (grant no. 948066).

\end{document}